\documentclass[a4paper,11pt,draft]{amsart}

\title
[Free homotopy decompositions of critical Sobolev maps]
{Estimates by gap potentials of free homotopy decompositions of critical Sobolev maps}
\author{Jean Van Schaftingen}
\address{Universit\'e catholique de Louvain (UCLouvain)\\ 
Institut de Recherche en Math\'ematique et Physique (IRMP)\\
Chemin du Cyclotron 2 bte L7.01.01\\
1348 Louvain-la-Neuve\\
Belgium}
\email{Jean.VanSchaftingen@uclouvain.be}
\thanks{J. Van Schaftingen was supported by the Mandat d'Impulsion Scientifique F.4523.17, ``Topological singularities of Sobolev maps'' of the Fonds de la Recherche Scientifique--FNRS}

\subjclass[2010]{
46E35
(%
55M25
,
55P99%
,
55Q25%
,
58A12%
)
}
\keywords{%
Free homotopy%
; 
Hurewicz homomorphism%
; 
vanishing mean oscillation (VMO)%
;
fractional Sobolev spaces%
}

\usepackage[utf8]{inputenc}
\usepackage[T1]{fontenc}
\usepackage{geometry}
\usepackage{microtype}
\usepackage{lmodern}
\allowdisplaybreaks

\usepackage{amssymb}
\usepackage{amsmath}
\usepackage{amsthm}

\usepackage{esint}

\usepackage[foot]{amsaddr}
\usepackage{constants}
\usepackage{paralist}
\usepackage{mathrsfs}

\newcommand{\eofs}{\,}

\usepackage[pdftitle={Estimates by gap potentials of free homotopy decompositions of critical Sobolev maps},%
pdfauthor={Jean Van Schaftingen},%
final]{hyperref}

\usepackage[nameinlink%
]{cleveref}

\usepackage[abbrev,backrefs]{amsrefs}

\newtheorem{proposition}{Proposition}[section]
\newtheorem{theorem}[proposition]{Theorem}
\newtheorem{lemma}[proposition]{Lemma}

\theoremstyle{definition}
\newtheorem{openproblem}{Open problem}

\theoremstyle{definition}
\newtheorem{definition}[proposition]{Definition}

\theoremstyle{remark}
\newtheorem{remark}[proposition]{Remark}

\numberwithin{equation}{section}

\newcommand{\abs}[1]{{\lvert #1 \rvert}}
\newcommand{\bigabs}[1]{{\bigl\lvert #1 \bigr\rvert}}

\newcommand{\biggabs}[1]{{\biggl\lvert #1 \biggr\rvert}}
\newcommand{\norm}[1]{{\lVert #1 \rVert}}

\newcommand{\dualprod}[2]{\langle #1, #2 \rangle}

\newcommand{\st}{\;\vert\;}

\newcommand{\Rset}{\mathbb{R}}
\newcommand{\Nset}{\mathbb{N}}
\newcommand{\Sset}{\mathbb{S}}
\newcommand{\Bset}{\mathbb{B}}
\newcommand{\Hset}{\mathbb{H}}
\newcommand{\Zset}{\mathbb{Z}}

\newcommand{\BMO}{\mathrm{BMO}}
\newcommand{\Hopf}{\deg_H}

\newcommand{\diff}{\,\mathrm{d}}
\newcommand{\compose}{\,\circ\,}
\newcommand{\manifold}[1]{\mathcal{#1}}
\newcommand{\mapsset}[1]{\mathscr{#1}}

\newcommand{\familyname}[1]{\textsc{#1}}
\newcommand{\defeq}{\triangleq}

\DeclareMathOperator{\dist}{dist}
\DeclareMathOperator*{\osc}{osc}
\DeclareMathOperator{\diam}{diam}

\DeclareMathOperator{\sech}{sech}

\begin{document}

\begin{abstract}
A free homotopy decomposition of any continuous map from a compact Riemmanian manifold $\mathcal{M}$ to a compact Riemannian manifold $\mathcal{N}$ 
into a finite number maps belonging to a finite set is constructed, in such a way that the number of maps in this free homotopy decomposition and the number of elements of the set to which they belong can be estimated a priori by the critical Sobolev energy of the map in $W^{s,p} (\mathcal{M}, \mathcal{N})$, with $sp = m = \dim \mathcal{M}$.
In particular, when the fundamental group $\pi_1 (\mathcal{N})$ acts trivially 
on the homotopy group $\pi_m (\mathcal{N})$, the number of homotopy classes to which a map can belong can be estimated by its Sobolev energy.
The estimates are particular cases of estimates under a boundedness assumption on gap potentials of the form 
\begin{equation*}
  \iint\limits_{\substack{(x, y) \in \mathcal{M} \times \mathcal{M} \\ d_\mathcal{N} (f (x), f (y)) \ge \varepsilon}}
    \frac{1}{d_\mathcal{M} (y, x)^{2 m}} \, \mathrm{d} y \, \mathrm{d}x.
\end{equation*}
When $m \ge 2$, the estimates scale optimally as $\varepsilon \to 0$. 
Linear estimates on the Hurewicz homorphism and the induced cohomology homomorphism are also obtained. 
\end{abstract}

\maketitle

\tableofcontents

\section{Introduction}

\subsection{Estimates on the topological degree}
Brouwer's \emph{topological degree} classifies the homotopy classes of continuous maps from a sphere \(\Sset^m\) to itself. 
If the map \(f \in \mapsset{C}^1 (\Sset^m, \Sset^m)\) is smooth, 
then its degree \(\deg (f)\) can be computed by the classical \emph{Kronecker formula}
\begin{equation}
\label{eq_formula_Kronecker}
 \deg (f) = \int_{\Sset^m} f_\sharp \omega_{\Sset^m}\eofs ,
\end{equation}
where \(\omega_{\Sset^m}\) is the \emph{volume form} on the sphere \(\Sset^m\) normalized so that \(\int_{\Sset^m} \omega_{\Sset^m} = 1\) and \(f_\sharp \omega_{\Sset^m}\) is the \emph{pullback} of the form \(\omega_{\Sset^m}\) by the map \(f\).
In view of the classical inequality between the geometric and quadratic means, we have
\begin{equation*}
    \abs{f_\sharp \omega_{\Sset^m}} 
  = 
    \abs{\det Df}
    \, 
    \omega_{\Sset^n} 
  \le 
    \frac{\abs{Df}^m}{m^{m/2}} \,\omega_{\Sset^m}
\end{equation*}
everywhere on the sphere \(\Sset^m\); 
this implies then the integral estimate on the degree (see \cite{BourgainBrezisMironescu2005CPAM}*{Remark 0.7}) 
\begin{equation}
\label{ineq_degree_W_1_n}
  \abs{\deg (f)} 
 \le 
  \frac{1}{m^{m/2}} 
  \fint_{\Sset^m} 
    \abs{Df}^m
 =
  \frac{1}{m^{m/2} \abs{\Sset^m}}
  \mathcal{E}^{1, m} (f)
    \eofs ,
\end{equation}
where we have defined the \emph{Sobolev energy} \(\mathcal{E}^{1, p}\) for \(p \in [1, + \infty)\) by
\[
  \mathcal{E}^{1, p} (f)
  \defeq
 \int_{\Sset^m} 
    \abs{Df}^p
    \eofs .
\]
This estimate \eqref{ineq_degree_W_1_n} remains valid under the weaker assumption that the map \(f\) lies in the \emph{Sobolev space} \(W^{1, m} (\Sset^m, \Sset^m)\) of weakly differentiable maps whose weak derivative satisfies the integrability condition that \(\int_{\Sset^m} \abs{Df}^n < +\infty\); the degree of \(f\) is then understood in the sense of maps of \emph{vanishing mean oscillation} (VMO) \cite{BrezisNirenberg1995}.
By the classical H\"older inequality, the estimate \eqref{ineq_degree_W_1_n} implies that the degree of \(f\) can also be controlled by the \(L^p\) norms of its derivative \(Df\) for \(p \in (m, + \infty]\).

Although the integral formula \eqref{eq_formula_Kronecker} does not have a clear sense when the map \(f\) does not have some kind of derivatives,
the natural counterpart of the integral estimate \eqref{ineq_degree_W_1_n} still holds for \emph{fractional Sobolev maps}: 
for every \(p \in (m, +\infty) \), there exists a constant \(C_{m, p}\) such that 
for every map \(f \in W^{m/p, p} (\Sset^m, \Sset^m)\), one has \cite{BourgainBrezisMironescu2005CPAM}*{theorem 0.6} (see also \citelist{\cite{Boutet_Georgescu_Purice}*{theorem A.3}\cite{Mironescu}*{theorem 2.3}})
\begin{equation}
\label{ineq_degree_W_s_p}
    \abs{\deg (f)} 
  \le 
    C_{m, p}
    \,
    \mathcal{E}^{m/p, p} (f)
  \eofs ,
\end{equation}
where the \emph{fractional Sobolev energy} of the map \(f : \Sset^m \to \Sset^m\) is defined for any \(s \in (0, 1)\) and \(p \in [1, +\infty)\) as 
\begin{equation}
\label{def_SobolevEnergy}
  \mathcal{E}^{s, p} (f)
  \defeq 
  \iint \limits_{\Sset^m \times \Sset^m}
    \frac{\abs{f (y) - f (x)}^p}{\abs{y - x}^{m + sp}}
    \diff x \diff y
\end{equation}
and the finiteness of the quantity \(\mathcal{E}^{s, p} (f) \) is equivalent to have the map \(f\) belonging to the \emph{fractional Sobolev space} \(W^{s, p} (\Sset^m, \Sset^m)\) with \(s = m/p\).

Although the degree is naturally defined for continuous maps or for maps of \emph{vanishing mean oscillation} (VMO), the associated norms in \(L^\infty\) or in \(\BMO\) remain bounded over the entire class of maps on the sphere and cannot control in any way the topology, 

Jean \familyname{Bourgain}, Haïm \familyname{Brezis}, Petru \familyname{Mironescu} and \familyname{Nguyên} Hoài-Minh have obtained gap potential estimates of the degree 
that imply integral estimates of the form \eqref{ineq_degree_W_1_n} and \eqref{ineq_degree_W_s_p} \citelist{\cite{BourgainBrezisNguyen2005CRAS}*{theorem 1.1}\cite{BourgainBrezisMironescu2005CPAM}*{open problem 2}\cite{Nguyen_2007}} (see also \cite{Nguyen_2014}): they have proved that for every \(\varepsilon \in (0, \sqrt{2(1 + \frac{1}{m +1})})\), there exists a constant \(C_{\varepsilon, m}\) such that for every map \(f \in \mapsset{C} (\Sset^m, \Sset^m)\), one has
\begin{equation}
\label{ineq_degree_estimate_nguyen}
  \abs{\deg (f)}
 \le \, 
  C_{\varepsilon, m}\hspace{-1.7em}
  \iint
    \limits_{
      \substack{
        (x, y) \in \Sset^m \times \Sset^m \\ 
        \abs{f (y) - f (x)} > \varepsilon}} 
    \hspace{-.3em}
    \frac{1}{\abs{y - x}^{2 m}} 
    \diff x 
    \diff y \eofs .
\end{equation}
In view of the definition \eqref{def_SobolevEnergy}, the gap potential estimate \eqref{ineq_degree_estimate_nguyen} implies the fractional Sobolev estimate \eqref{ineq_degree_W_s_p}.
If \(m \ge 2\), the constant can be taken to satisfy \(C_{\varepsilon, m} \le C_m \varepsilon^m\) \cite{Nguyen_2017}.

The \emph{gap potentials} on the right-hand side of \eqref{ineq_degree_estimate_nguyen} also appeared in estimates on lifting of maps into the circle \cite{Nguyen_2008_CRAS}*{theorem 2} and were showed to characterize as \(\varepsilon \to 0\) Sobolev spaces \citelist{\cite{Nguyen_2006}\cite{Nguyen_2007_CRAS}\cite{Nguyen_2008}\cite{Nguyen_Pinamonti_Squassina_Vecchi_2018}} and provide a property stronger than VMO \cite{Brezis_Nguyen_2011}.

\subsection{Estimates on the Hopf invariant}

Continuous maps from \(\Sset^3\) into \(\Sset^2\) are classified by the \emph{Hopf invariant} \(\Hopf (f)\). Tristan \familyname{Rivière} has proved that 
\cite{Riviere_1998}*{lemma III.1} (see also \citelist{\cite{Gromov_1999}\cite{Gromov_1999_book}*{Lemma 7.12}})
\begin{equation}
\label{ineq_Hopf_Riviere}
    \abs{\Hopf{f}} 
  \le 
    C \,\biggl(\int_{\Sset^3} \abs{Df}^3\biggr)^{1 + \frac{1}{3}}.
\end{equation}
Compared to the corresponding estimate of the topological degree \eqref{ineq_degree_W_1_n}, a power \(1 + \frac{1}{3}\) applied to the integral appears, related to the Whitehead formula for the Hopf invariant \cite{Whitehead_1947}. 
No fractional counterpart to \eqref{ineq_Hopf_Riviere} seems to be know (see \cref{problem_hopf_Sobolev} below).

Rivi\`ere's bound \eqref{ineq_Hopf_Riviere} extends straightforwardly to a its higher-dimensional counterpart which is a homotopy invariant for maps from the sphere \(\Sset^{2 n - 1}\) into \(\Sset^n\) \cite{Bott_Tu_1982}*{proposition 17.22}, resulting in an estimate
\begin{equation}
\label{ineq_Riviere_n}
 \abs{\Hopf{f}} 
 \le C \,\biggl(\int_{\Sset^{2n - 1}} \hspace{-1em} \abs{Df}^{2n - 1}\biggr)^{1 + \frac{1}{2n - 1}}.
\end{equation}
The Hopf invariant takes nontrivial values when \(n\) is even \cite{Bott_Tu_1982}*{proposition 17.22} but is not necessarily injective (when \(n \in \{1, \dotsc, 20\}\), it is injective if and only if \(n \in \{2, 6\}\) \cite{Toda_1962}).
In all cases, only finitely many homotopy classes share the same value of the Hopf invariant and thus any set of maps which is bounded in the Sobolev space \(W^{1, 2m - 1} (\Sset^m, \Sset^{2m - 1})\) is contained in finitely homotopy classes of maps. 

By a theorem of Jean-Pierre \familyname{Serre} \cite{Serre_1951}, all the other classes of continuous maps between spheres of different dimension consist only of finitely many homotopy classes; 
thus in general a bounded set of maps in \(W^{1, 3} (\Sset^m, \Sset^n)\) is contained in finitely many homotopy classes.

\subsection{Estimates on free homotopy decompositions}
The results outlined above for maps between spheres raise the question whether sets which are bounded Sobolev norms are contained in finitely many homotopy classes of maps.

When \(s \in (0, 1)\), \(p \in (1, + \infty)\) and \(s p > m\), the classical Morrey--Sobolev embedding (see for example \citelist{\cite{Brezis_2011}*{theorem 9.12}\cite{Willem_2013}*{lemma 6.4.3}}) ensures that sets which are bounded in energy in \(W^{s, p} (\Sset^m, \manifold{N})\) are also bounded in the space of H\"older-continuous functions \(\mapsset{C}^{0, s - m/p} (\Sset^m, \manifold{N})\), and thus by the Ascoli compactness criterion and the local invariance of homotopy classes, they are contained in finitely homotopy classes.

A slighly more subtle case is \(W^{1, 1} (\Sset^1, \manifold{N})\): although there is no compact embedding in the set of continuous maps, each map is homotopic to a map whose Lipschitz constant is controlled; hence bounded sets are contained in finitely many homotopy classe.

In the general case of \(W^{s, p} (\Sset^m, \manifold{N})\) with \(sp=m\) and \(p > 1\) with an arbitrary target manifold \(\manifold{N}\), such a control turns out to be impossible. 
In order to construct infinitely many non-homotopic maps whose Sobolev energies remain bounded, we rely on the following definition:

\begin{definition}%
[Free homotopy decomposition]%
\label{definition_homotopy_power}%
A map \(f \in \mapsset{C} (\Sset^m, \manifold{N})\) has a \emph{free homotopy decomposition into} the maps \(f_1, \dotsc, f_k \in \mapsset{C} (\Sset^m, \Sset^m)\) whenever there exists a map \(g \in \mapsset{C} (\Sset^m, \manifold{N})\) homotopic to \(f\) on \(\Sset^m\) and nontrivial geodesic balls \(B_{\rho_1} (a_1), \dotsc, B_{\rho_\ell} (a_k) \subset \Sset^m\) 
such that \(g\) is constant on \(\Sset^m \setminus \textstyle \bigcup_{i = 1}^k B_{\rho_i} (a_i)\) and for every \(i \in \{1, \dotsc, k\}\), its restriction \(g\vert_{\Bar{B}_{\rho_i}(a_i)}\) is homotopic to some \(f_i \in \mapsset{F}\) on \(\Sset_m \simeq \Bar{B}_{\rho_i}(a_i)/\partial B_{\rho_i} (a_i)\).
\end{definition}

The map \(g\) is well defined on the quotient \(\Bar{B}_{\rho_i}(a_i)/\partial B_{\rho_i} (a_i) \simeq \Sset^m\) because it is constant on \(\partial B_{\rho_i} (a_i)\).

The free homotopy decomposition appears in the construction of harmonic and polyharmonic maps that are known in many instances to generate through free homotopy decomposition all the homotopy classes \citelist{
\cite{SacksUhlenbeck1981}*{theorem 5.5}
\cite{GastelNerf2013}*{theorem 14}
}.

The free homotopy decomposition is an invariant under homotopies of the maps, but is not in general a faithful invariant: for example if \(\manifold{N} = (\Sset^1 \times \Sset^{2m} \cup \Sset^m \times \Sset^{m + 1})/\Sset^{2 m}\), then there are two maps into which \emph{infinitely many} homotopy classes decompose freely (see \cref{proposition_unbounded_glue} below).

The next result shows that maps that have the same free homotopy decomposition satisfy up to homotopy the same fractional Sobolev bound.

\begin{theorem}
[Bound on the Sobolev energy by free homotopy decomposition]
\label{theorem_bubbles_are_bounded}
Let \(m \in \Nset_*\) and \(\manifold{N}\) be a connected Riemannian manifold.
If \(f \in \mapsset{C} (\Sset^m, \manifold{N})\) has a free homotopy decomposition into \(f_1, \dotsc, f_k \in \mapsset{C} (\Sset^m, \manifold{N})\), then for every \(s \in (0, 1]\) and \(p \in [m, +\infty)\) such that \(p = m/s > 1\), 
\begin{equation*}
    \inf 
      \,
      \bigl\{
        \mathcal{E}^{s, p} (g)
      \st 
        g \in (\mapsset{C} \cap W^{s, p}) (\Sset^m, \manifold{N}) \text{ is homotopic to } f 
      \bigr\}
  \le 
    \sum_{i = 1}^k
      \mathcal{E}^{s, p} (f_i)
      \eofs .
\end{equation*}
\end{theorem}

In particular, \cref{theorem_bubbles_are_bounded} implies that all the homotopy classes that decompose freely into the maps \(f_1, \dotsc, f_k\) satisfy the same energy bound; if there are infinitely many such homotopy classes then there are infinitely many nonhomotopic map satisfying the same energy bound.

The proof of \cref{theorem_bubbles_are_bounded} is performed by gluing together the maps  \(f_1, \dotsc, f_k\) with an arbitrarily small energetic cost of gluing, performed through conformal transformations by Mercator projections.
\Cref{theorem_bubbles_are_bounded} does not cover the case \(s = p = m = 1\). This is consistent with our observation that a Sobolev energy bound gives a control on the homotopy classes.

\medskip

By taking the phenomenon described in \cref{theorem_bubbles_are_bounded} into account, it has been proved that for every \(\lambda > 0\), there exists a finite set \(\mapsset{F}\) and \(k \in \Nset\) such that every map \(f \in (W^{s, p} \cap \mapsset{C}) (\Sset^m, \manifold{N})\) satisfying 
\( \mathcal{E}^{s, m/s} (g) 
      \le 
        \lambda 
\) 
has a free homotopy decomposition into \(k\) maps of the set \(\mapsset{F}\)
for \(m = 1\), \(s = \frac{1}{2}\) and \(p = 2\) by Ernst \familyname{Kuwert} \cite{Kuwert1998},
when \(m \ge 1\), \(s = 1\) by Frank \familyname{Duzaar} and Ernst \familyname{Kuwert} \cite{DuzaarKuwert1998}*{theorem 4}, when \(m \ge 1\) and \(s = 1-\frac{1}{m + 1}\) by Thomas \familyname{Müller} \cite{Muller2000}*{theorem 5.1} and when \(m = 2\) and \(s = 1\) by Richard \familyname{Schoen} and Jon \familyname{Wolfson} \cite{SchoenWolfson2001}*{lemma 5.2}.

The critical case \(sp = m\) for estimates can be seen as a limiting case between the classical continuous picture of homotopy classes in the supercritical \(sp > m\) and the combination of collapses and appearance of homotopy classes in the subcritical case \(sp < m\) \citelist{\cite{BrezisLi2001JFA}\cite{BrezisLi2000CRAS}\cite{White1986JDG}\cite{Hang_Lin_2001_MRL}\cite{HangLin2003Acta}\cite{Hang_Lin_2003_CPAM}\cite{Hang_Lin_2005}}.

Our main result shows that these estimates are in fact consequences
of a stronger gap potential estimate similar to \eqref{ineq_degree_estimate_nguyen}.

\begin{theorem}
[Free homotopy decompositions controlled by a gap potential]
\label{theorem_bounded_finite_bubbles}
Let \(m \in \Nset_*\) and \(\manifold{N}\) be a compact Riemannian manifold.
If \(\varepsilon > 0\) is small enough, then there is a constant \(C > 0\) such that for every \(\lambda > 0\), there exists a finite set \(\mapsset{F}^\lambda \subset \mapsset{C} (\Sset^m, \manifold{N})\) such that any map \(f \in \mapsset{C} (\Sset^m, \manifold{N})\) satisfying
\begin{equation*}
          \iint
            \limits_{
              \substack{
                (x, y) \in \Sset^m \times \Sset^m \\                 
                d_{\manifold{N}} (f (y), f (x)) > \varepsilon}} 
              \frac{1}{\abs{y - x}^{2 m}} 
            \diff y
            \diff x
        \le 
          \lambda\eofs ,
\end{equation*}
has a free homotopy decomposition into \(f_1, \dotsc, f_k \in \mapsset{F}^\lambda\) with \(k \le C \lambda\).
\end{theorem}

In fact it can be observed that under the assumptions of \cref{theorem_bounded_finite_bubbles} any \emph{measurable map} that satisfies the integrability condition with \(\varepsilon\) small enough has a small mean oscillation on small scales \citelist{\cite{Brezis_Nguyen_2011}*{proposition 1}\cite{Nguyen_2011_CVAR}} and therefore can be associated naturally and uniquely to a homotopy class of continuous maps from \(\Sset^m\) to \(\manifold{N}\) (see \cite{BrezisNirenberg1995}*{(8), remark 7 and lemma A.5}).

The appearance of free homotopy decompositions in which the way of gluing the \(k\) maps together is arbitrary and uncontrolled can be thought of as a \emph{topological bubbling phenomenon}, which is a topological version of the geometric bubbling phenomenon in conformally invariant geometric problems \citelist{\cite{Parker_2003}\cite{Druet_Hebey_Robert}\cite{SacksUhlenbeck1981}}.
In many cases however, \cref{theorem_bounded_finite_bubbles} implies that maps satisfying a bound on the gap potential can only belong to finitely many homotopy classes.

\begin{theorem}
[Finitely many homotopy classes under a gap potential bound]
\label{theorem_bounded_homotopy_classes}
Let \(m \in \Nset_*\) and \(\manifold{N}\) be a compact Riemannian manifold.
If \(m = 1\) and every conjugacy class of \(\pi_1 (\manifold{N})\) is finite or \(m \ge 2\) and every orbit of the action of \(\pi_1 (\manifold{N})\) on \(\pi_m (\manifold{N})\) is finite, and if \(\varepsilon > 0\) is small enough, then for every \(\lambda > 0\), there exists a finite set \(\mapsset{G}^\lambda \subset \mapsset{C} (\Sset^m, \manifold{N})\) such that any map \(f \in \mapsset{C} (\Sset^m, \manifold{N})\) satisfying
\begin{equation*}
          \iint
            \limits_{
              \substack{
                (x, y) \in \Sset^m \times \Sset^m \\                 
                d_{\manifold{N}} (f (y), f (x)) > \varepsilon}} 
              \frac{1}{\abs{y - x}^{2 m}} 
            \diff x 
            \diff y 
        \le 
          \lambda\eofs ,
\end{equation*}
is homotopic to some \(g \in \mapsset{G}^\lambda\).
\end{theorem}

The assumptions of \cref{theorem_bounded_homotopy_classes} are satisfied in particular when \(\pi_1 (\manifold{N})\) is finite, if \(m = 1\) and \(\pi_1 (\manifold{N})\) is abelian or if \(m \ge 2\) and the action of \(\pi_1 (\manifold{N})\) on \(\pi_m (\manifold{N})\) is trivial.

In particular, under the assumptions of \cref{theorem_bounded_homotopy_classes}, the homotopy group \(\pi_m (\manifold{N})\) endowed with the norm naturally induced by a Sobolev energy  
satisfier a sufficient condition for compactness of the currents with coefficients on an abelian group
\cite{Fleming1966TAMS}*{assumption (H), lemma 7.4 and corollary 7.5} (when \(m = 1\),  this only makes sense when the group \(\pi_1 (\manifold{N})\) is abelian).

When \(m \ge 2\), in analogy with the optimal scaling \(\varepsilon^m\) when \(\varepsilon \to 0\) of estimates \cite{Nguyen_2017}, 
we obtain a similar optimal scaling in \(\varepsilon\) (see \cref{theorem_bounded_finite_bubbles_scaling} below), with a different strategy of proof than \cite{Nguyen_2017}.

\medskip

The proof of \cref{theorem_bounded_finite_bubbles} is performed in a geometric setting where 
the sphere \(\Sset^m\) is considered as the boundary at infinity of the \emph{hyperbolic space} \(\Hset^{m + 1}\) and the manifold \(\manifold{N}\) is \emph{embedded isometrically} into a Euclidean space \(\Rset^\nu\).
The extension of the map \(f\) by averaging at each point \(x \in \Hset^{m + 1}\) over the sphere at infinity --- which is also in fact the \emph{hyperharmonic extension} --- provides a Lipschitz-continuous extension \(F : \Hset^{m + 1} \to \Rset^{\nu}\).
The set on which the values of the map \(F\) cannot be retracted to \(\manifold{N}\) is contained in a number of balls whose diameter and number is controlled allowing to construct the families of maps by a classical Ascoli compactness argument for continuous maps.

\medskip

In view of \cref{theorem_bubbles_are_bounded}, \cref{theorem_bounded_finite_bubbles} describes sharply the homotopy classes that can be encountered under a boundedness assumption on the double integral. However, our proof exhibits a set of maps \(\mapsset{F}^\lambda\) by a compactness argument and gives thus double exponential bound of the form \(\exp (C \sinh (C' \lambda))\) on the cardinal of \(\mapsset{F}^\lambda\). This brings the question whether a better explicit control like the linear estimate \eqref{ineq_degree_estimate_nguyen}.

When the homotopy classes can be controlled by the homology, that is, when the Hurewicz homomorphism from \(\pi_m (\manifold{N})\) to the rational homology group \(H_m (\manifold{N})\) has a finite kernel, we recover a linear control on the number of homotopy classes that satisfy a given bound  (see \cref{proposition_Hurewicz_nonlocal} below).

When the domain \(\Sset^m\) is replaced by a general \(m\)--dimensional manifold \(M\), \cref{theorem_bounded_finite_bubbles} has a natural generalization, in which the corresponding homotopy classes are generated by a finite set of homotopy classes of \(\mapsset{C} (\manifold{M}, \manifold{N})\) glued together with a finite number of maps taken in finitely many homotopy classes of \(\mapsset{C} (\Sset^m, \manifold{N})\) (see \cref{section_domain_manifold} below). 
As before, there can be in general infinitely many homotopy classes generated in this way by finitely many homotopy classes. 
The strategy of the proof is similar.

\medskip

As perspectives of the present work, 
several \emph{open problems} are presented in the last section of the present work (see \cref{section_problems}).

\section{Free homotopy decomposition}

\subsection{Free homotopy decomposition and homotopy groups}
The notion of free homotopy decomposition of \cref{definition_homotopy_power} plays an important role in the present work. 
We describe here free homotopy decomposition in terms of homotopy groups.

We define \(f \in \mapsset{C} (\Sset^m, \manifold{N})\) and \(\gamma \in \pi_m (\manifold{N})\) to be \emph{homotopic} whenever any representative of the relative homotopy class \(\gamma\) is homotopic to the map \(f\). 
Since we have not fixed a base point in the homotopy between the representative in \(\gamma \in \pi_m (\manifold{N})\) and the map \(f\), a given map \(f \in \mapsset{C} (\Sset^m, \manifold{N})\) can be homotopic to \emph{several distinct elements} of \(\pi_m (\manifold{N})\).

\medbreak

When \(m = 1\), the elements of the fundamental group \(\pi_1 (\manifold{N})\) homotopic to a \emph{free homotopy class} of maps from the circle \(\Sset^1\) to \(\manifold{N}\) form a \emph{conjugacy class} of the fundamental group \(\pi_1 (\manifold{N})\) (see for example \cite{Hatcher_2002}*{exercise 1.1.6 and proposition 4A.2}). 

\begin{proposition}%
[Free decompositions and the fundamental group]
\label{free_decomposition_pi_1}
Assume that the maps \(f, f_1, \dotsc, f_k \in \mapsset{C} (\Sset^1, \manifold{N})\) are respectively homotopic to \(\gamma, \gamma_1, \dotsc, \gamma_k \in \pi_1 (\manifold{N})\). Then 
\(f\) has a free homotopy decomposition into \(f_1, \dotsc, f_k\) if and only if there exist \(\beta_1, \dotsc, \beta_k \in \pi_1 (\manifold{N})\) such that \(\gamma = \beta_1 \gamma_1 \beta_1^{-1} \dotsm \beta_k \gamma_k \beta_k^{-1}\). 
\end{proposition}

In particular, when the fundamental group \(\pi_1 (\manifold{N})\) is abelian, the homotopy classes of \(\mapsset{C} (\Sset^1, \manifold{N})\) correspond to elements in \(\pi_1 (\manifold{N})\) and the map  \(f\) has a free homotopy decomposition into \(f_1, \dotsc, f_k\) if and only if \(\gamma =  \gamma_1  \dotsm \gamma_k\).

\medbreak

When \(m \ge 2\), the elements of \(\pi_m (\manifold{N})\) corresponding to a \emph{free homotopy class} of maps from the circle \(\manifold{N}\) correspond to orbits of the action of the fundamental group \(\pi_1 (\manifold{N})\) on the homotopy group \(\pi_{m} (\manifold{N})\) (see for example \cite{Hatcher_2002}*{proposition 4A.2}).

\begin{proposition}%
[Free decompositions and the homotopy groups]%
\label{free_decomposition_pi_m}
Let \(m \ge 2\) and assume that the maps \(f, f_1, \dotsc, f_k \in \mapsset{C} (\Sset^m, \manifold{N})\) are respectively homotopic to \(\gamma, \gamma_1, \dotsc, \gamma_k \in \pi_m (\manifold{N})\). Then 
\(f\) has a free homotopy decomposition into \(f_1, \dotsc, f_k\) if and only if there exists \(\beta_1, \dotsc, \beta_k \in \pi_1 (\manifold{N})\) such that \(\gamma = \beta_1 \cdot \gamma_1 + \dotsb + \beta_k \cdot \gamma_k\). 
\end{proposition}

Here \(\beta_i \cdot \gamma_i\) denotes the action of \(\beta_i \in \pi_1 (\manifold{N})\) on \(\gamma_i \in \pi_m (\manifold{N})\) (see \cite{Hatcher_2002}*{\S 4.1}). If the action of \(\pi_1 (\manifold{N})\) on \(\pi_m (\manifold{N})\) happens to be trivial, then the map \(f\) has a free homotopy decomposition into \(f_1, \dotsc, f_k\) if and only if \(\gamma =  \gamma_1 + \dotsb + \gamma_k\).

\begin{proposition}
\label{proposition_finite_homotopy_classes}
Assume that \(m = 1\) and every conjugacy class of \(\pi_1 (\manifold{N})\) is abelian, or that \(m \ge 2\) and every orbit of the action of \(\pi_1 (\manifold{N})\) on \(\pi_m (\manifold{N})\) is finite.
If \(k \in \Nset\) and \(f_1, \dotsc, f_k \in \mapsset{C} (\Sset^m, \manifold{N})\),
then there exists a finite set \(\mapsset{G} \subset \mapsset{C} (\Sset^m, \manifold{N})\) such that 
if \(f \in \mapsset{C} (\Sset^m, \manifold{N})\) has a free homotopy decomposition into \(f_1, \dotsc, f_k\), then \(f\) is homotopic to some \(g \in \mapsset{G}\).
\end{proposition}

\begin{proof}
When \(m = 1\), let \(\gamma_1, \dotsc, \gamma_k \in \pi_1 (\manifold{N})\) be elements of \(\pi_1 (\manifold{N})\) respectively homotopic to \(f_1, \dotsc, f_k\). 
We consider the set 
\[
 \Gamma = 
 \bigl\{\beta_1 \gamma_1 \beta_1^{-1} \dotsm \beta_k \gamma_k \beta_k^{-1} \st  \beta_1, \dotsc, \beta_k \in \pi_1 (\manifold{N}) \bigr\} \subseteq \pi_1 (\manifold{N}).
\]
By hypothesis, for every \(i \in \{1, \dotsc, k\}\), the set \(\{\beta_i \gamma_i \beta_i^{-1} \st\beta_i \in \pi_1 (\manifold{N})\}\) is finite and thus the set \(\Gamma\) is also finite. 
We choose \(\mapsset{G} \subset \mapsset{C} (\Sset^1, \manifold{N})\) to be a finite set such that each \(\gamma \in \Gamma\) is homotopic to some \(g \in \mapsset{G}\). 
By \cref{free_decomposition_pi_1}, any map \(f\) that has a free homotopy decomposition into \(f_1, \dotsc, f_k\) is homotopic to a map in \(\mapsset{G}\).

When \(m \ge 2\), the proof is similar and follows from the application of \cref{free_decomposition_pi_m}.
\end{proof}

\subsection{Infinitely many homotopy classes sharing the same free homotopy decomposition}
We now show that for some manifolds infinitely many homotopy classes can be decomposed freely into a finite set of maps.
This implies in particular that the left-hand side in \cref{theorem_bubbles_are_bounded} goes through infinitely many homotopy classes.

\begin{proposition}%
[Infinitely many homotopy classes sharing a free homotopy decomposition]%
\label{proposition_unbounded_glue}%
For every \(m \in \Nset_*\), 
there exists a compact manifold \(\manifold{N}\) 
and a map \(f \in \mapsset{C} (\Sset^m, \manifold{N})\) such that for every 
\(k \in \{2, 3, \dotsc\}\), 
there exists infinitely many homotopy classes in \(\mapsset{C} (\Sset^m, \manifold{N})\) having a free homotopy decomposition into \(k\) copies of \(f\).
\end{proposition}

In the one-dimensional case \(m = 1\), examples can be provided by tori with at least two holes.
The next lemma shows that a \(g\)--hole torus --- or equivalently, an orientable surface of genus \(g\) --- has a fundamental group which is not less complex than a free group on \(g\) generators.

\begin{lemma}
[Free group in the fundamental group of \(g\)--hole tori]
\label{lemma_torus_free_group}
If \(\manifold{N}\) is a \(g\)-hole torus, 
then there exists a surjective homomorphism \(\tau : \pi_1 (\manifold{N}) \to \langle \alpha_1, \dotsc, \alpha_g\rangle\).
\end{lemma}

Here, \(\langle \alpha_1, \dotsc, \alpha_g\rangle\) is the \emph{free group} on the \(g\) generators \(\alpha_1, \dotsc, \alpha_g\).

\begin{proof}%
[Proof of \cref{lemma_torus_free_group}]
The fundamental group \(\pi_1 (\manifold{N})\) of the \(g\)--hole torus \(\manifold{N}\) can be characterized by the group presentation 
\begin{equation*}
  \pi_1 (\manifold{N}) 
  =
  \langle 
    a_1, b_1, \dotsc, a_g, b_g 
  \st
    [a_1, b_1] \dotsm [a_g, b_g] = 1
   \rangle \eofs ,
\end{equation*}
where \([a_i, b_i] = a_i b_i a_i^{-1} b_i^{-1}\) \cite{Hatcher_2002}*{\S 1.2, p.\thinspace 51}.
We define the group homomorphism \(\hat{\tau} : \langle a_1, b_1, \dotsc, a_g, b_g \rangle \to \langle \alpha_1, \dotsc, \alpha_g\rangle\) by setting for each \(i \in \{1, \dotsc, g\}\),  \(\Hat{\tau} (a_i) \defeq \alpha_i\) and \(\Hat{\tau} (b_i) \defeq 1\), and we observe that for each \(i \in \{1, \dotsc, g\}\), 
\(\Hat{\tau} ([a_i, b_i]) = [\Hat{\tau} (a_i), \Hat{\tau} (b_i)] = [\alpha_i, 1] = 1\). 
Hence, we have \(\Hat{\tau}([a_1, b_1] \dotsm [a_g, b_g]) = \Hat{\tau}([a_1, b_1]) \dotsm \Hat{\tau}([a_g, b_g]) = 1\) and thus \(\Hat{\tau}\) induces a quotient homomorphism \(\tau : \pi_1 (\manifold{N}) \to \langle \alpha_1, \dotsc, \alpha_g \rangle \). Since \(\tau (a_i) = \alpha_i\) for each \(i \in \{1, \dotsc, g\}\), the homomorphism \(\tau\) is surjective.
\end{proof}

The next lemma will allow us to prove in algebraic terms that maps in \(\mapsset{C} (\Sset^1, \manifold{N})\) lie in different homotopy groups.

\begin{lemma}%
[Nonconjugacy along a conjugation orbit in a free group]
\label{lemmaFree}
If \(k \in \{2, 3, \dotsc\}\) and if \(\ell, j \in \Nset\), then there exists \(h \in \langle \alpha_1, \dotsc, \alpha_g\rangle\) such that 
\begin{equation*}
 h^{-1} \alpha_1 \alpha_2^{-\ell} \alpha_1^{k - 1} \alpha_2^\ell h
 =
  \alpha_1 \alpha_2^{-j} \alpha_1^{k - 1} \alpha_2^{j}
\end{equation*}
if and only if \(\ell = j\).
\end{lemma}
\begin{proof}
If \(k = \ell\) the statement holds with \(h = 1\).

Conversely, it can be observed that \(\alpha_1 \alpha_2^{-\ell} \alpha_1^{k - 1} \alpha_2^\ell\) and \(\alpha_1 \alpha_2^{-j} \alpha_1^{k - 1} \alpha_2^{j}\) are cyclically reduced words which can be conjugate in a free group if and only the words are cyclic permutation of each other \cite{Magnus_Karrass_Solitar_1966}*{theorem 1.3}.
The statement can also be proved directly. We assume by contradiction that \(\ell > k \ge 0\) and that there exists \(h \in \langle \alpha_1, \dotsc, \alpha_g\rangle\) such that the identity holds. Then both corresponding reduced words should have the same length. Since \(\ell >j \ge 0\), this means that there should be  \((\ell - j) + \operatorname{length} (h)\) cancellations between inverses on the left-hand side, and thus at least one cancellation at the beginning and one cancellation at the end of the word on the left-hand side. Since \(\ell \ne 0\), the cancellation on the left implies that the first letter of \(h\) is \(\alpha_1\) and the cancellation on the right that the first letter of \(h\) is \(\alpha_2\); this is a contradiction.
\end{proof}

\begin{proof}
[Proof of \cref{proposition_unbounded_glue} when \(m = 1\)]
We take \(\manifold{N}\) to be a \(g\)--hole torus, with \(g \ge 2\).
Let \(\tau : \pi_1 (\manifold{N}) \to \langle \alpha_1, \dotsc, \alpha_g\rangle\) 
be the homomorphism of \cref{lemma_torus_free_group} and let \(f \in C (\Sset^1, \manifold{N})\) be homotopic to \(a_1 \in \tau^{-1} (\{\alpha_1\}) \subset \pi_1 (\manifold{N})\). We also fix \(a_2 \in \tau^{-1} (\{\alpha_2\})\).
For every natural number \(\ell \in \Nset\), we choose \(f_\ell \in \mapsset{C} (\Sset^1, \manifold{N})\)
that is homotopic to 
\(a_1 a_2^{-\ell} a_1^{k - 1} a_2^{\ell} \in \pi_1 (\manifold{N})\). 
By \cref{free_decomposition_pi_1}, the map \(f_\ell\) has a free homotopy decomposition into \(k\) copies of the map \(f\).
If for some \(\ell, j \in \Nset\), the maps \(f_\ell\) and \(f_j\) are homotopic, 
then \(a_1 a_2^{-\ell} a_1^{k - 1} a_2^{\ell}\) and \(a_1 a_2^{-j} a_1^{k - 1} a_2^{j}\) are conjugate in \(\pi_1 (\manifold{N})\) and thus, since \(\tau\) is a homomorphism, we deduce that 
\(\alpha_1 \alpha_2^{-\ell} \alpha_1^{k - 1} \alpha_2^{\ell}\) and \(\alpha_1 \alpha_2^{-j} \alpha_1^{k - 1} \alpha_2^{j}\) are conjugate in \(\langle \alpha_1, \dotsc, \alpha_g\rangle\).
By \cref{lemmaFree}, this implies that \(\ell = j\), and thus the maps \(f_\ell\) and \(f_j\) are homotopic if and only if \(\ell = j\).
\end{proof}

For \(m \ge 2\), we rely on the following construction of manifolds:

\begin{lemma}
[Manifold with nontrivial action by the fundamental group]
\label{lemma_counterexample_m_ge_2}
For every \(m \ge 2\), there exists a \((2m +1)\)--dimensional compact Riemannian manifold \(\manifold{N}\) isometrically embedded into \(\Rset^{2m + 2}\) such that 
\(\pi_1 (\manifold{N}) \simeq \Zset\),
\(\pi_m (\manifold{N}) \simeq \Zset^{\Zset}\)
and \(\pi_1 (\manifold{N})\) acts on \(\pi_m (\manifold{N})\) as the translation operator.
\end{lemma}
\begin{proof}
If \(\manifold{X} \defeq \Sset^1 \vee \Sset^m\) is the CW complex obtained by the bouquet construction applied between the circle \(\Sset^1\) and the sphere \(\Sset^m\),
then \(\pi_1 (\manifold{X}) \simeq \Zset\), \(\pi_m (\manifold{X}) \simeq \Zset^{\Zset}\) and \(\pi_1 (\manifold{X})\) acts on \(\pi_m (\manifold{X})\) as the translation operator (see for example \cite{Hatcher_2002}*{example 4.27}). 

We embed the CW complex \(\manifold{X}\) in \(\Rset^{2 m + 2}\) and we consider a neighbourhood \(U\) of \(\manifold{X}\) in \(\Rset^{2 m + 2}\) that has a smooth boundary and such that \(\manifold{X}\) is a retraction of \(U\) and \(\partial U\) is a retraction of \(U \setminus \manifold{X}\).
We define \(\manifold{N} \defeq \partial U\).

We then observe that any Lipschitz-continuous homotopy \(h : \Sset^k \times [0, 1]  \to U\) has a \((k + 1)\)--dimensional image. Since the set \(U \subset \Rset^{2 m + 2}\) is open, if \(k \le m\), the homotopy \(h\) can be perturbed in such a way of not intersecting the \(m\)--dimensional set \(\manifold{X}\). 
This implies that \(\pi_1 (\manifold{N}) \simeq \pi_1 (U \setminus \manifold{X}) \simeq \pi_1 (U) \simeq \pi_1 (\manifold{X}) \simeq \Zset\) 
and \(\pi_m (\manifold{N}) \simeq \pi_m (U \setminus \manifold{X}) \simeq \pi_m (U) \simeq \pi_m (\manifold{X}) \simeq \Zset^\Zset\), with isomorphisms between the actions of \(\pi_1 (\manifold{N})\) on \(\pi_m (\manifold{N})\) and of \(\pi_1 (\manifold{X})\) on \(\pi_m (\manifold{X})\).
\end{proof}

The manifold \(\manifold{N}\) constructed in the proof of \cref{lemma_counterexample_m_ge_2} can be described as the result of gluing \(\Sset^1 \times \Sset^{2m}\) to \(\Sset^m \times \Sset^{m + 1}\) along a trivial sphere \(\Sset^{2 m}\).

\begin{remark}
When \(m = 2\), the construction of the proof of \cref{lemma_counterexample_m_ge_2} yields a \(3\)--dimensional compact Riemannian manifold \(\manifold{N}\) embedded into \(\Rset^{4}\) such that 
\(\pi_1 (\manifold{N})\) is a free group on two generators.
\end{remark}

\begin{proof}
[Proof of \cref{proposition_unbounded_glue} when \(m \ge 2\)]
Let \(\manifold{N}\) be the manifold given by \cref{lemma_counterexample_m_ge_2}.
We fix a map \(f \in \mapsset{C} (\Sset^m, \manifold{N})\) that is not homotopic to a constant and we choose \(a_0 \in \pi_m (\manifold{N})\) homotopic to \(f\). For each \(k \in \Zset\), let \(a_k\) be the result of the action of \(k \in \Zset \simeq \pi_1 (\manifold{N})\) on \(a_0 \in \pi_m (\manifold{N})\).
By \cref{proposition_finite_homotopy_classes}, the homotopy classes that have a free homotopy decomposition into \(k\) copies of the map \(f\) correspond to sets of the form
\(
 \{ a_{i_1 + \ell} + \dotsb a_{i_k + \ell} \st \ell \in \Zset\}\eofs ,
\)
with \(i_1, \dotsc, i_k \in \Zset\). If \(k \ge 2\), there are infinitely many such sets. 
\end{proof}

\section{Upper bound on Sobolev energies by free homotopy decomposition}

\Cref{theorem_bubbles_are_bounded} will be obtained by induction from the corresponding result with \(k = 2\):

\begin{proposition}%
[Estimate of Sobolev energy by free homotopy decomposition into two maps]
\label{proposition_bouquet_pair}%
Let \(m \in \Nset_*\),  \(\manifold{N}\) be a connected Riemannian manifold \(s \in (0, 1]\) and \(p \in [m, +\infty)\). 
If \(p = m/s > 1\) and if \(f \in \mapsset{C} (\Sset^m, \manifold{N})\) has a free homotopy decomposition into \(f_+, f_- \in (\mapsset{C} \cap W^{s, p}) (\Sset^m, \manifold{N})\), then 
\begin{equation*}
  \inf 
    \,\bigl\{
        \mathcal{E}^{s, p} (g) 
      \st 
        g \in (\mapsset{C} \cap W^{s, p}) (\Sset^m, \manifold{N}) 
        \text{ is homotopic to \(f\)}
    \bigr\}\\
 \le 
      \mathcal{E}^{s, p} (f_+) 
    + 
      \mathcal{E}^{s, p} (f_-)
\eofs .
\end{equation*}
\end{proposition}

Computations will be facilitated by parametrizing the sphere \(\Sset^m\) through its \emph{Mercator projection} on the cylinder \(\Sset^{m - 1} \times \Rset\). When \(m = 2\), this corresponds to the projection used by Mercator on the cylinder to cartography the earth.
The Mercator projection is a conformal transformation, and preserves thus the critical Sobolev energy.

\begin{lemma}%
[Conformal derivative integrals under Mercator cylindrical projection]%
\label{lemma_Mercator_1}
For every \(m \in \Nset_*\) and for every \(f \in \Sset^m \to \manifold{N}\), we have \(f \in W^{1, m} (\Sset^m, \manifold{N})\) if and only if \(f \compose \Upsilon \in W^{1, m} (\Sset^{m - 1} \times \Rset, \manifold{N})\), where the map \(\Upsilon : \Sset^{m - 1} \times \Rset \to \Sset^m\) is defined for each \((z, s) \in \Sset^{m - 1} \times \Rset\)  by \(\Upsilon (z, s) \defeq (z \sech s, \tanh s)\). Moreover, 
\begin{equation*}
 \int_{\Sset^m}
 \abs{D f}^m 
 =\
 \int\limits_{\Rset \times \Sset^{m - 1} } 
    \hspace{-.5em}
   \abs{D (f \compose \Upsilon)}^m
\eofs .
\end{equation*}
\end{lemma}

\begin{proof}
We compute, if \((z, s) \in \Sset^{m - 1} \times \Rset \), if \((u, r) \in \Rset^{m}\times \Rset \) and if \(u \cdot z = 0\),
\begin{equation}
\label{eq_KeQui7eik3}
\begin{split}
  \bigabs{D \Upsilon (z, s)[(r, u)]}^2
 &
 = 
  \bigabs{z\, r \sech s \tanh s - u \sech s}^2
  + \bigabs{r\, (\sech s)^2}^2\\
 &
 = r^2 (\sech s)^2 \bigl((\sech s)^2 + (\tanh s)^2\bigr) + \abs{u}^2 \sech s^2\\
 &= (\sech s)^2 \bigl(\abs{r}^2+ \abs{u}^2\bigr)\eofs ; 
\end{split}
\end{equation}
since \(\abs{z} = 1\). It thus follows that the mapping \(\Upsilon\) is conformal and the identity holds.
\end{proof}

The fractional counterpart of \cref{lemma_Mercator_1} is an identity between the fractional integral on the sphere and a fractional integral with exponenially decaying potential in the longitudinal direction of the cylinder.

\begin{lemma}%
[Conformal fractional integrals under Mercator cylindrical projection]%
\label{lemma_Mercator}
For every \(m \in \Nset_*\), for every \(p \in (0, +\infty)\) and for every \(f : \Sset^m \to \manifold{N}\), 
\begin{multline*}
 \int_{\Sset^m}\int_{\Sset^m}
 \frac{d_{\manifold{N}} (f (y), f (x))^p}{\abs{y - x}^{2m}}\diff y \diff x\\
 =\int_{\Sset^{m - 1}} \int_{\Rset}
 \int_{\Sset^{m - 1} } 
 \int_{\Rset} 
  \frac
    {d_{\manifold{N}} \bigl(f (w \sech t, \tanh t), f (z \sech s, \tanh s)\bigr)^p}
    {\bigl((2\sinh \tfrac{t - s}{2})^2 + \abs{w - z}^2\bigr)^m} 
    \diff t
    \diff w 
    \diff s 
    \diff z \eofs .
\end{multline*}
\end{lemma}
\begin{proof}
We define the Mercator projection \(\Upsilon : \Sset^{m - 1} \times \Rset \to \Sset^m\) as in the statement of \cref{lemma_Mercator_1} and we observe that \eqref{eq_KeQui7eik3} holds and thus for every \((z, s) \in \Sset^{m - 1} \times \Rset\), we have \(\operatorname{Jac} \Upsilon (z, s) = \sech s\).
Moreover, if \((z, s), (w, t) \in \Sset^{m - 1} \times \Rset\), since \(\abs{z} = \abs{w} = 1\), we have 
\begin{equation*}
\begin{split}
 \bigabs{\Upsilon (t, w) - \Upsilon (s, z)}^2 
 &= \abs{w \sech t - z \sech s}^2 + \abs{\tanh t - \tanh s}^2\\
  &= \sech t \, \sech s  \, \abs{w - z}^2
  + \abs{\sech t - \sech s}^2
   +\abs{\tanh t - \tanh s}^2\\
 & =  \sech t \, \sech s \, \bigl(\abs{w - z}^2
 + 2 (\cosh t \cosh s - \sinh t \sinh s - 1)\bigr)\\
 & =  \sech t \, \sech s\, \bigl(\abs{w - z}^2
 + 2 (\cosh (t - s) - 1)\bigr)\\
 &= \sech t \, \sech s\, \Bigl(\bigl(2\sinh \tfrac{t - s}{2}\bigr)^2 + \abs{w - z}^2\Bigr)\eofs .
\end{split}
\end{equation*}
The identity follows then by a change of variable \(x = \Upsilon (z, s)\) and
\(y = \Upsilon (w, t)\).
\end{proof}

The proof of \cref{proposition_bouquet_pair} also relies on a construction of maps that are constant on some set.

\begin{lemma}
[Approximation of the identity by maps degenerate at a point]
\label{lemma_constant_disk}
For every \(b \in \manifold{N}\) and every \(\varepsilon > 0\), there exists a map \(\Theta \in \mapsset{C}^1 (\manifold{N}, \manifold{N})\) which is homotopic to the identity and such that \(\Theta = b\) in a neighbourhood of \(b\) and for every \(y, z \in \manifold{N}\), \(d (\Theta (z),  \Theta (y)) \le (1 + \varepsilon) \,d (z, y)\).
\end{lemma}
\begin{proof}%
\resetconstant
We fix a function \(\eta \in \mapsset{C}^\infty (\Rset, [0, +\infty))\) such that \(\eta = 0\) on \((-\infty, -2]\) and \(\eta = 1\) on \([-1, + \infty)\). We first define the function \(\Xi_\lambda : \Rset^m \to \Rset^m\) for \(\lambda \in (0, + \infty)\) and for \(u \in \Rset^m\) by \(
    \Xi_\lambda (u)
  \defeq
    \eta (\lambda \ln \abs{u}) u, 
\)
and we observe that for every \(u, v \in \Rset^{m - 1}\),
\(
 \abs{D \Xi_{\lambda} (u) [v]} \le (1 +  \C \lambda) \abs{v}\).
We
define now for each \(y \in \manifold{N}\), 
\[
 \Theta_\lambda (y) \defeq 
 \begin{cases}
      \exp_b \bigl(\,\Xi_\lambda (\exp_b^{-1} (y))\,\bigr)
      & \text{if \(d (y, b) \le \operatorname{inj}_{\manifold{N}} (b)\)},\\
      y & \text{otherwise}.
   \end{cases}
\]
where \(\exp_b\) is the Riemmanian exponential map on \(\manifold{N}\) at \(b\) and \(\operatorname{inj}_{\manifold{N}} (b)\) is the injectivity radius of the Riemannian manifold \(\manifold{N}\) at the point \(b\).
We obtain the conclusion by taking \(\lambda > 0\) small enough.
\end{proof}

\begin{proof}
[Proof of \cref{proposition_bouquet_pair}]%
\resetconstant%
We choose a coordinate system so that \(a = (0, \dotsc, 0, 1) \in \Sset^m \subset \Rset^{m + 1}\).
By \cref{lemma_constant_disk}, for every \(\varepsilon > 0\), there exists maps \(\Theta^{\pm} : \manifold{N} \to \manifold{N}\) that are constant in a neighborhood of the point \(f_{\pm} (\mp a)\). It follows then that \(g_+ \defeq \Theta^+ \compose f_+\) is constant in a neighborhood of \(-a\) and \(g_- \defeq \Theta^- \compose f_-\) is constant in a neighborhood of \(a\) and
\begin{equation*}
    \mathcal{E}^{s, p} (g_\pm)
  \le
    (1 + \varepsilon)^p
    \,
    \mathcal{E}^{s, p} (f_\pm).
\end{equation*}

Up to a homotopy, we can consider that the map \(f\) is constant in a neighborhood of the equator \(\partial B_{\pi/2} (a) = \partial B_{\pi/2} (-a)\), that \(f \vert_{\Bar{B}_{\pi/2} (a)}\) is homotopic to \(f_+\) on \(\Sset^m \simeq \Bar{B}_{\pi/2} (a)/\partial B_{\pi/2} (a)\) and that \(f \vert_{\Bar{B}_{\pi/2} (-a)}\) is homotopic to \(f_-\) on \(\Sset^m \simeq \Bar{B}_{\pi/2} (-a)/\partial B_{\pi/2} (-a)\).

We consider the cylinder \(K = \partial (\Bset^m \times [-1, 1]) = \Bset^{m} \times \{-1, 1\} \cup \Sset^{m - 1} \times [-1, 1]\)
and a map \(\Phi : K \to \Sset^{m}\) such that  \(\Phi \vert_{\Bset^{m} \times \{\pm 1\}}\) is a homeomorphism with \(B_{\pi/2}(\pm a)\) and \(\Phi (x, s) = x\) for every \((x, s) \in \Sset^{m - 1} \times [-1, 1]\).
If we define \(\Psi : K \to \Sset^m\) by \(\Psi (x, s) \defeq (x, s)/\abs{(x, s)}\), we observe that \(\Psi\) is a homeomorphism and that the maps \(\Psi\) and \(\Phi\) are homotopic.

Since \(f\) is homotopic to \(g_{\pm}\) on \(\Sset^m \simeq \Bar{B}_{\pi/2} (\pm a)/\partial B_{\pi/2} (\pm a)\), there exists a homotopy \(H \in \mapsset{C} ( (\Bset^{m} \times \{-1, 1\}) \times [0, 1], \manifold{N})\) such that \(H (\cdot, 0) = f \compose \Phi\) on \(\Bset^m \times \{-1, 1\}\), \(H (\cdot, \pm 1, 1) = g_{\pm} \compose \Phi\) on \(\Bset^m\) and \(H (\cdot, t)\) is constant for every \(t \in [0, 1]\) on both sets \(\partial \Bset^{m} \times \{-1\}\) and \(\partial \Bset^{m} \times \{1\}\) (with possibly constant values differering on one set from the other).
By the homotopy extension property (see for example \cite{Hatcher_2002}*{Proposition 0.16}), there exists a homotopy \(\Gamma \in \mapsset{C} ([-1, 1] \times [0, 1], \manifold{N})\) such that 
\(\Gamma (s, 0) = f \vert_{\partial B_{\pi/2} (a)}\) and for every \(t \in [0, 1]\),
\(\Gamma (\pm 1, t) = H (\cdot, \pm 1, t)\) on \(\partial \Bset^m\).
We define the map \(\gamma \in \mapsset{C} ([0, 1], \manifold{N})\) for each \(s \in [0, 1]\) by \(\gamma (s) \defeq \Gamma (s, 1)\).
By a regularization argument, we can assume that \(\gamma = \Bar{\gamma} \vert_{[-1, 1]}\) for some \(\Bar{\gamma} \in \mapsset{C}^1 (\Rset, \manifold{N})\) such that \(\Bar{\gamma} = g_- (a)\) on \((-\infty, -1]\) and \(\Bar{\gamma} = g_+ (a)\) on \([1, +\infty)\).
We observe that \(f \compose \Phi\) is homotopic to the map \(h : K \to \manifold{N}\) defined for \((x, s) \in K\) by 
\begin{equation*}
h (x, s)
\defeq
\begin{cases}
  g_- \bigl(\Phi (x,s)\bigr) & \text{if \(s = -1\)},\\
  \gamma (s) & \text{if \(-1 < s < 1\)},\\
  g_+ \bigl(\Phi (x,s)\bigr) & \text{if \(s = 1\)}.
\end{cases}
\end{equation*}
It follows then that \(h \compose \Psi^{-1}\) is homotopic to \(f\) on \(\Sset^m\).

We now consider the maps \(\Tilde{g}_\pm : \Sset^{m - 1} \times \Rset \to \manifold{N}\) defined
for \((z, s) \in \Sset^{m - 1} \times \Rset\) by 
\begin{equation*}
      \Tilde{g}_\pm (z, s) 
    \defeq 
      g_\pm (z \sech s, \tanh s)\eofs .
\end{equation*}
We observe that there exists \(s_+ \in \Rset\) such that if \(s \le -s_+\) and \(z \in \Sset^{m - 1}\), then \(\Tilde{g}_+ (z, s) = g_+ (a_-)\). 
Similarly, there exists \(s_- \in \Rset\) such that if \(s \ge s_-\) and \(z \in \Sset^{m - 1}\), then \(\Tilde{g}_- (z, s) = g_- (a_+)\). 
We construct now for each \(\lambda \in (0, + \infty)\), the map \(\Tilde{g}_\lambda : \Sset^{m - 1} \times \Rset \to \manifold{N}\) by setting for each \((z, s) \in \Sset^{m - 1} \times \Rset\),
\begin{equation*}
\Tilde{g}_\lambda (z, s)
\defeq 
\begin{cases}
    \Tilde{g}_- (z, s + 2 \lambda + s_- )& \text{if \(s \in (-\infty, -2 \lambda]\)},\\
    \gamma (s/\lambda) & \text{if \(s \in [-2\lambda, 2 \lambda]\)},\\
    \Tilde{g}_+ (z, s - 2 \lambda - s_+ )& \text{if \(s \in [2 \lambda, +\infty) \).}
  \end{cases}
\end{equation*}
We define now for every \(\lambda \in (0, + \infty)\) the map \(g_\lambda: \Sset^{m} \to \manifold{N}\) for each \((y, t) \in \Sset^m \subset \Rset^m \times \Rset\),
\begin{equation*}
    g_\lambda (y, t) 
  \defeq 
    \Tilde{g}_\lambda \bigl(y/\abs{y}, \tanh^{-1} (t)\bigr)
    \eofs .
\end{equation*}
By construction, the map \(g_\lambda\) is homotopic to \(h \compose \Psi^{-1}\) on \(\Sset^m\), which in turn is homotopic to the map \(f\) on \(\Sset^m\).
It remains to estimate its Sobolev energy \(\mathcal{E}^{s, p} (g_\lambda)\).

If \(s = 1\), we have by \cref{lemma_Mercator_1},
\begin{equation}
\label{identity_flambda_energy_1}
\begin{split}
    \mathcal{E}^{1, m} (g_\lambda)
    &
    =
    \int_{\Sset^m}
      \abs{D g_\lambda}^m 
  =
    \int\limits_{\Sset^{m - 1} \times \Rset}
        \abs{D \Tilde{g}_\lambda}^m\\
  &=  
    \int\limits_{\Sset^{m - 1} \times (-\infty, s_-]} \hspace{-1em} \abs{D \Tilde{g}_-}^m
    +  \abs{\Sset^{m - 1}} \,\lambda^{m - 1} \int_{-1}^1 \abs{\gamma'}^m 
    + \int\limits_{\Sset^{m - 1} \times [-s_+, +\infty)}  \hspace{-1em} \abs{D \Tilde{g}_+}^m\\
  & \le \mathcal{E}^{1, m} (g_+) + \mathcal{E}^{1, m} (g_-) + \C \lambda^{m - 1} 
        \eofs .
\end{split}
\end{equation} 
The conclusion then follows by letting \(\lambda \to 0\) and \(\varepsilon \to 0\).

If \(0 < s < 1\), we have by \cref{lemma_Mercator}, 
\begin{multline}
\label{identity_gpm_energy}
\int_{\Sset^m}\int_{\Sset^m}
\frac{d_\manifold{N} (g_\pm (y), g_\pm (x))^p}{\abs{y - x}^{2m}}\diff y \diff x\\
  =
    \int_{\Sset^{m - 1}} 
    \int_{\Rset}
    \int_{\Sset^{m - 1} } 
    \int_{\Rset} 
      \frac
        {d_\manifold{N} (\Tilde{g}_\pm (w, t), \Tilde{g}_\pm (z, s))^p}
        {\bigl((2\sinh \tfrac{t - s}{2})^2 + \abs{w - z}^2\bigr)^m} 
      \diff t 
      \diff w 
      \diff s 
      \diff z 
      \eofs .
\end{multline}
and for every \(\lambda > 0\),
\begin{multline}
\label{identity_flambda_energy}
    \int_{\Sset^m}\int_{\Sset^m}
      \frac
        {d_{\manifold{N}} (g_\lambda (y), g_\lambda (x))^p}
        {\abs{y - x}^{2m}}
        \diff y 
        \diff x\\
  =
    \int_{\Sset^{m - 1}} 
      \int_{\Rset}
        \int_{\Sset^{m - 1} } 
            \int_{\Rset} 
              \frac{d_{\manifold{N}} (\Tilde{g}_\lambda (w, t), \Tilde{g}_\lambda (z, s))^p}
              {\bigl((2\sinh \tfrac{t - s}{2})^2 + \abs{w - z}^2\bigr)^m} 
              \diff t
            \diff w 
          \diff s 
        \diff z \eofs .
\end{multline}
We first estimate the tails in \eqref{identity_flambda_energy} 
\begin{multline}
\label{estimate_tail_minus}
\int_{\Sset^{m - 1}} \int_{-\infty}^{-\lambda}
\int_{\Sset^{m - 1} } \int_{-\infty}^{-\lambda} \frac{d_\manifold{N} (\Tilde{g}_\lambda (w, t), \Tilde{g}_\lambda (z, s))^p}{\bigl((2\sinh \tfrac{t - s}{2})^2 + \abs{w - z}^2\bigr)^m} \diff t \diff w \diff s \diff z\\
\le
\int_{\Sset^{m - 1}} \int_{\Rset}
\int_{\Sset^{m - 1} } \int_{\Rset} \frac{d_{\manifold{N}}(\Tilde{g}_- (w, t), \Tilde{g}_- (z, s))^p}{\bigl((2\sinh \tfrac{t - s}{2})^2 + \abs{w - z}^2\bigr)^m} \diff t \diff w \diff s \diff z
\end{multline}
and 
\begin{multline}
\label{estimate_tail_plus}
\int_{\Sset^{m - 1}} \int_{\lambda}^{+\infty}
\int_{\Sset^{m - 1} } \int_{\lambda}^{+\infty} 
    \frac
      {d_{\manifold{N}} (\Tilde{g}_\lambda (w, t), \Tilde{g}_\lambda (z, s))^p}
      {\bigl((2\sinh \tfrac{t - s}{2})^2 + \abs{w - z}^2\bigr)^m} \diff s \diff w \diff t \diff z\\
\le
\int_{\Sset^{m - 1}} \int_{\Rset}
\int_{\Sset^{m - 1} } \int_{\Rset} 
    \frac
      {d_{\manifold{N}} (\Tilde{g}_+ (w, t), \Tilde{g}_+ (z, s))^p}
      {\bigl((2\sinh \tfrac{t - s}{2})^2 + \abs{w - z}^2\bigr)^m} 
      \diff t \diff w \diff s \diff z \eofs .
\end{multline}

Next, if \(m \ge 2\), we apply a change of variable through 
a stereographic projection on \(\Sset^{m - 1}\):
for \(v \in z^\perp \simeq \Rset^{m - 1}\), we set 
\(w = \Sigma_z (v) \defeq (z (1 - \abs{v}^2) + 2 v)/(1 + \abs{v}^2)\), so that, since \(\abs{w} = 1\), we have 
\(
  \abs{w - \Sigma_z (v)}^2
  = 4 \abs{v}^2/(1 + \abs{v}^2)
\),
and, for every \(k \in z^\perp\), 
\(
 \abs{D \Sigma_z (v)[k]}
 = 2 \abs{k} /(1 + \abs{v}^2)
\),
so that \(\operatorname{Jac} \Sigma_z = 2^{m - 1}/(1 + \abs{v}^2)^{m - 1}\) and therefore, for every \(s, t \in [-2\lambda, 2 \lambda]\), 
\begin{equation*}
\begin{split}
  \int_{\Sset^{m - 1} }  
    \int_{\Sset^{m - 1}} 
    &
    \frac
    {d_{\manifold{N}} (\Tilde{g}_\lambda (w, t), \Tilde{g}_\lambda (z, s))^p}{\bigl((\sinh \tfrac{t - s}{2})^2 + \abs{w - z}^2\bigr)^m} \diff w \diff z\\
  &= 
    \C  
    \,
    d_{\manifold{N}} \bigl(\gamma (t/\lambda), \gamma (s/\lambda)\bigr)^p 
    \int_{\Rset^{m - 1}}  
      \frac
        {(1 + \abs{v}^2) }
        {\bigl((\sinh \tfrac{t - s}{2})^2 (1+ \abs{v}^2) + \abs{v}^2\bigr)^m}
      \diff v\\
  &= 
    \Cl{cst_ahPaz4uo5A} 
    \,
    d_{\manifold{N}} \bigl(\gamma (t/\lambda), \gamma (s/\lambda)\bigr)^p 
    \int_{0}^{+ \infty} 
      \frac
        {(1 + r^2)\, r^{m - 2}}
        {\bigl((2\sinh \tfrac{t - s}{2})^2 (1+ r^2) + r^2\bigr)^m} 
      \diff r\\
  &\le 
    \frac
      {\Cr{cst_ahPaz4uo5A} }
      {(2 \sinh \tfrac{t - s}{2})^2}
    \,
    d_{\manifold{N}} \bigl(\gamma (t/\lambda), \gamma (s/\lambda)\bigr)^p 
    \int_{0}^{+ \infty} 
      \frac
        {1}
        {\bigl((2\sinh \tfrac{t - s}{2})^2 (1+ r^2) + r^2\bigr)^{\frac{m}{2}}} 
      \diff r\\  
  &\le 
    \frac
      {\C}
      {(\sinh \tfrac{t - s}{2})^2} 
    \,
      d_{\manifold{N}} \bigl(\gamma (t/\lambda), \gamma (s/\lambda)\bigr)^p 
      \int_{0}^{+\infty} 
        \frac
          {1}
          {(\abs{\sinh \tfrac{t - s}{2}} + r)^{m}} 
        \diff r\\
&= \frac{\Cl{cst_1d_sp} \,d_{\manifold{N}} \bigl(\gamma (t/\lambda), \gamma (s/\lambda)\bigr)^p}{\bigabs{\sinh \tfrac{t - s}{2}}^{m + 1}}\eofs .
\end{split}
\end{equation*}
The same estimate still holds when \(m = 1\).
We have thus
\begin{equation}
\begin{split}
\label{estimate_connection}
    \int_{\Sset^{m - 1}} 
      \int_{-2\lambda}^{2 \lambda}
          \int_{\Sset^{m - 1} }
            \int_{-2\lambda}^{2 \lambda} &
              \frac
                {d_{\manifold{N}} (\Tilde{g}_\lambda (w, t), \Tilde{g}_\lambda (z, s))^p}
                {\bigl((2\sinh \tfrac{t - s}{2})^2 + \abs{w - z}^2\bigr)^m}
              \diff t 
            \diff w 
          \diff s 
        \diff z\\
  &\le 
    \Cr{cst_1d_sp}
    \int_{-2\lambda}^{2\lambda} 
      \int_{-2 \lambda}^{2\lambda} 
        \frac
          {d_{\manifold{N}}(\gamma (t/\lambda), \gamma (s/\lambda))^p}{\abs{\sinh \tfrac{t - s}{2}}^{m + 1}} 
        \diff t
      \diff s\\
&\le \C
  \frac{1}{\lambda^p} \int_{-2 \lambda}^{2\lambda } \int_{-2 \lambda}^{2\lambda} 
    \frac
      {\abs{t - s}^p}
      {\bigabs{\sinh \tfrac{t - s}{2}}^{m + 1}} 
    \diff t 
    \diff s 
  \le \frac{\C}{\lambda^{p - 1}}\eofs .
\end{split}
\end{equation}
Finally, we observe that if \((s, t) \in \Rset^2 \setminus (]-\infty, -\lambda]^2 \cup [-2\lambda, 2 \lambda]^2 \cup [\lambda, +\infty)^2\) and \(s \le t\),
then \(s \le t -\lambda\), \(s \le \lambda\) and \(t \ge -\lambda\).
We have then, under the changes of variables \(\sigma = t - s\) and \(\tau = t + \lambda\),
\begin{equation}
\label{estimate_long_range}
\begin{split}
  \iint
    \limits_{
      \substack{
        (s, t) \in \Rset^2\\ 
        s \le t -\lambda\\
        s \le \lambda \\ 
        t \ge -\lambda }}  
  &\int_{\Sset^{m - 1}} \int_{\Sset^{m - 1}}  
    \frac
      {d_{\manifold{N}} (\Tilde{g}_\lambda (w, t), \Tilde{g}_\lambda (z, s))^p}{\bigl((2\sinh \tfrac{t - s}{2})^2 + \abs{w - z}^2\bigr)^m} \diff w \diff z \diff s \diff t \\[-2em]
  &\le 
      \Cl{cst_Aedai3Aibe} 
      \int_{-\lambda}^{+\infty} 
        \int_{-\infty}^{\min (t - \lambda, \lambda)} 
        \hspace{-1em}
        \frac{1}{\abs{\sinh \tfrac{t - s}{2}}^{2 m}}
  \diff s 
  \diff t
= \Cr{cst_Aedai3Aibe} \int_{0}^{+\infty} \int_{\max (\lambda, \tau - 2 \lambda)}^{+\infty} \frac{1}{\abs{\sinh \tfrac{\sigma}{2}}^{2 m}}\diff \sigma \diff \tau\\
&
\le \C \int_{0}^{+\infty} \int_{\frac{\lambda + \tau}{4}}^{+\infty} \frac{1}{\abs{\sinh \tfrac{\sigma}{2}}^{2 m}}\diff \sigma \diff \tau
  \le 
    \Cl{cst_disxdp} 
    \int_{0}^{+\infty} 
      e^{-\frac{m}{4} (\lambda + \tau)} 
      \diff \tau
  =
    \frac
      {
        4
        \,
        \Cr{cst_disxdp}
        \,
        e^{-\frac{m \lambda}{4}}
      }
      {m}
    ,
\end{split}
\end{equation}
since \(\frac{\lambda + \tau}{4} \le \frac{\tau - 2 \lambda}{4} + \frac{3 \lambda}{4}
\le \max (\tau - 2 \lambda, \lambda)\).
By combining the identities \eqref{identity_flambda_energy} and  \eqref{identity_gpm_energy} together with the estimates \eqref{estimate_tail_minus}, \eqref{estimate_tail_plus}, \eqref{estimate_connection} and \eqref{estimate_long_range}, we obtain
\begin{multline*}
\int_{\Sset^m}\int_{\Sset^m}
\frac{d_{\manifold{N}} (g_\lambda (y), g_\lambda (x))^p}{\abs{y - x}^{2m}}\diff y \diff x\\
\le \int_{\Sset^m}\int_{\Sset^m}
\frac{d_{\manifold{N}} (g_+ (y), g_+ (x))^p}{\abs{y - x}^{2m}}\diff y \diff x
+  \int_{\Sset^m}\int_{\Sset^m}
\frac{d_{\manifold{N}} (g_- (y), g_- (x))^p}{\abs{y - x}^{2m}}\diff y \diff x
+ \frac{\C}{\lambda^{p - 1}},
\end{multline*}
and we reach thus the conclusion, by taking  \(\lambda > 0\) and \(\varepsilon > 0\) arbitrarily small.
\end{proof}

\section{Estimates of free homotopy decomposition on the sphere}

\subsection{Extension}
In order to prove \cref{theorem_bounded_finite_bubbles}, we first extend the map \(f\) on the sphere \(\Sset^m\) to a map \(F\) on the ball \(\Bset^{m + 1}\) taking its value into the ambient space, by relying on the next proposition which provides a suitably controlled extension.
When we endow the ball \(\mathbb{B}^{m + 1}\) with the \emph{Poincar\'e metric} of the hyperbolic space \(\Hset^{m + 1}\), that is, if  we consider the metric defined as quadratic form for \(z \in \Bset^{m + 1}\) and \(v \in \Rset^{m + 1}\) by 
\begin{equation}
\label{def_Poincare_Model}
    g_z (v) 
  \defeq
    \frac{4 \abs{v}^2}{\bigl(1 - \abs{z}^2\bigr)^2},
\end{equation}
we obtain uniform estimates on the measure of the set on which the function \(F\) is far frow the set of values on the boundary \(f (\Sset^{m})\).

\begin{proposition}%
[Extension to the hyperbolic space]%
\label{lemmaExtensionHyperbolic}%
Let \(m \in \Nset_*\). There exists a constant \(C > 0\) such that for every \(\nu \in \Nset_*\) and every function \(f \in \mapsset{C}^\infty (\Sset^m, \Rset^\nu)\),
there exists a function \(F \in \mapsset{C}^\infty (\Bar{\Bset}^{m + 1}, \Rset^\nu) \cap \mapsset{C}^\infty (\Bset^{m + 1}, \Rset^\nu)\) such that 
\begin{enumerate}%
[(i)]
\item \label{itExtensionHyperbolicTrace} 
  \(F \vert_{\partial \Bset^{m + 1}} = f\),
\item 
    \label{itExtensionHyperbolicLipschitz} 
  for every point \(x \in \Hset^{m + 1} \simeq \Bset^{m + 1}\), 
  \begin{equation*}
      \abs{D F (x)}_{\Hset^{m + 1}} 
    \le 
      m \osc_{\Sset^m} f
      \eofs ,
  \end{equation*}
\item 
  \label{itExtensionHyperbolicMeasure} 
  if \(\delta > \varepsilon\), 
  then
  \begin{equation*}
    \mu_{\Hset^{m+1}} 
      \bigl(
        \bigl\{ x \in \Hset^{m + 1} 
               \st \dist \bigl(F (x), f (\Sset^m)\bigr) 
                    \ge \delta 
        \bigr\}
      \bigr)
  \le 
    \frac
      {C}
      {\delta - \varepsilon}      
    \hspace{-.5em}
    \iint 
      \limits_{\Sset^m \times \Sset^m}
        \hspace{-.5em}
      \frac
        {\bigl(\abs{f (y) - f (x)} - \varepsilon)_+}
        {\abs{y - x}^{2 m}}
        \diff y \diff x
      \eofs .
  \end{equation*}
\end{enumerate}
\end{proposition}

In this statement, the \emph{oscillation} of the function \(f\) is defined as
\begin{equation*}
    \osc_{\Sset^m} f 
  \defeq 
    \diam f (\Sset^m)
  =
    \sup_{x, y \in \Sset^m} d_{\manifold{N}} \bigl(f (y), f (x)\bigr)
    \eofs ,
\end{equation*}
\emph{In Euclidean terms,} the estimates of \cref{lemmaExtensionHyperbolic} read in view of the definition of the Poincaré metric \eqref{def_Poincare_Model} as follows:
for every \(z \in \Bset^{m + 1}\), 
\begin{equation*}
   \abs{D F (z)} \le \frac{2 m \osc_{\Sset^m} f}{1 - \abs{z}^2} 
\end{equation*}
and for every \(\delta > \varepsilon\),
\begin{equation}
\label{ineq_meas_osc}
  \int\limits_{\substack{x \in \Hset^{m + 1}\\
                   \dist (F (x), f (\Sset^m)) \ge \delta}}
      \frac{2^{m + 1}}{\bigl(1 - \abs{x}^2\bigr)^{m + 1}} \diff x
  \le 
    \frac{C}
          {\varepsilon - \delta} 
    \iint 
      \limits_{\Sset^m \times \Sset^m}
      \frac
        {\bigl(\abs{f (y) - f (x)} - \varepsilon)_+}
        {\abs{y - x}^{2 m}} 
      \diff y 
      \diff x
            \eofs .
\end{equation}
When the function \(f\) is bounded, the latter inequality \eqref{ineq_meas_osc} is a direct consequence of the work of Jean \familyname{Bourgain}, Haïm \familyname{Brezis} and \familyname{Nguyên} Hoài-Minh \cite{BourgainBrezisNguyen2005CRAS}*{lemma 2.1}.

\resetconstant
When \(f \in W^{s, p} (\Sset^m, \Rset^\nu)\) with \(s \in (0, 1)\) and \(s p = m\), 
the assertion \eqref{itExtensionHyperbolicMeasure} in \cref{lemmaExtensionHyperbolic} with \(\varepsilon  =  \frac{\delta}{2}\) implies that 
\begin{equation}
\label{ineq_hypmeas_Wsp} 
    \mu_{\Hset^{m+1}} 
      \bigl(
        \bigl\{ 
          x \in \Hset^{m + 1} 
        \st 
          \dist \bigl(F (x), f (\Sset^m)\bigr) \ge \delta 
        \bigr\}
      \bigr)
  \le 
    \frac
      {\C}
      {\delta^{p}} 
    \iint 
      \limits_{\Sset^m \times \Sset^m}
      \frac
        {\abs{f (y) - f (x)}^{p}}
        {\abs{y - x}^{2 m}} 
      \diff y \diff x
  \eofs .
\end{equation}
This inequality \eqref{ineq_hypmeas_Wsp} can be obtained  when \(s + \frac{1}{p} = s (1 + \frac{1}{m}) < 1\) by combining the classical extension \(F \in W^{s + 1/p, p} (\Bset^{m + 1}, \Rset^{\nu})\) of \(f \in W^{s, p} (\Sset^m, \Rset^{\nu})\) for linear fractional Sobolev spaces together with a fractional Hardy inequality \cite{Dyda2004}*{theorem 1.1} applied to the function \(G = \dist (F, f (\Sset^m)) \in W^{s + 1/p, p}_0 (\Bset^{m + 1}, \Rset^\nu)\):
\begin{equation*}
\begin{split}
    \int_{\Bset^{m + 1}} 
      \frac
        {\abs{G (x)}^{p}}
        {(1-\abs{x})^{m + 1}} 
      \diff x
  &
  \le
    \Cl{cst_AeDief8wah}
    \iint\limits_{\Bset^{m + 1} \times \Bset^{m + 1}} 
      \frac
        {\abs{G (y) - G (x)}^{p}}
        {\abs{y - x}^{2 m + 2}} 
      \diff y 
      \diff x
 \\
  &
  \le
    \Cr{cst_AeDief8wah}
    \iint\limits_{\Bset^{m + 1} \times \Bset^{m + 1}} 
      \frac
        {\abs{F (y) - F (x)}^{p}}
        {\abs{y - x}^{2 m + 2}} 
      \diff y 
      \diff x
 \\ &
 \le 
   \C 
    \iint\limits_{\Sset^{m} \times \Sset^{m}} 
      \frac
        {\abs{f (y) - f (x)}^{p}}
        {\abs{y - x}^{2 m}} 
      \diff y
      \diff x
      \eofs;
  \end{split}
\end{equation*}
the estimate \eqref{ineq_hypmeas_Wsp} follows then from the classical Chebyshev inequality.

The proofs of the counterpart of \cref{theorem_bounded_finite_bubbles} for \(W^{1/2, 2} (\Sset^1, \manifold{N})\) \cite{Kuwert1998}, \(W^{1, m} (\Sset^m, \manifold{N})\) \cite{DuzaarKuwert1998} and  \(W^{1-1/m, m} (\Sset^m, \manifold{N})\) \cite{Muller2000}, rely on a compactness argument on an extension of the map and do not explicitly estimate singular sets as in \cref{lemmaExtensionHyperbolic}. 

\medbreak

The proof of \cref{lemmaExtensionHyperbolic} follows the strategy of Jean \familyname{Bourgain}, Haïm \familyname{Brezis}, Petru \familyname{Mironescu} and \familyname{Nguyên} Hoài-Minh \citelist{\cite{BourgainBrezisMironescu2005CPAM}*{lemma 1.3}\cite{BourgainBrezisNguyen2005CRAS}}.
Since in the sequel we will work with the Poincaré ball model of the hyperbolic space, the proof uses the hyperharmonic extension as in \citelist{\cite{DuzaarKuwert1998}\cite{PetracheVanSchaftingen2017}}; this construction corresponds to the harmonic extension in the two-dimensional case \(m +1 = 2\) \cite{Kuwert1998}*{\S 2} and to the biharmonic extension when \(m + 1 = 4\) \cite{PetracheRiviere2015}.

\begin{proof}%
[Proof of \cref{lemmaExtensionHyperbolic}]%
\resetconstant
We define the function \(F : \Bset^{m + 1} \to \Rset^\nu\) to be the hyperharmonic extension  of the function \(f\), defined for each \(z \in \Bset^{m + 1}\) by \cite{Ahlfors1981}*{\S V}
\begin{equation}
\label{eq_def_hyperharmonic}
    F (z) 
  \defeq
    (1 - \abs{z}^2)^m 
    \fint_{\Sset^{m}} 
      \frac{f (y)}{\abs{z - y}^{2m}} 
      \diff y
  = 
    \frac{(1 - \abs{z}^2)^m}{\abs{\Sset^m}} 
    \int_{\Sset^{m}} 
      \frac{f (y)}{\abs{z - y}^{2m}} 
      \diff y
  \;
  .
\end{equation}
The hyperharmonic extension is equivariant under the action of the conformal transformations of the ball and of the sphere, both corresponding to the group of \((m+1)\)--dimensional Möbius transformation preserving the unit ball:
if \(T : \Bset^{m + 1} \to \Bset^{m + 1}\) is a conformal transformation, then \(F \compose T\) is the hyperharmonic extension of \(f \compose T\).

The assertion \eqref{itExtensionHyperbolicTrace} holds since by conformal invariance for every \(z \in \Bset^{m + 1}\), 
\[
  \frac{(1 - \abs{z}^2)^m}{\abs{\Sset^m}} 
    \int_{\Sset^{m}} 
      \frac{1}{\abs{z - y}^{2m}} 
      \diff y = 1
\eofs.
\]

In order to prove the assertion \eqref{itExtensionHyperbolicLipschitz}, we first note that the Möbius transformations preserving the ball are exactly the isometries of the hyperbolic space in the Poincaré disk model \cite{Ahlfors1981}*{\S II},
and thus, in view of the equivariance of the hyperharmonic extension, it is sufficient to consider the case \(z = 0\). 
We have then for every \(x \in \Sset^m\)
\begin{equation*}
 D F (0)
 = 2m \fint_{\Sset^{m}} f (y) \otimes y \diff y
 = 2m \fint_{\Sset^{m}} \bigl(f (y) - f (x)\bigr)  \otimes y \diff y,
\end{equation*}
since \(\int_{\Sset^{m}} y \diff y = 0\),
and thus 
\begin{equation*}
    \abs{D F (0)}_{\Hset^{m + 1}} 
  = 
    \frac{1}{2} 
    \abs{D F (0)} 
  \le 
    m
    \osc_{\Sset^m} f 
\eofs.
\end{equation*}

For the assertion \eqref{itExtensionHyperbolicMeasure}, we observe that for every \(x \in \Sset^m\) and \(r \in [0, 1)\), we have by \eqref{eq_def_hyperharmonic}
\begin{equation*}
    \dist \bigl(F (r x), f (\Sset^n) \bigr)
  \le 
    \abs{F (r x) - f (x)}
  \le 
    (1 - r^2)^m 
    \fint_{\Sset^m} 
      \frac
        {\abs{f (y) - f (x)}}
        {\abs{y- rx}^{2 m}} 
      \diff y
  \eofs .
\end{equation*}
We deduce therefrom that for every \(\varepsilon > 0\)
\begin{equation*}
    \dist \bigl(F (r x), f (\Sset^n) \bigr)
  \le 
    \varepsilon 
    + 
    (1 - r^2)^m
    \fint_{\Sset^m}
      \frac{\bigl(\abs{f (y) - f (x)} - \varepsilon\bigr)_+}
            {\abs{y - rx}^{2 m}} 
    \diff y
  \eofs .
\end{equation*}
We next observe that, by the triangle inequality, for every \(x, y \in \Sset^m\) and \(r \in [0, 1)\), we have 
\begin{equation*}
    \abs{y - x} 
  \le 
    \abs{y - r x} + \abs{rx - x}
 = 
    \abs{y - r x}  + \abs{y} - \abs{r x}
 \le 
    2 \abs{y - r x}
  \eofs .
\end{equation*}
Therefore, we have 
\begin{equation}
\label{ineq_dist_f_image_epsilon}
    \dist \bigl(F (r x), f (\Sset^n) \bigr)
  \le 
    \varepsilon 
    + 
    \frac{(1 - r^2)^m \, 4^m}
          {\abs{\Sset^m}}
    \int_{\Sset^m}
      \frac
        {\bigl(\abs{f (y) -  f (x)} - \varepsilon\bigr)_+}
        {\abs{y - x}^{2 m}} 
    \diff y
  \eofs .
\end{equation}
We define the set 
\begin{equation*}
    A_\delta 
  \defeq 
    \bigl\{ 
      z \in \Bset^{m + 1} 
    \st 
      \dist \bigl(F (z), f (\Sset^m)\bigr) \ge \delta 
    \bigr\}\eofs .
\end{equation*}
For each \(x \in \Sset^m\), we set 
\begin{equation*}
    \rho_{\delta} (x) 
  \defeq 
    \sup 
      \,\{
        r \in [0, 1) 
      \st 
        r x \in A_\delta
      \}
\eofs, 
\end{equation*}
(with the convention that \(\rho_\delta (x) \defeq 0\) if \(r x \not \in A_\delta\) for every \(r \in [0, 1)\))
and, since \(m \ge 1\), we compute that 
\begin{equation*}
\begin{split}
    \mu_{\Hset^{m + 1}} (A_\delta)
  &\le 
    \int_{\Sset^m} 
      \int_0^{\rho_\delta (x)} 
        \frac{2^{m + 1} \, r^m}
          {\bigl(1 - r^2\bigr)^{m + 1}} 
      \diff r 
    \diff x\\
  &\le
    \int_{\Sset^m} 
      \int_0^{\rho_\delta (x)} 
        \frac{2^{m + 1} \, r}{\bigl(1 - r^2\bigr)^{m + 1}} 
      \diff r 
    \diff x
  \le 
    \int_{\Sset^m} 
      \frac{2^{m}}
            {m\, \bigl(1 - \rho_{\delta} (x)^2\bigr)^m} 
    \diff x \eofs.
\end{split}
\end{equation*}
Since \(\rho_{\delta} (x) x \in A_\delta\), we deduce from \eqref{ineq_dist_f_image_epsilon} that 
\begin{equation}
\label{ineq_rho_delta}
    \frac{1}{\bigl(1 - \rho_{\delta} (x)^2\bigr)^m}
  \le 
    \frac{4^m}
          {\abs{\Sset^m}(\delta - \varepsilon)} 
    \int_{\Sset^m}
      \frac{\bigl(\abs{f (y) - f (x)} - \varepsilon\bigr)_+}
            {\abs{y - x}^{2 m}} 
    \diff y
\end{equation}
and we conclude that 
\begin{equation*}
    \mu_{\Hset^{m + 1}} (A_\delta) 
  \le 
    \frac
      {8^m}
      {m \,\abs{\Sset^m}\,(\delta - \varepsilon)} 
    \iint 
      \limits_{\Sset^m \times \Sset^m}
      \frac
        {\bigl(\abs{f (y) - f (x)} - \varepsilon\bigr)_+}
        {\abs{y - x}^{2 m}} 
    \diff y 
    \diff x 
\eofs .
\qedhere
\end{equation*}
\end{proof}

\begin{remark}
The proof of \cref{lemmaExtensionHyperbolic} controls in fact the hyperbolic volume of the star-shaped hull \(A_\delta^{\star, 0}\) of the set \(A_\delta\) with respect to \(0\) defined as the smallest subset which is starshaped with respect to \(0\) and contains \(A_\delta\). 
By invariance under the Möbius group that models the isometries of the hyperbolic space in the Poincaré ball model, the volume of the starshaped hull  \(A_\delta^{\star,x}\) of the set \(A_\delta\) with respect to any \(x \in \Hset^{m + 1}\) is also controlled. 
\end{remark}

\subsection{Ball merging}

Our second tool for proving \cref{theorem_bounded_finite_bubbles} is a construction that merges balls in a covering.

\begin{lemma}[Merging balls on manifold]
\label{lemmaDisjointCovering}
Let \(\manifold{M}\) be a Riemannian manifold. 
Given a positive integer \(\ell \in \Nset_*\), 
points \(a_1, a_2, \dotsc, a_\ell \in \manifold{M}\) and radii \(r_1, r_2, \dotsc, r_\ell \in (0, + \infty)\), there exists
\(\ell' \in \{1, 2, \dotsc, \ell\}\), points \(a_1', \dotsc, a_\ell' \in \manifold{M}\) and radii \(r_1', \dotsc, r_\ell'\) such that 
\begin{align*}
 &\bigcup_{i = 1}^\ell B^{\manifold{M}}_{r_i} (a_i) \subseteq \bigcup_{i = 1}^{\ell'} B^{\manifold{M}}_{r_i'} (a_i')&
&\text{ and }&
 &\sum_{i = 1}^{\ell'} r_i' \le \sum_{i = 1}^\ell r_i
 \eofs,
\end{align*}
and for every \(i, j \in \{1, 2, \dotsc, \ell\}\) such that \(i \ne j\),
\(
  \Bar{B}^{\manifold{M}}_{r_i'} (a_i') \cap \Bar{B}^{\manifold{M}}_{r_j'} (a_j') = \emptyset.
\)
\end{lemma}

\Cref{lemmaDisjointCovering} merges a covering by balls into a covering by disjoint balls, while keeping the sums of the radii under control.
This lemma is classical in the Euclidean case where it was a key tool lower bounds for Ginzburg--Landau energies \citelist{\cite{Sandier_1998}*{proof of theorem 1}\cite{Jerrard_1999}*{lemma 3.1}\cite{Sandier_Serfaty_2007}*{lemma 4.1}}.

\begin{proof}[Proof of \cref{lemmaDisjointCovering}]
We proceed by induction. The lemma holds trivially when \(\ell = 1\). We assume now that \(\ell > 1\) and that the conclusion holds for \(\ell - 1\).

If for every \(i, j \in \{1, 2, \dotsc, \ell\}\), we have \(\Bar{B}^{\manifold{M}}_{r_i} (a_i) \cap \Bar{B}^{\manifold{M}}_{r_j} (a_j) = \emptyset\), the lemma is proved by taking \(\ell'=\ell\), and, for each \(i \in \{1, 2, \dotsc, \ell\}\), \(a_i' = a_i\) and \(r_i' = r_i\).

Otherwise, we can assume without loss of generality that \(\Bar{B}^{\manifold{M}}_{r_{\ell - 1}} (a_{\ell - 1}) \cap \Bar{B}^{\manifold{M}}_{r_\ell} (a_\ell) \ne \emptyset\). 
By the triangle inequality, this implies that \(d_{\manifold{M}} (a_{\ell - 1}, a_{\ell}) \le r_{\ell - 1} + r_{\ell}\). 
Since the distance \(d_{\manifold{M}}\) is a geodesic distance on the manifold \(\manifold{M}\), there exists a point \(\Tilde{a}_{\ell - 1} \in \manifold{M}\) such that 
\begin{align*}
    d_{\manifold{M}} (\Tilde{a}_{\ell - 1}, a_{\ell - 1}) 
  &= 
    \tfrac{1}{2}\bigl(d_{\manifold{M}} (a_{\ell - 1}, a_\ell) + r_{\ell} - r_{\ell - 1}\bigr)\eofs,&
\intertext{and}
    d_{\manifold{M}} (\Tilde{a}_{\ell - 1}, a_{\ell}) 
  &= 
    \tfrac{1}{2}\bigl(d_{\manifold{M}} (a_{\ell - 1}, a_\ell) + r_{\ell-1} - r_{\ell}\bigr)
\eofs.
\end{align*}
We now set \(\Tilde{r}_{\ell - 1} \defeq \frac{1}{2}\bigl(d_{\manifold{M}} (a_{\ell - 1}, a_\ell) + r_{\ell-1} + r_{\ell}\bigr)\). We observe that
\begin{align*}
\Tilde{r}_{\ell - 1} &\le r_{\ell - 1} + r_\ell &
&\text{and}&
B^{\manifold{M}}_{r_{\ell - 1}} (a_{\ell - 1}) \cup 
B^{\manifold{M}}_{r_{\ell}} (a_{\ell}) & \subseteq B^{\manifold{M}}_{\Tilde{r}_{\ell - 1}} (\Tilde{a}_{\ell - 1})\,.
\end{align*}
We set, for \(i \in \{1, 2, \dotsc, \ell - 2\}\), \(\Tilde{a}_i \defeq a_i\) and \(\Tilde{r}_i \defeq r_i\) and \(\Tilde{\ell} \defeq \ell - 1\). We conclude by applying our induction hypothesis to \(\Tilde{\ell}\), \(\Tilde{a}_1, \Tilde{a}_2, \dotsc, \Tilde{a}_{\Tilde{\ell}}\) and  \(\Tilde{r}_1, \Tilde{r}_2, \dotsc, \Tilde{r}_{\Tilde{\ell}}\).
\end{proof}

\begin{remark}
The proof of \cref{lemmaDisjointCovering} relies on the fact that \(\manifold{M}\) is a geodesic metric space to construct the point \(\Tilde{a}_{\ell - 1}\).
If \(\manifold{M}\) is merely a length metric space, the proof gives weaker forms of \cref{lemmaDisjointCovering} where either the inequality on the sum of radius holds up to an arbitrary error \(\varepsilon > 0\) or the open balls, instead of the closed balls, are disjoint. If \(\manifold{M}\) is simply a metric space, than we can still take \(\Tilde{a}_{\ell - 1} \in \Bar{B}_{r_{\ell - 1}} (a_{\ell - 1}) \cap \Bar{B}_{r_\ell} (a_\ell)\) and \(\Tilde{r}_{\ell - 1} = 2 \max \{r_{\ell - 1}, r_\ell\}\) and obtain the conclusion with an additional unbounded \(2^{\ell - 1}\) factor multiplying the sum of radii on the right-hand side.
\end{remark}

We will also rely on a straightforward characterization of the geometry of hyperbolic spheres \cite{Fenchel_1989}*{\S III.5}.

\begin{lemma}
[Description of spheres in the hyperbolic space]
\label{lemma_HyperbolicSphere}
For every \(\rho > 0\) and every \(m \in \Nset\), the hyperbolic sphere \(\partial B^{\Hset^{m + 1}}_\rho (a)\) is isometric to the Euclidean \(\Sset^m_{\sinh \rho}\).
\end{lemma}

\begin{proof}
We work in the Poincaré ball model and assume without loss of generality that \(a = 0 \in \Bset^{m + 1} \simeq \Hset^{m + 1}\).
The hyperbolic ball \(B_\rho^{\Hset^{m + 1}} (0)\) is modelled by the Euclidean ball \(B_{\tanh \frac{\rho}{2}} (0)\). 
The Poincaré metric at every point \(x \in \partial B_{\tanh \frac{\rho}{2}} (0)\)  on this sphere is 
\(\sqrt{g_x (v)} = 2 \abs{v}/(1 - (\tanh \frac{\rho}{2})^2) = 2 (\cosh \frac{\rho}{2})^2 \abs{v}\) which means that the radius of the isometric Euclidean sphere is \(2 (\cosh \frac{\rho}{2})^2 \tanh \frac{\rho}{2} = \sinh \rho\).
\end{proof}

\subsection{Proof of the theorem}
\Cref{theorem_bounded_finite_bubbles} will follow from the following slightly stronger statement, involving a truncated fractional integral.

\begin{theorem}
[Estimate on free homotopy decomposition by a truncated fractional energy]
\label{theorem_bounded_finite_bubbles_linear}
If \(\varepsilon > 0\) is small enough, there is a constant \(C > 0\) such that for every \(\lambda > 0\), there exists a finite set \(\mapsset{F}^\lambda \subset \mapsset{C} (\Sset^m, \manifold{N})\) such that any map \(f \in \mapsset{C} (\Sset^m, \manifold{N})\) satisfying
\begin{equation*}
    \iint\limits_{\Sset^m \times \Sset^m}
      \frac{\bigl(d_{\manifold{N}} (f (y), f (x)) - \varepsilon\bigr)_+}{\abs{y - x}^{2 m}} 
      \diff x 
      \diff y 
  \le 
    \lambda
\eofs,
\end{equation*}
fas a free homotopy decomposition into \(f_1, \dotsc, f_k \in \mapsset{F}^\lambda\) with \(k \le C \lambda\).
\end{theorem}

\begin{proof}[Proof of \cref{theorem_bounded_finite_bubbles_linear}]
\resetconstant
We apply \cref{lemmaExtensionHyperbolic} to \(f\). 
We define for each \(\delta > 0\) the sets 
\(\manifold{N}_\delta \defeq \{ y \in \Rset^\nu \st \dist (y, \manifold{N}) < \delta \}\) 
and 
\begin{equation*}
    A_\delta 
  \defeq 
    F^{-1} (\Rset^\nu \setminus \manifold{N}_\delta) 
  = 
    \bigl\{ 
      x \in \Hset^{m + 1} 
    \st 
      \dist (F (x), \manifold{N}) \ge \delta \bigr\} 
\eofs.
\end{equation*}
Since \(\manifold{N}\) is a smooth submanifold of \(\Rset^\nu\), 
there exists \(\delta_* > 0\) and a Lipschitz-continuous retraction \(\Pi : \manifold{N}_{\delta_*} \to \manifold{N}\), that is, one has for every \(y \in \manifold{N}_{\delta_*}\), \(\Pi (y) \in \manifold{N}\) and for every \(y \in \manifold{N}\), \(\Pi (y) = y\).

By the estimate \eqref{itExtensionHyperbolicLipschitz} in \cref{lemmaExtensionHyperbolic},
we observe that if \(a \in A_{\delta_*}\), then for every \(x \in \Hset^{m + 1}\) we have
\begin{equation*}
    F (x) 
  \ge 
      \delta_* 
    - 
      m 
      \, 
      \diam (\manifold{N})
      \,
      d_{\Hset^{m + 1}} (x, a)
  \eofs,
\end{equation*}
If we take \(\rho \defeq \frac{\delta_*}{2 m \diam (\manifold{N})}\), we have 
\begin{equation}
\label{eqInclusionDelta2}
    \Bar{B}^{\Hset^{m + 1}}_{\rho} (a) 
  \subseteq 
    A_{\delta_*/2}
\eofs.
\end{equation}

We consider now a maximal set of points \(A \subseteq A_{\delta_*}\) such that if \(a, b \in A\) and \(a \ne b\) then \(d_{\Hset^{m + 1}} (a, b) \ge 2 \rho\). By construction, we have 
\begin{equation*}
    A_{\delta_*} 
  \subseteq 
    \bigcup_{a \in A} 
      B^{\Hset^{m + 1}}_{2 \rho} (a)
\eofs.
\end{equation*}
On the other hand the balls \((B^{\Hset^{m + 1}}_\rho (a))_{a \in A}\) are disjoint and thus by \eqref{eqInclusionDelta2}, we have 
\begin{equation*}
    \sum_{a \in A} 
      \mu_{\Hset^{m + 1}} \bigl(B^{\Hset^{m + 1}}_\rho (a)\bigr) 
  \le 
    \mu_{\Hset^{m + 1}} \bigl(A_{\delta_*/2}\bigr)
\eofs.
\end{equation*}
By the invariance of the volume of balls in the hyperbolic space, we deduce that 
\begin{equation*}
    \# A 
  \le 
    \Cl{cst_Phi4Shae9i}
    \iint \limits_{\Sset^m \times \Sset^m}  
      \frac
        {\bigl(\abs{f (y) - f (x)} - \varepsilon\bigr)_+}
        {\abs{y - x}^{2 m}}
    \diff y \diff x
  \le
    \Cr{cst_Phi4Shae9i} \lambda
 \eofs.
\end{equation*}
By \cref{lemmaDisjointCovering}, we obtain the existence of \(k \in \{1, \dotsc, \# A\}\), \(a_1, \dotsc, a_k \in \Hset^{m + 1}\) and \(\rho_1,  \dotsc, \rho_k \in [2\rho, 2 \Cr{cst_Phi4Shae9i} \lambda \rho]\) such that if \(i, j \in \{1,  \dotsc, k\}\) and \(i \ne j\), 
then \(\Bar{B}^{\Hset^{m + 1}}_{\rho_i} (a_i) \cap \Bar{B}^{\Hset^{m + 1}}_{\rho_j} (a_j) = \emptyset\) and 
\begin{equation*}
    A_{\delta_*}  
  \subseteq 
    \bigcup_{i =1}^k \Bar{B}^{\Hset^{m + 1}}_{\rho_i} (a_i)
\eofs.
\end{equation*}

We conclude by defining the set 
\begin{equation*}
    U 
  \defeq 
    \Hset^{m + 1} \setminus \bigcup_{i =1}^k B^{\Hset^{m + 1}}_{\rho_i} (a_i)
\end{equation*}
and the map 
\begin{equation*}
    \Tilde{F} 
  \defeq
    \Pi \compose F\vert_{U} : U \to N
\eofs.
\end{equation*}
We observe that \(\partial U \cap \Bset^{m + 1} = \bigcup_{i =1}^k \partial B^{\Hset^{m + 1}}_{\rho_i} (a_i)\) and that for every \(i \in \{1, \dotsc, k\}\), the set \(\partial B^{\Hset^{m + 1}}_{\rho_i} (a_i)\) is isometric by \cref{lemma_HyperbolicSphere} to a Euclidean \(m\)--dimensional sphere of radius \(\sinh \rho_i\). 
This implies then that for every \(i \in \{1, \dotsc, k\}\), the map \(\Tilde{F} \vert_{\partial B^{\Hset^{m + 1}}_{\rho_i} (a_i)}\) is homotopic on \(\partial B^{\Hset^{m + 1}}_{\rho_i} (a_i) \simeq \Sset^m\) to some map \(g_i : \Sset^m \to \manifold{N}\) whose Lipschitz constant is controlled by \(\Cl{cst_Ooqu5shoh2} \sinh (\Cl{cst_ieSu6Quue7} \lambda)\). 
By the Ascoli compactness theorem, there exists a finite set of maps \(\mapsset{F}^\lambda \subset \mapsset{C} (\Sset^m, \Sset^m)\) such that any map from \(\Sset^m\) to \(\manifold{N}\) whose Lipschitz constant does not exceed \(\Cr{cst_Ooqu5shoh2} \sinh (\Cr{cst_ieSu6Quue7} \lambda)\) is homotopic to some map in \(\mapsset{F}^\lambda\). 
In particular, for every \(i \in \{1, \dotsc, k\}\), there exists a map \(f_i \in \mapsset{F}^\lambda\) which is homotopic to \(g_i\) on \(\Sset^m\) and thus to \(\Tilde{F} \vert_{\partial B^{\Hset^{m + 1}}_{\rho_i} (a_i)}\) on \(\Sset^m \simeq \partial B^{\Hset^{m + 1}}_{\rho_i} (a_i)\).

We consider now a ball \(B_{2\rho}^{\Hset^{m + 1}} (a_*) \subset \Tilde{U} \) and a map \(\Breve{F} \in \mapsset{C} (U, \manifold{N})\) such that \(\Breve{F} = \Tilde{F}\) in \(U \setminus B_{2 \rho}^{\Hset^{m + 1}} (a_*)\) and \(\Breve{F}\) is constant on \(B^{\Hset^{m + 1}}_{\rho} (a_*)\).
We now consider maps \(\Phi_i : \Bar{\Bset}^{m + 1} \to U \setminus B_{\rho} (a_*)\) such that for every \(i \in \{1, \dotsc, k\}\), \(\Phi_i\) is injective, \(B^{\Hset^{m + 1}}_{\rho_i} (a_i) \subset \Phi_i (\Bset^{m + 1})\) and \(\Phi_i (\Bar{\Bset^{m + 1}}) \cap \partial B^{\Hset^{m + 1}}_{\rho} (a_*)\) is a nondegenerate geodesic ball in \(\partial B^{\Hset^{m + 1}}_{\rho} (a_*)\).
We define \(\Breve{U} = U \setminus (B_{\rho} (a_*) \cup \bigcup_{i = 1}^k \Phi_i (\Bset^{m + 1})) \subset \Hset^{m + 1}\), and we observe that \(\partial \Breve{U} \cap \Bset^{m + 1}\) is homeorphic to \(\Sset^{m + 1}\) and that \(\Breve{F} \vert_{\partial \Breve{U} \cap \Bset^{m + 1}}\) has a free homotopy decomposition into \(f_1, \dotsc, f_k\), and hence by homotopy invariance, \(f\) also has a free homotopy decomposition into \(f_1, \dotsc, f_k\).
\end{proof}

We deduce now \cref{theorem_bounded_finite_bubbles}
from \cref{theorem_bounded_finite_bubbles_linear}.

\begin{proof}[Proof of \cref{theorem_bounded_finite_bubbles}]
We note that, since the map \(f : \Sset^m \to \manifold{N}\) is bounded, we have
\begin{equation}
\label{ineq_diam_N}
    \iint\limits_{\Sset^m \times \Sset^m}
      \frac{\bigl(\abs{f (y) - f (x)} - \varepsilon\bigr)_+}{\abs{y - x}^{2 m}} 
    \diff y
    \diff x
  \le  
    \diam (\manifold{N})
    \hspace{-1em}
    \iint
      \limits_{
        \substack{
          (x, y) \in \Sset^m \times \Sset^m \\                 
          \abs{f (y) - f (x)} > \varepsilon}} 
          \hspace{-.7em}
        \frac{1}{\abs{y - x}^{2 m}} 
      \diff y
      \diff x
      \eofs,
\end{equation}
and the conclusion then follows from \cref{theorem_bounded_finite_bubbles_linear}.
\end{proof}

We will observe in the sequel that when \(m \ge 2\), an estimate of the form \eqref{ineq_diam_N} holds without any boundedness assumption on the map \(f\) and with a constant of the order of \(\varepsilon\) (see \cref{proposition_p_q_scaling_Omega} below).

\begin{proof}%
[Proof of \cref{theorem_bounded_homotopy_classes}]
This follows from \cref{theorem_bounded_finite_bubbles} and \cref{proposition_finite_homotopy_classes}.
\end{proof}

\section{Scaling and comparison of truncated fractional energies}

In this section we improve the estimate of \cref{theorem_bounded_finite_bubbles} into an estimate that scales optimally with respect to \(\varepsilon\) as \(\varepsilon \to 0\). 
Our results are the counterpart of \familyname{Nguyên} Hoài-Minh's estimates on the topological degree \cite{Nguyen_2017}, but are obtained with a different strategy.

\subsection{Scaling of truncated fractional energies}

In order to improve the estimate of \cref{theorem_bounded_finite_bubbles},
we first study how truncated fractional integral scale with varying values of the truncation in the next proposition.

\begin{proposition}
[Scaling of truncated fractional energies on a convex set]
\label{lemma_fract_gap_integral_growth}
For every \(p \in [0, +\infty)\) and every \(m \in \Nset\), 
there exists a constant \(C > 0\) such that 
for every convex set \(\Omega \subset \Rset^m\)
and for every map \(f : \Omega \to \manifold{N}\),
if \(\delta < \varepsilon\) one has
\begin{multline*}
    \iint
    \limits_{
      \substack{
        (x, y) \in \Omega \times \Omega\\
        d_{\manifold{N}} (f (y), f (x)) \ge \varepsilon}
      }
      \hspace{-1.5em}
      \frac
        {\bigl(d_{\manifold{N}} (f (y), f (x)) - \varepsilon\bigr)^p}
        {\abs{y - x}^{2 m}}
      \diff y
      \diff x\\[-2em]
  \le 
    C 
    \biggl(
    \frac
      {\delta}
      {\varepsilon}
    \biggr)^{m - 1 - (p - 1)_+}
    \hspace{-2em}
    \iint
    \limits_{
      \substack{
        (x, y) \in \Omega \times \Omega\\
        d_{\manifold{N}} (f (y), f (x)) \ge \delta}
      }
      \hspace{-1.5em}
      \frac
        {\bigl(d_{\manifold{N}} (f (y), f (x)) - \delta\bigr)^p}
        {\abs{y - x}^{2 m}}
      \diff y
      \diff x
\eofs.
\end{multline*}
\end{proposition}

When either \(1 \le p < m\), or \(p \le 1\) and \(m > 2\), then \(m - 1 > (p - 1)_+\) and the estimate of \cref{lemma_fract_gap_integral_growth} improves the straightforward monotonicity estimate: if \(\delta \le \varepsilon\), then 
\begin{equation*}
    \iint
      \limits_{
      \substack{
        (x, y) \in \Omega \times \Omega\\
        d_{\manifold{N}} (f (y), f (x)) \ge \varepsilon}
      }
      \hspace{-1.5em}
      \frac
        {\bigl(d_{\manifold{N}} (f (y), f (x)) - \varepsilon\bigr)^p}
        {\abs{y - x}^{2 m}}
      \diff y
      \diff x
  \le 
    \iint\limits_{
      \substack{
        (x, y) \in \Omega \times \Omega\\
        d_{\manifold{N}} (f (y), f (x)) \ge \delta}
      }
      \hspace{-1.5em}
      \frac
        {\bigl(d_{\manifold{N}} (f (y), f (x)) - \delta\bigr)^p}
        {\abs{y - x}^{2 m}}
      \diff y
      \diff x
\eofs.
\end{equation*}

If the set \(\Omega \subset \Rset^m\) is bounded and if the map \(f : \Omega \to \Rset^m\) is the identity, one has
\begin{equation}
\label{eq_scaling_identity}
\begin{split}
    \iint
      \limits_{
      \substack{
        (x, y) \in \Omega \times \Omega\\
        d_{\manifold{N}} (f (y), f (x)) \ge \varepsilon}
      } 
      \hspace{-1.5em}
      \frac
        {\bigl(d_{\manifold{N}} (f (y), f (x)) - \varepsilon\bigr)^p}
        {\abs{y - x}^{2 m}} 
      \diff y
      \diff x
  & = 
    \iint
      \limits_{
      \substack{
        (x, y) \in \Omega \times \Omega\\
        \abs{y - x} \ge \varepsilon}
      } 
      \frac
        {(\abs{y - x} - \varepsilon)^p}
        {\abs{y - x}^{2 m}} 
      \diff y
      \diff x
        \\
 &
  \simeq 
    \int_\varepsilon^1
      \frac{(r - \varepsilon)^p}{r^{m + 1}} 
    \diff r
  = 
    \frac
      {1}
      {\varepsilon^{m - p}}
    \int_0^{\frac{1}{\varepsilon} - 1}
    \hspace{-1em}
      \frac{t^p}{(t + 1)^{m + 1}}
    \diff t
    \\
  &
  \simeq
  \begin{cases}
    \frac{1}{\varepsilon^{m - p}} & \text{if \(p < m\)},\\
    \ln \frac{1}{\varepsilon} & \text{if \(p = m\)},\\
    1 & \text{if \(p > m\)},
  \end{cases}
\end{split}
\end{equation}
as \(\varepsilon \to 0\),
by the change of variables \(r = \varepsilon (t + 1)\).
This computation means that the scaling estimate of \cref{lemma_fract_gap_integral_growth} is optimal when \(1 \le p < m\). 
We do not know whether the estimate can be improved when \(0 \le p < 1\) 
(see \cref{problem_scaling} below). 
The estimate will already appear to be strong enough to obtain some comparison between truncated fractional integrals of different exponents in \cref{proposition_p_q_scaling_Omega} below.

\begin{proof}%
[Proof of \cref{lemma_fract_gap_integral_growth}]
By the triangle inequality, we have
\begin{equation*}
\begin{split}
    \iint
      \limits_{
      \substack{
        (x, y) \in \Omega \times \Omega\\
        d_{\manifold{N}} (f (y), f (x)) \ge \varepsilon}}
      \hspace{-1.5em}
      &
      \frac
        {\bigl(d_{\manifold{N}} (f (y), f (x)) - \varepsilon\bigr)^p}
        {\abs{y - x}^{2 m}}
      \diff y
      \diff x
  \\
  &
  \le
    2^{(p - 1)_+}
      \hspace{-1.5em}
      \iint
        \limits_{
          \substack{
            (x, y) \in \Omega \times \Omega\\
            d_{\manifold{N}} (f (y), f (\frac{x + y}{2})) \ge \frac{\varepsilon}{2}
            }
          }
        \hspace{-1.5em}
        \frac
          {\bigl(d_{\manifold{N}} (f (y), f (\frac{x + y}{2})) - \frac{\varepsilon}{2}\bigr)^p}
          {\abs{y - x}^{2 m}}
        \diff y
        \diff x\\
    &
    \hspace{5em}
    +
     2^{(p - 1)_+}
        \hspace{-1.5em}
      \iint
        \limits_{
          \substack{
            (x, y) \in \Omega \times \Omega\\
            d_{\manifold{N}} (f (\frac{x + y}{2}), f (x)) \ge \frac{\varepsilon}{2}
            }        
          }
    \hspace{-1.5em}
        \frac
          {\bigl(d_{\manifold{N}} (f (\frac{x + y}{2}), f (x)) - \frac{\varepsilon}{2}\bigr)^p}
          {\abs{y - x}^{2 m}}
        \diff y
        \diff x
\eofs ,
\end{split}
\end{equation*}
and thus  by symmetry under exchange of \(x\) and \(y\)
\begin{equation}
\label{ineq_fract_double_gap_tgl}
\begin{split}
    \iint
      \limits_{
      \substack{
        (x, y) \in \Omega \times \Omega\\
        d_{\manifold{N}} (f (y), f (x)) \ge \varepsilon}}
      \hspace{-1.5em}
      &
      \frac
        {\bigl(d_{\manifold{N}} (f (y), f (x)) - \varepsilon\bigr)^p}
        {\abs{y - x}^{2 m}}
      \diff y
      \diff x
  \\[-.5em]
   &
  = 
    2^{1 + (p - 1)_+}      
    \hspace{-2em}
    \iint
      \limits_{
        \substack{
          (x, y) \in \Omega \times \Omega\\
          d_{\manifold{N}} (f (\frac{x + y}{2}), f (x)) \ge \frac{\varepsilon}{2}
          }        
        }
    \hspace{-2em}
      \frac
        {\bigl(d_{\manifold{N}} (f (\frac{x + y}{2}), f (x)) - \frac{\varepsilon}{2}\bigr)^p}
        {\abs{y - x}^{2 m}}
      \diff y
      \diff x
      \eofs .
\end{split}
\end{equation}
We apply now the change of variable \(y = 2 z - x\) and we obtain
\begin{equation}
\label{ineq_fract_double_gap_cov}
\begin{split}
     \iint
      \limits_{
          \substack{
            (x, y) \in \Omega \times \Omega\\
            d_{\manifold{N}} (f (\frac{x + y}{2}), f (x)) \ge \frac{\varepsilon}{2}
            }
          }
    \hspace{-2em}
    \frac
        {
          \bigl(
                d_{\manifold{N}} (f (\frac{x + y}{2}), f (x)) 
              -
                \frac{\varepsilon}{2}
          \bigr)^p
        }
        {\abs{y - x}^{2 m}}
      \diff y
      \diff x
  &
  =
    \frac
      {1}
      {2^m}
    \int_{\Omega}
      \Biggl(
      \int_{\Sigma_x}
        \frac
          {
            \bigl(
                d_{\manifold{N}} (f (z), f (x))
              - 
                \frac{\varepsilon}{2}
            \bigr)^p
          }
          {\abs{z - x}^{2 m}}  
        \diff z
      \Biggr)
      \diff x
  \\
  &
  \le
    \frac
      {1}
      {2^m}
    \hspace{-1.5em}
    \iint
      \limits_{
      \substack{
        (x, y) \in \Omega \times \Omega\\
        d_{\manifold{N}} (f (y), f (x)) \ge \frac{\varepsilon}{2}
        }
      }
    \hspace{-1.5em}
        \frac
          {
            \bigl(
              d_{\manifold{N}} (f (y), f (x)) - \frac{\varepsilon}{2}
            \bigr)^p
          }
          {\abs{y - x}^{2 m}}  
      \diff y      
      \diff x  
  \eofs ,
\end{split}
\end{equation}
where for every \(x \in \Omega\), we have defined the set 
\begin{equation*}
    \Sigma_x 
  \defeq 
    \bigl\{
      z \in \Omega 
    \st 
      2 z - x \in \Omega 
      \text{ and }
      d_{\manifold{N}} (f (z), f (x)) \ge \tfrac{\varepsilon}{2}
    \bigr\}
    \eofs.
\end{equation*}
By combining the inequalities \eqref{ineq_fract_double_gap_tgl} and \eqref{ineq_fract_double_gap_cov}, we deduce that for every \(\varepsilon > 0\),
\begin{multline}
\label{ineq_fract_double_gap}
    \iint
      \limits_{
        \substack{
          (x, y) \in \Omega \times \Omega\\
          d_{\manifold{N}} (f (y), f (x)) \ge \varepsilon
          }        
        }
      \hspace{-1.5em}
      \frac
        {\bigl(d_{\manifold{N}} (f (y), f (x)) - \varepsilon\bigr)^p}
        {\abs{y - x}^{2 m}}
      \diff y
      \diff x
      \\[-2em]
  \le
    2^{(p - 1)_+ - (m - 1)}
    \hspace{-1.5em}
    \iint
      \limits_{
        \substack{
          (x, y) \in \Omega \times \Omega\\
          d_{\manifold{N}} (f (y), f (x)) \ge \frac{\varepsilon}{2}
          }        
        }
      \hspace{-1.5em}
      \frac
        {
          \bigl(
            d_{\manifold{N}} (f (y), f (x)) - \frac{\varepsilon}{2}
          \bigr)^p
        }
        {\abs{y - x}^{2 m}}
      \diff y
      \diff x
\eofs.    
\end{multline}
By iterating the estimate \eqref{ineq_fract_double_gap}, 
we deduce that for every nonnegative integer \(\ell \in \Nset\),
\begin{multline*}
    \iint
      \limits_{
        \substack{
          (x, y) \in \Omega \times \Omega\\
          d_{\manifold{N}} (f (y), f (x)) \ge \varepsilon
          }        
        }
      \hspace{-1.5em}
      \frac
        {\bigl(d_{\manifold{N}} (f (y), f (x)) - \varepsilon\bigr)^p}
        {\abs{y - x}^{2 m}}
      \diff x
      \diff y\\[-2em]
  \le
    2^{\ell ((p- 1)_+ - (m - 1))}
    \hspace{-1.5em}
    \iint
      \limits_{
        \substack{
          (x, y) \in \Omega \times \Omega\\
          d_{\manifold{N}} (f (y), f (x)) \ge \frac{\varepsilon}{2^\ell}
          }        
        }
      \hspace{-1.5em}
      \frac
        {(d_{\manifold{N}} (f (y), f (x)) - \frac{\varepsilon}{2^\ell} )^p}
        {\abs{y - x}^{2 m}}
      \diff y
      \diff x
\eofs.    
\end{multline*}
If \(\delta \in (0, \varepsilon)\), we let \(\ell \in \Nset\) 
be defined by the condition \(2^{-(\ell + 1)} \varepsilon \le \delta < 2^{-\ell} \varepsilon\) and
we conclude that 
\begin{equation*}
 \begin{split}
    &\iint
      \limits_{
        \substack{
          (x, y) \in \Omega \times \Omega\\
          d_{\manifold{N}} (f (y), f (x)) \ge \varepsilon
          }        
        }
      \hspace{-1.5em}
      \frac
        {\bigl(d_{\manifold{N}} (f (y), f (x)) - \varepsilon\bigr)^p}
        {\abs{y - x}^{2 m}}
      \diff y
      \diff x
      \\[-2em]
  &\hspace{10em} \le 
    \biggl(
      \frac
        {2 \delta}
        {\varepsilon}
    \biggr)^{m - 1 - (p - 1)_+}
    \hspace{-1.5em}
    \iint
      \limits_{
        \substack{
          (x, y) \in \Omega \times \Omega\\
          d_{\manifold{N}} (f (y), f (x)) \ge \delta
          }        
        }
      \hspace{-1.5em}
      \frac
        {\bigl(d_{\manifold{N}} (f (y), f (x)) - \delta\bigr)^p}
        {\abs{y - x}^{2 m}}
      \diff y
      \diff x 
  \;
  \eofs.
  \qedhere
\end{split}
\end{equation*}    
\end{proof}

In order to improve the statement of \cref{theorem_bounded_finite_bubbles},
we will derive the counterpart of \cref{lemma_fract_gap_integral_growth} for spheres.

\begin{proposition}
[Scaling of truncated fractional energies on a sphere]%
\label{proposition_fract_gap_integral_growth_sphere}
For every \(p \in [0, +\infty)\) and every \(m \in \Nset\), 
there exists a constant \(C > 0\) such that 
for every map \(f : \Sset^m \to \manifold{N}\),
if \(\delta < \varepsilon\) one has
\begin{multline*}
    \iint
      \limits_{
        \substack{
          (x, y) \in \Sset^m \times \Sset^m\\
          d_{\manifold{N}} (f (y), f (x)) \ge \varepsilon
          }        
        }    
      \hspace{-1.7em}
      \frac
        {\bigl(d_{\manifold{N}} (f (y), f (x)) - \varepsilon\bigr)^p}
        {\abs{y - x}^{2 m}}
      \diff y
      \diff x\\[-2em]
  \le 
    C 
    \biggl(
    \frac
      {\delta}
      {\varepsilon}
    \biggr)^{m - 1 - (p - 1)_+}
    \hspace{-1.7em}
    \iint
      \limits_{
        \substack{
          (x, y) \in \Sset^m \times \Sset^m\\
          d_{\manifold{N}} (f (y), f (x)) \ge \delta
          }        
        }
      \hspace{-1.7em}
      \frac
        {\bigl(d_{\manifold{N}} (f (y), f (x)) - \delta\bigr)^p}
        {\abs{y - x}^{2 m}}
      \diff y
      \diff x
\eofs.
\end{multline*}
\end{proposition}

The proof of \cref{proposition_fract_gap_integral_growth_sphere} will rely on its counterpart on a convex set of the Euclidean space \cref{lemma_fract_gap_integral_growth} and on a suitable covering of the sphere by spherical caps.

\begin{lemma}
[Covering pairs of points of the sphere]
\label{lemma_pair_covering}
Let \(m \in \Nset_*\). 
If \(a_0, \dotsc, a_{m + 1} \in \Sset^{m} \subset \Rset^{m + 1}\) are the vertices of a regular simplex, 
then for every \(x, y \in \Sset^m\) there exists \(i \in \{0, \dotsc, m + 1\}\) such that 
\begin{align*}
    a_i \cdot x 
  &
  \ge 
    -\frac{1}{\sqrt{m + 1}}
&
&
\text{ and }
&
    a_i \cdot y 
  &
  \ge 
    -\frac{1}{\sqrt{m + 1}}
\eofs.
\end{align*}
\end{lemma}

\begin{proof}
Let 
\begin{equation*}
    I 
  = 
    \bigl\{
      i \in \{0, \dotsc, m + 1\}
    \st
      a_i \cdot x_i < -\tfrac{1}{\sqrt{m + 1}} 
    \bigr\}
\eofs.
\end{equation*}
We claim that \(\# I < \frac{m}{2} + 1\).
We assume without loss of generality that \(\# I > 0\). We then have
\begin{equation}
\label{eq_Eejau2oo3t}
   \sum_{i \in I}
      a_i \cdot x
   <
     - \frac{\# I}{\sqrt{m + 1}} 
\end{equation}
Since for every \(i, j \in \{0, 1, \dotsc, m + 1\}\), \(a_i \cdot a_j = \frac{m + 2}{m + 1} \delta_{ij} - \frac{1}{m + 1}\), where \(\delta_{ij}\) denotes Kronecker's delta, we have
\begin{equation}
\label{eq_aingo6haiK}
  \biggabs{\sum_{i \in I} a_i}^2
  =
    \sum_{i \in I} 
    \sum_{j \in I}
      a_i \cdot a_j
  = 
  \frac
    { \# I (m + 2 - \# I)}
    {m + 1}
\eofs.
\end{equation}
Hence, by the Cauchy--Schwarz inequality we have, since \(\abs{x} = 1\), by \eqref{eq_Eejau2oo3t} and \eqref{eq_aingo6haiK} 
\begin{equation*}
 \frac{(\# I)^2}{m + 1}
 < 
 \biggl(
  \sum_{i \in I}
      a_i \cdot x
  \biggr)^2
 \le
    \biggabs{\sum_{i \in I} a_i}^2
    \,
    \abs{x}^2
 =
 \frac
    { \# I (m + 2 - \# I)}
    {m + 1}
\eofs,
\end{equation*}
so that 
\(
 \# I < m + 2 - \# I
\),
and therefore \(\# I < \frac{m}{2} + 1\).

If we set similarly
\begin{equation*}
    J
  =
    \bigl\{
      i \in \{0, \dotsc, m + 1\}
    \st
      a_i \cdot y < -\tfrac{1}{\sqrt{m + 1}} 
    \bigr\}
\eofs,    
\end{equation*}
we obtain that 
\(\# I + \# J < m + 2\). 
Hence, there exists \(i \in \{0, 1, \dotsc, m + 1\}\setminus (I \cup J)\) and thus 
the conclusion holds by definition of the sets \(I\) and \(J\).
\end{proof}

\begin{proof}[Proof of \cref{proposition_fract_gap_integral_growth_sphere}]
\resetconstant
Let \(a_0, \dotsc, a_{m + 1} \in \Sset^m \subset \Rset^{m + 1}\) be the vertices of an equilateral simplex and define for each \(i \in \{0, 1, \dotsc, m + 1\}\), the spherical cap
\begin{equation*}
    A_i
  \defeq
    \bigl\{ 
      x \in \Sset^m
    \st
        a_i \cdot x 
      \ge 
        -\tfrac{1}{\sqrt{m + 1}}
    \bigr\}
  \;
  .
\end{equation*}
In view of \cref{lemma_pair_covering},
we have 
\(\Sset^m \times \Sset^m = \bigcup_{i = 0}^{m + 1} A_i \times A_i\), and therefore
\begin{equation}
\label{ineq_Sm_Ai}
    \iint
      \limits_{
        \substack{
          (x, y) \in \Sset^m \times \Sset^m\\
          d_{\manifold{N}} (f (y), f (x)) \ge \varepsilon
          }}   
      \hspace{-1.7em}     
      \frac
        {\bigl(d_{\manifold{N}} (f (y), f (x)) - \varepsilon\bigr)^p}
        {\abs{y - x}^{2 m}}
      \diff y
      \diff x
  \le
    \sum_{i = 0}^{m + 1}
    \iint
      \limits_{
        \substack{
          (x, y) \in A_i \times A_i\\
          d_{\manifold{N}} (f (y), f (x)) \ge \varepsilon
          }}
        \hspace{-1.7em}
        \frac
          {\bigl(d_{\manifold{N}} (f (y), f (x)) - \varepsilon\bigr)^p}
          {\abs{y - x}^{2 m}}
        \diff y
        \diff x
  \eofs .
\end{equation}
Since for every \(i \in \{0, \dotsc, m + 1\}\), the spherical cap \(A_i\) is diffeomorphic to a ball of \(\Rset^m\), we have in view of \cref{lemma_fract_gap_integral_growth},
\begin{equation}
\label{ineq_fract_gap_integral_growth_Ai_Sm}
\begin{split}
    \iint
      \limits_{
        \substack{
          (x, y) \in A_i \times A_i\\
          d_{\manifold{N}} (f (y), f (x)) \ge \varepsilon
          }}
      \hspace{-1.7em}
      &\frac
        {\bigl(d_{\manifold{N}} (f (y), f (x)) - \varepsilon\bigr)^p}
        {\abs{y - x}^{2 m}}
      \diff y
      \diff x\\
  &
  \le
    \Cl{cst_xuaxr}
    \biggl(
    \frac
      {\delta}
      {\varepsilon}
    \biggr)^{m - 1 - (p - 1)_+}
      \hspace{-2em}
        \iint
      \limits_{
        \substack{
          (x, y) \in A_i \times A_i\\
          d_{\manifold{N}} (f (y), f (x)) \ge \delta
          }}
      \hspace{-1.7em}
      \frac
        {\bigl(d_{\manifold{N}} (f (y), f (x)) - \delta\bigr)^p}
        {\abs{y - x}^{2 m}}
      \diff y
      \diff x
  \\
  &
  \le 
    \Cr{cst_xuaxr}
        \biggl(
    \frac
      {\delta}
      {\varepsilon}
    \biggr)^{m - 1 - (p - 1)_+} 
      \hspace{-2em}
    \iint
      \limits_{
        \substack{
          (x, y) \in \Sset^m \times \Sset^m\\
          d_{\manifold{N}} (f (y), f (x)) \ge \delta
          }}
      \hspace{-1.7em}
      \frac
        {\bigl(d_{\manifold{N}} (f (y), f (x)) - \delta\bigr)^p}
        {\abs{y - x}^{2 m}}
      \diff y
      \diff x
  \;
  .
\end{split}
\end{equation}
We conclude by combining \eqref{ineq_Sm_Ai} and 
\eqref{ineq_fract_gap_integral_growth_Ai_Sm} with \(C = (m + 2)\,\Cr{cst_xuaxr}\).
\end{proof}

\begin{theorem}%
[Free homotopy decompositions controlled by a scaled truncated Sobolev energy]%
\label{theorem_bounded_finite_bubbles_linear_scaling}%
Let \(m\in \Nset_*\) and \(\manifold{N}\) be a compact Riemannian manifold. There are constants \(\varepsilon_0 > 0\) and \(C > 0\) such that for every \(\lambda > 0\), there exists a finite set \(\mapsset{F}^\lambda \subset \mapsset{C} (\Sset^m, \manifold{N})\) such that if \(0 < \varepsilon < \varepsilon_0 \), any map \(f \in \mapsset{C} (\Sset^m, \manifold{N})\) satisfying
\begin{equation*}   
\varepsilon^{m - 1}
          \iint
            \limits_{\Sset^m \times \Sset^m} 
              \frac{(d_{\manifold{N}} (f (y), f (x)) - \varepsilon)_+}{\abs{y - x}^{2 m}} 
            \diff y
            \diff x 
        \le 
          \lambda
          \eofs,
\end{equation*}
has a free homotopy decomposition into \(f_1, \dotsc, f_k \in \mapsset{F}^\lambda\) with \(k \le C \lambda\).
\end{theorem}
\begin{proof}
This follows from \cref{theorem_bounded_finite_bubbles_linear} and \cref{proposition_fract_gap_integral_growth_sphere}.
\end{proof}

\subsection{Comparison between fractional truncated energies}

In passing form \cref{theorem_bounded_finite_bubbles_linear} to \cref{theorem_bounded_finite_bubbles} we relied on \eqref{ineq_diam_N}, 
which is not optimal when \(\varepsilon\) is small and \(f\) is the identity mapping (see \eqref{eq_scaling_identity}). 
In this section, we derive estimates that compare different gap integrals with optimal scaling on a convex subset \(\Omega\) of the Euclidean space \(\Rset^m\).

\begin{proposition}
[Comparison between truncated fractional energies]
\label{proposition_p_q_scaling_Omega}
For every \(m \ge 2\), \(p \in [0, m)\), \(q \in [0, + \infty)\) and \(\eta \in (0, 1)\), there exists a constant \(C > 0\) 
such that 
for every 
convex set \(\Omega \subset \Rset^m\),
for every map \(f : \Omega \to \manifold{N}\) and
for every \(\varepsilon > 0\), 
one has 
\begin{equation*}
    \iint\limits_{
      \substack{
        (x, y) \in \Omega \times \Omega\\
        d_{\manifold{N}} (f (y), f (x)) \ge \varepsilon}
      }
      \hspace{-1.5em}
      \frac
        {\bigl(
            d_{\manifold{N}} (f (y), f (x))
          -
            \varepsilon
        \bigr)^p}
        {\abs{y - x}^{2 m}}
      \diff y
      \diff x
  \le
    C 
    \,
    \varepsilon^{p - q}
    \hspace{-1.5em}
    \iint\limits_{
      \substack{
        (x, y) \in \Omega \times \Omega\\
        d_{\manifold{N}} (f (y), f (x)) \ge \eta \varepsilon}
      }
      \hspace{-1.5em}
      \frac{
        \bigl(
            d_{\manifold{N}} (f (y), f (x))
          -
            \eta \varepsilon
        \bigr)^q}
        {\abs{y - x}^{2 m}}
      \diff y
      \diff x
  \eofs .
\end{equation*}
\end{proposition}

In view of the asymptotics \eqref{eq_scaling_identity} on the integrals when \(f\) is the identity, 
the scaling of the estimate in \cref{proposition_p_q_scaling_Omega} is optimal
and the estimate of  \cref{proposition_p_q_scaling_Omega} fails when \(p \ge m\) and \(p > q\).

When \(p < q\), the estimate follows from the elementary inequality: for \(t \ge \varepsilon\), 
\begin{equation*}
   (t - \varepsilon)^p
   \le 
   \frac
    {(t - \eta \varepsilon)^q}
    {(1 - \eta)^{q - p} \,\varepsilon^{q - p}}
    \eofs;
\end{equation*}
the interest of the estimate lies essentially thus in the case \(q < p < m\).
The proof of \cref{theorem_bounded_finite_bubbles_scaling}, will be relying only on the case \(q = 0\) and \(p = 1\).

The proof of \cref{proposition_p_q_scaling_Omega} relies on \cref{lemma_fract_gap_integral_growth} and the next \cref{lemma_ineq_p_q_scaling}.

\begin{lemma}
[Integral estimate of truncated powers]
\label{lemma_ineq_p_q_scaling}
For every \(p, q \in [0, +\infty)\) and \(\eta \in (0, 1)\), 
there exists a constant \(C > 0\) such that for every \(t, s \in [0, +\infty)\)
such that \(t \ge s\),
\begin{equation*}
  (t - s)^p
 \le 
  C
  \int_{\eta s}^t
    \frac
      {(t - r)^q}
      {r^{1 + q - p}}
    \diff r
  \eofs .
\end{equation*}
\end{lemma}

If \(p = q + 1\), \cref{lemma_ineq_p_q_scaling} has a direct proof with \(\eta = 1\): indeed, for every \(t, s \in \Rset\), one has
\begin{equation*}
    (t - s)^p
  = 
    p
    \int_{s}^t 
      \frac
        {(t - r)^q}
        {r^{1 + q - p}}
      \diff r
  \eofs.
\end{equation*}

\begin{proof}[Proof of \cref{lemma_ineq_p_q_scaling}]
If we set \(t = \tau s\), with \(\tau \ge 1\), we have
\begin{equation}
\label{ineq_p_q_scaling_l}
   (t - s)^p = s^p \,(\tau - 1)^p
\end{equation}
and, under the change of variables \(r = s \rho\),
\begin{equation}
\label{ineq_p_q_scaling_r}
  \int_{\eta s}^t
    \frac
      {(t - r)^q}
      {r^{1 + q - p}}
    \diff r
 =
  s^p
  \int_\eta^\tau
    \frac
      {(\tau - \rho)^q}
      {\rho^{1 + q - p}}
    \diff \rho \eofs.
\end{equation}
We observe that the function \(g : [1, +\infty) \to \Rset\) 
defined for each \(\tau \in [1, + \infty)\) by
\begin{equation*}
 g (\tau) 
 \defeq
 \int_\eta^\tau
    \frac
      {(\tau - \rho)^q}
      {\rho^{1 + q - p}}
    \diff \rho
    \eofs,
\end{equation*}
is continuous and positive and that by the change of variable \(\rho = \tau \sigma\),
\begin{equation*}
 \lim_{\tau \to +\infty}
  \frac{g (\tau)}{\tau^p}
 =
  \lim_{\tau \to +\infty}
   \int_{\eta/\tau}^1
    \frac{(1 - \sigma)^q}{\sigma^{1 + q - p}} \diff \sigma
 =
  \int_0^1
    \frac{(1 - \sigma)^q}{\sigma^{1 + q - p}} \diff \sigma
 > 
  0
 \eofs.
\end{equation*}
Hence we have \( (\tau - 1)^p \le C g (\tau)\) for each \(\tau \in [1, +\infty)\) and the conclusion follows by 
\eqref{ineq_p_q_scaling_l} and \eqref{ineq_p_q_scaling_r}.
\end{proof}

\begin{proof}%
[Proof of \cref{proposition_p_q_scaling_Omega}]
\resetconstant
We first observe that by \cref{lemma_ineq_p_q_scaling} applied at each \(x, y \in \Omega\) with \(t = d_{\manifold{N}} (f (y), f (x))\) and \(s = \varepsilon\), we have
\begin{multline}
\label{ineq_pqsO_1}
  \iint\limits_{
      \substack{
        (x, y) \in \Omega \times \Omega\\
        d_{\manifold{N}} (f (y), f (x)) \ge \varepsilon}
      }
      \hspace{-1.5em}
      \frac
        {\bigl(
            d_{\manifold{N}} (f (y), f (x))
          -
            \varepsilon
        \bigr)^p}
        {\abs{y - x}^{2 m}}
      \diff y
      \diff x
      \\[-2em]
  \le
    \C
    \int_{\eta \varepsilon}^{+\infty}
      \hspace{-1.5em}
      \iint\limits_{
          \substack{
            (x, y) \in \Omega \times \Omega\\
            d_{\manifold{N}} (f (y), f (x)) \ge r}
          }
          \hspace{-1.5em}
          \frac
            {\bigl(
                d_{\manifold{N}} (f (y), f (x))
              -
                r
            \bigr)^q}
            {r^{1 + q - p} \, \abs{y - x}^{2 m}}
          \diff y
          \diff x 
        \diff r
        \eofs.
\end{multline}
Since the set \(\Omega \subseteq \Rset^m\) is convex, by \cref{lemma_fract_gap_integral_growth}, we have for every \(r \in (0, +\infty)\), 
\begin{multline}
\label{ineq_pqsO_2}
  \iint\limits_{
      \substack{
        (x, y) \in \Omega \times \Omega\\
        d_{\manifold{N}} (f (y), f (x)) \ge r}
      }
      \hspace{-1.5em}
      \frac
        {\bigl(
            d_{\manifold{N}} (f (y), f (x))
          -
            r
        \bigr)^q}
        {\abs{y - x}^{2 m}}
      \diff y
      \diff x
      \\[-2em]
  \le
   \C
   \,
   \Bigl(\frac{\varepsilon}{r} \Bigr)^{(m - 1) - (q - 1)_+}
   \hspace{-2em}
  \iint\limits_{
      \substack{
        (x, y) \in \Omega \times \Omega\\
        d_{\manifold{N}} (f (y), f (x)) \ge \eta \varepsilon}
      }
      \hspace{-1.5em}
      \frac
        {\bigl(
            d_{\manifold{N}} (f (y), f (x))
          -
            \eta \varepsilon
        \bigr)^q}
        {\abs{y - x}^{2 m}}
      \diff y
      \diff x 
      \eofs.
\end{multline}
By combining \eqref{ineq_pqsO_1} and \eqref{ineq_pqsO_2}, 
we deduce that 
\begin{multline*}
  \iint\limits_{
      \substack{
        (x, y) \in \Omega \times \Omega\\
        d_{\manifold{N}} (f (y), f (x)) \ge \varepsilon}
      }
      \hspace{-1.5em}
      \frac
        {\bigl(
            d_{\manifold{N}} (f (y), f (x))
          -
            \varepsilon
        \bigr)^p}
        {\abs{y - x}^{2 m}}
      \diff y
      \diff x
      \\[-.5em]
  \le 
    \C 
    \,
    \int_{\eta \varepsilon}^{+\infty}
        \frac
          {\varepsilon^{(m - 1) - (q - 1)_+}}
          {r^{m + 1 - (1 - q)_+  - p}}
        \diff r
     \hspace{-1em}
    \iint\limits_{
      \substack{
        (x, y) \in \Omega \times \Omega\\
        d_{\manifold{N}} (f (y), f (x)) \ge \eta \varepsilon}
      }
      \hspace{-1.5em}
      \frac
        {\bigl(
            d_{\manifold{N}} (f (y), f (x))
          - \eta \varepsilon
        \bigr)^p}
        {\abs{y - x}^{2 m}}
      \diff y
      \diff x \eofs.
\end{multline*}
If \(
    p
  <
    m - (1 - q)_+
\), 
then
\begin{equation*}
\begin{split}
    \int_{\eta \varepsilon}^{+\infty}
        \frac{\varepsilon^{(m - 1) - (q - 1)_+}}{r^{m + 1 - (1 - q)_+  - p}}
        \diff r
&
  = 
    \frac{\varepsilon^{(m - 1) - (q - 1)_+}}{(m - (1 - q)_+ - p) \, (\eta \varepsilon)^{m - (1 - q)_+ - p} }
\\
& =
    \frac{\varepsilon^{p - q}}{(m - (1 - q)_+ - p) \, \eta^{m - (1 - q)_+ - p} }
\end{split}
\end{equation*}
and 
the estimate is satisfied.

In order to cover the case \(
    m - (1 - q)_+ \le p < m
\), we observe that since \(m \ge 2\) and \(q \ge 0\), 
the estimate holds for every \(p \in [0, 1)\). By iterating a second time the estimate, we obtain the estimate for each \(p \in [0, m)\).
\end{proof}

The proof shows that when \(m = 1\), the estimate of \cref{proposition_p_q_scaling_Omega} holds if \(p < \min (1, q)\), in which case the estimate is in fact elementary.

The estimate of \cref{proposition_p_q_scaling_Omega}
also holds when the domain is a sphere \(\Sset^m\).

\begin{proposition}
[Comparison of truncated fractional energies on the sphere]
\label{proposition_p_q_scaling_Sphere}
For every \(m \ge 2\), \(p \in [0, m)\), \(q \in [0, + \infty)\) and \(\eta \in (0, 1)\), there exists a constant \(C > 0\) 
such that 
for every map \(f : \Sset^m \to \manifold{N}\) and
for every \(\varepsilon > 0\), 
one has 
\begin{equation*}
    \iint\limits_{
      \substack{
        (x, y) \in \Sset^m \times \Sset^m\\
        d_{\manifold{N}} (f (y), f (x)) \ge \varepsilon}
      }
      \hspace{-1.5em}
      \frac
        {\bigl(
            d_{\manifold{N}} (f (y), f (x))
          -
            \varepsilon
        \bigr)^p}
        {\abs{y - x}^{2 m}}
      \diff y
      \diff x
  \le
    C 
     \varepsilon^{p - q}
      \hspace{-2em}
    \iint\limits_{
      \substack{
        (x, y) \in \Sset^m \times \Sset^m\\
        d_{\manifold{N}} (f (y), f (x)) \ge \eta \varepsilon}
      }
      \hspace{-1.5em}
      \frac{
        \bigl(
            d_{\manifold{N}} (f (y), f (x))
          -
            \eta \varepsilon
        \bigr)^q}
        {\abs{y - x}^{2 m}}
      \diff y
      \diff x
  \eofs .
\end{equation*}
\end{proposition}

\begin{proof}
The proof follows the lines of the proof of \cref{proposition_fract_gap_integral_growth_sphere},
relying on the covering given by \cref{lemma_pair_covering} and the estimate on a convex set of \cref{proposition_p_q_scaling_Omega}.
\end{proof}

We conclude this section with a scaled version of \cref{theorem_bounded_finite_bubbles}.

\begin{theorem}%
[Free homotopy decompositions controlled by a scaled gap potential]
\label{theorem_bounded_finite_bubbles_scaling}
Let \(m \in \Nset \setminus \{0, 1\}\) and \(\manifold{N}\) be a compact Riemannian manifold.
There are constants \(\varepsilon_0 > 0\) and \(C > 0\), such that for every \(\lambda > 0\), there exists a finite set \(\mapsset{F}^\lambda \subset \mapsset{C} (\Sset^m, \manifold{N})\) such that if \(0 < \varepsilon < \varepsilon_0\), any map \(f \in \mapsset{C} (\Sset^m, \manifold{N})\) satisfying
\begin{equation*}
          \iint
            \limits_{
              \substack{
                (x, y) \in \Sset^m \times \Sset^m \\                 
                d_{\manifold{N}} (f (y), f (x)) > \varepsilon}} 
              \frac{\varepsilon^m}{\abs{y - x}^{2 m}} 
            \diff y
            \diff x
        \le 
          \lambda
          \eofs,
\end{equation*}
has a free homotopy decomposition into  \(f_1, \dotsc, f_k \in \mapsset{F}^\lambda\) with \(k \le C \lambda\).
\end{theorem}

We do not know whether \cref{theorem_bounded_finite_bubbles_scaling} holds when \(m = 2\) (see \cref{problem_1d_scaling} below).
\begin{proof}[Proof of \cref{theorem_bounded_finite_bubbles_scaling}]
\resetconstant
Since \(m \ge 2\), by \cref{proposition_p_q_scaling_Sphere}, 
we have for every \(f : \Sset^m \to \manifold{N}\) and every \(\varepsilon > 0\),
\begin{equation*}
    \varepsilon^{m - 1}
    \hspace{-1em}
    \iint
      \limits_{\Sset^m \times \Sset^m} 
      \hspace{-.5em}
        \frac{\bigl(d_{\manifold{N}} (f (y), f (x)) - \varepsilon\bigr)_+}{\abs{y - x}^{2 m}} 
      \diff y
      \diff x
  \le
    \C\,
    \varepsilon^m
    \hspace{-1.5em}
    \iint
      \limits_{
        \substack{
          (x, y) \in \Sset^m \times \Sset^m \\                 
          d_{\manifold{N}} (f (y), f (x)) > \varepsilon}} 
        \frac{1}{\abs{y - x}^{2 m}}
      \diff y 
      \diff x
      \;
      ;
\end{equation*}
the conclusion then follows from \cref{theorem_bounded_finite_bubbles_linear_scaling}.
\end{proof}

\section{Estimates of the Hurewicz homomorphism on the sphere}

The \emph{Hurewicz homomorphism} is a homotopy invariant of maps that describes how a mapping from \(\Sset^m\) to \(\manifold{N}\) acts on the cohomology of \(\manifold{N}\).
For a \emph{smooth map} \(f \in \mapsset{C}^1 (\Sset^m, \manifold{N})\) and for a closed differential form \(\omega \in \mapsset{C}^1 (\bigwedge^m \manifold{N})\), that is, a form such that \(d \omega = 0\), we define
\begin{equation}
\label{eqDefinitionHurewicz}
    \dualprod{\operatorname{Hur}(f)}{\omega}
  \defeq 
    \int_{\Sset^m} f_\sharp \omega
  \eofs,
\end{equation}
where \(f_\sharp \omega \in C (\bigwedge^m \Sset^m)\) denotes the pull-back by \(f\) of \(\omega\). 
We note that \(\operatorname{Hur}\) acts trivially on exact forms: if \(\omega = d \eta\) with \(\eta \in \mapsset{C}^2 (\bigwedge^{m - 1} \manifold{N})\), then by the Stokes--Cartan theorem
\begin{equation*}
    \dualprod{\operatorname{Hur}(f)}{\omega}
  = 
    \int_{\Sset^m} f_\sharp (d\eta)
  = 
    \int_{\Sset^m} d(f_\sharp \eta)
  = 
    0
    \eofs,
\end{equation*}
and thus by quotienting by the space of exact forms \(d (\mapsset{C}^2 (\bigwedge^{m - 1} \manifold{N}))\), \(\operatorname{Hur} (f)\) induces a linear map on the de Rahm cohomology \(H^m_{\mathrm{dR}} (\manifold{N})\), which by de Rahm's theorem, defines then an element in the singular homology group \(\operatorname{Hur} (f) \in H_m (\manifold{N}, \Rset)\simeq H_m (\manifold{N}, \Zset) \otimes \Rset\).
This map \(\operatorname{Hur} (f)\) does not depend on the homotopy of the original map \(f\): if \(H \in \mapsset{C}^2 (\Sset^m \times [0, 1]  ) \to \manifold{N}\) is a homotopy, then 
\begin{equation*}
  \int_{\Sset^m} H (\cdot, 1)_\sharp \omega
  - \int_{\Sset^m} H (\cdot, 0)_\sharp \omega
  = (-1)^m \int_{\Sset^m \times [0, 1]} d (H_\sharp \omega) 
  = 
    0
    \eofs.
\end{equation*}
In particular, \(\operatorname{Hur}\) induces a map from the homotopy group \(\pi_{m} (\manifold{N})\) to the homology group \(H_m (\manifold{N}, \Rset)\).
The Hurewicz homomorphism generalizes the degree of maps into the sphere: if \(\manifold{N} = \Sset^n\), we have \(\dualprod{\operatorname{Hur} (f)}{\omega_{\Sset^n}} = \operatorname{deg} (f)\); it extends more generally the degree of maps when \(\dim \manifold{N} = m\)  \citelist{\cite{GoldsteinHajlasz2012}\cite{HajlaszIwaniecMalyOnninen2008}*{\S 8}}.

Since the Hurewicz homomorphism is invariant under homotopies, it is well-defined for maps of vanishing mean oscillation. Moreover, by standard approximation, the formula \eqref{eqDefinitionHurewicz} is still valid whenever \(f \in W^{1, m} (\Sset^m, \manifold{N})\) (see \cite{BrezisNirenberg1995}*{(19)}).

The estimate \eqref{ineq_degree_W_1_n} can be generalized immediately to the Hurewicz homomorphism: if \(f \in W^{1, m} (\Sset^m, \manifold{N})\), then 
\begin{equation}
\label{proposition_Hurewicz_W_1_n}
 \abs{\operatorname{Hur} (f)}
 \le 
    C 
    \int_{\Sset^m} 
      \abs{D f}^m
      \eofs.
\end{equation}
Indeed, this follows from the definition of \(\operatorname{Hur}\) in \eqref{eqDefinitionHurewicz} and the fact that \(\abs{f_\sharp (\omega)} \le \abs{\omega}\, \abs{Df}^m\) almost everywhere on \(\Sset^m\).
When \(\manifold{N} = \Sset^m\), then  \eqref{proposition_Hurewicz_W_1_n} is equivalent to the degree estimate \eqref{ineq_degree_W_1_n}.

\begin{theorem}%
[Estimate of Hurewicz homomorphism by a truncated fractional energy]
\label{proposition_Hurewicz_nonlocal}%
Let \(m \in \Nset_*\) and \(\manifold{N}\) be a compact Riemannian manifold.
If \(\varepsilon > 0\) is small enough, then there exists a constant \(C > 0\) 
such that if \(f \in \mapsset{C} (\Sset^m, \manifold{N})\), then
\begin{equation*}
    \abs{\operatorname{Hur} (f)}
  \le 
    C 
    \iint 
      \limits_{\Sset^m \times \Sset^m}
      \frac
        {\bigl(d_{\manifold{N}} (f (y), f (x)) - \varepsilon)_+}
        {\abs{y - x}^{2 m}}
      \diff y 
      \diff x 
      \eofs .
\end{equation*}
\end{theorem}

When \(m \ge 2\) \eqref{proposition_Hurewicz_W_1_n} can be deduced from \cref{proposition_Hurewicz_nonlocal} since the right-hand side in \cref{proposition_Hurewicz_nonlocal} can always be controlled by the right-hand side in \eqref{proposition_Hurewicz_W_1_n} \cite{Nguyen_2006}*{lemma 2}.
When \(m = 1\), there is no such estimate \cite{Nguyen_2006}*{proposition 3},
but it might be possible to refine existing \(\Gamma\)--convergence results \cite{Nguyen_2011}*{theorem 2} in order to deduce \eqref{proposition_Hurewicz_W_1_n} from \cref{proposition_Hurewicz_nonlocal}.

\begin{proof}%
[Proof of \cref{proposition_Hurewicz_nonlocal}]%
\resetconstant%
Since \(\manifold{N}\) is a compact manifold embedded into \(\Rset^\nu\), there exists an open set \(U \subset \Rset^\nu\) such that \(\manifold{N} \subset U\) and a smooth retraction \(\Pi \in \mapsset{C}^\infty (U, \manifold{N})\).
We also consider a smooth map \(\eta \in \mapsset{C}^\infty_c (U, \Rset)\) such that \(\eta (y) = 1\) if \(y \in \manifold{N}\). 
Given \(f \in \mapsset{C} (\Sset^m, \manifold{N})\), we let \(F \in \mapsset{C}^{\infty} (\Bset^{m + 1}, \Rset^\nu)\) be given by \cref{lemmaExtensionHyperbolic} and we compute by the Stokes--Cartan formula 
\begin{equation*}
\begin{split}
    \dualprod{\operatorname{Hur}(f)}{\omega}
  = 
    \int_{\Sset^m} f_\sharp \omega
  &= 
    \int_{\Sset^m} F_\sharp (\eta \wedge \Pi_\sharp \omega)\\
  &=
    \int_{\Bset^{m + 1}} d \bigl(F_\sharp (\eta \wedge \Pi_\sharp \omega)\bigr)
  = 
    \int_{\Bset^{m + 1}} F_\sharp (d \eta \wedge \Pi_\sharp \omega)
 \eofs,
\end{split}
\end{equation*}
since \(d (\Pi_\sharp \omega) = \Pi_{\#} (d\omega) = 0\).
Hence we have, by the estimates  given by \cref{lemmaExtensionHyperbolic}
\begin{equation*}
\begin{split}
  \abs{\dualprod{\operatorname{Hur}(f)}{\omega}}
 &\le 
  \mu_{\Hset^{m + 1}} \bigl(\bigl\{ x \in \Hset^{m + 1} \st \dist (F (x), \manifold{N}) \ge \delta\bigr\}\bigr)
  \norm{D F}^{m + 1}_{L^\infty (\Hset^{m + 1})}
  \norm{\omega}_{L^\infty} \\
 &\le 
  \C
  \norm{\omega}_{L^\infty} 
  \iint 
    \limits_{\Sset^m \times \Sset^m}
      \frac{\bigl(d_{\manifold{N}} (f (y), f (x)) - \varepsilon)_+}{\abs{y - x}^{2 m}}\diff y \diff x \eofs .\qedhere
\end{split}
\end{equation*}
\end{proof}

We also have an estimate of the Hurewicz homomorphism with optimal scaling when \(m \ge 2\).

\begin{theorem}%
[Estimate of the Hurewicz homomorphism by a scaled gap potential]
\label{proposition_Hurewicz_nonlocal_scaling}%
If \(m \ge 2\) and \(\manifold{N}\) is a compact Riemmanian manifold, then there exists constants \(C > 0\) and \(\varepsilon_0 > 0\) such that if \(\varepsilon \in (0, \varepsilon_0)\) and \(f \in \mapsset{C} (\Sset^m, \manifold{N})\), then
\begin{equation*}
    \abs{\operatorname{Hur} (f)}
  \le 
    C
    \iint 
      \limits_{\substack{
        x, y \in \Sset^m \\ 
        d_{\manifold{N}} (f (y), f (x)) \ge \varepsilon}}
      \frac
        {\varepsilon^m}
        {\abs{y - x}^{2 m}} 
      \diff y 
      \diff x
\eofs .
\end{equation*}
\end{theorem}
\begin{proof}
This follows from \cref{proposition_Hurewicz_nonlocal} in view of \cref{proposition_fract_gap_integral_growth_sphere} and \cref{proposition_p_q_scaling_Sphere}.
\end{proof}

When \(\manifold{N} = \Sset^m\) we recover the estimate on the degree of \familyname{Nguyên} Hoài-Minh \cite{Nguyen_2017}; the latter estimate was obtained through the John--Nirenberg estimate and seems different from our  direct approach. When \(m = 1\), the question whether \cref{proposition_Hurewicz_nonlocal_scaling} holds is an open problem (\cref{problem_1d_scaling}).

\section{Homotopy estimates on a compact manifold}
\label{section_domain_manifold}

\subsection{Free homotopy decompositions upon a mapping}
We consider the problem of controlling the homotopy classes of maps from a general compact manifold \(\manifold{M}\) to another compact manifold \(\manifold{N}\). 
The notion of free homotopy decomposition (\cref{definition_homotopy_power})
generalizes into the free homotopy decomposition upon a mapping.
Since the circle \(\Sset^1\) is, up to a conformal transformation, the only connected compact one-dimensional Riemannian manifold, we assume throughout this section that \(\dim (\manifold{M}) = m \ge 2\).

\begin{definition}%
[Free homotopy decomposition upon a mapping]
\label{homotopy_sum_manifold}
Let \(\manifold{M}\) and \(\manifold{N}\) be connected compact Riemannian manifolds and let \(m = \dim \manifold{M}\).
A map \(f \in \mapsset{C} (\manifold{M}, \manifold{N})\) has a \emph{free homotopy decomposition into} the maps \(f_1, \dotsc, f_k \in \mapsset{C} (\Sset^m, \manifold{N})\) \emph{upon} the map \(f_0 \in \mapsset{C} (\manifold{M}, \manifold{N})\)
whenever there exist maps \(g, g_0 \in \mapsset{C} (\manifold{M}, \manifold{N})\) and nondegenerate topologically trivial balls \(B_{\rho_0} (a_0), \dotsc, B_{\rho_\ell} (a_k)\), such that \(g\) is homotopic to \(f\), \(g_0\) is homotopic to \(f_0\), 
\(g = g_0\) on \(\manifold{M} \setminus B_{\rho_0} (a_0)\), \(g_0\) is constant on \(\Bar{B}_{\rho_0} (a_0)\),  and 
for every \(i \in \{1, \dotsc, k\}\), \(\Bar{B}_{\rho_i} (a_i) \subset B_{\rho_0} (a_0)\) and \(g\vert_{\Bar{B}_{\rho_i}(a_i)}\) is homotopic to \(f_i\) on \(\Sset_m \simeq \Bar{B}_{\rho_i}(a_i)/\partial B_{\rho_i} (a_i)\). 
\end{definition}

\begin{remark}
\label{remark_manifold_sphere}
When \(\manifold{M} = \Sset^m\), a map \(f \in \mapsset{C} (\Sset^m, \manifold{N})\) has a free homotopy decomposition into \(f_1, \dotsc, f_k \in \mapsset{C} (\Sset^m, \manifold{N})\) in the sense of \cref{definition_homotopy_power} if and only if \(f\) has a free homotopy decomposition into \(f_1, \dotsc, f_k \in \mapsset{C} (\Sset^m, \manifold{N})\) upon a constant map \(f_0 \in \mapsset{C} (\Sset^m, \manifold{N})\) in the sense of \cref{homotopy_sum_manifold}.
\end{remark}

The next propopsition describes free homotopy decompositions upon a mapping on a general manifold through free homotopy decompositions on the sphere.

\begin{proposition}
\label{proposition_free_decomposition_manifold_sphere}
Let \(f_0 \in \mapsset{C} (\manifold{M}, \manifold{N})\) and \(f_1, \dotsc, f_k \in (\Sset^m, \manifold{N})\). Assume that \(f_0\) is constant on some nondegenerate topologically trivial ball \(B_\rho (a)\subset \manifold{M}\).
The map \(f \in \mapsset{C} (\manifold{M}, \manifold{N})\) has a free homotopy decomposition into \(f_1, \dotsc, f_k\) upon \(f_0\) if and only there exists a map \(h \in \mapsset{C} (\manifold{M}, \manifold{N})\) such that \(h\) is homotopic to \(f\), \(h = f_0\) on \(\manifold{M} \setminus B_\rho (a)\) and \(h \vert_{\Bar{B}_\rho (a)}\) has a free homotopy decomposition into \(f_1, \dotsc, f_k\) under the identification \(\Sset^m \simeq \Bar{B}_\rho (a)/\partial B_\rho (a)\).
\end{proposition}

Since the definition of free homotopy decomposition upon a mapping (\cref{homotopy_sum_manifold}) is invariant under homotopies, the condition that the map \(f_0\) is constant on some nondegenerate topologically trivial ball can always be satisfied. 
Free decompositions upon a given mapping on a manifold are thus not more complex than a collection of homotopy classes of maps on a sphere relative to some point. 
The free homotopy decompositions into given maps upon a given map can be precisely identified and enumerated by obstruction cohomology classes with local groups \citelist{\cite{Hu_1959}*{Chapter VI}\cite{Baues_1977}*{\S 4.2}}.

\begin{proof}%
[Proof of \cref{proposition_free_decomposition_manifold_sphere}]
If the map \(f\) is homotopic to \(h\), it follows directly from the definition of free decompositions \cref{definition_homotopy_power} and \cref{homotopy_sum_manifold} that \(f\) has a free homotopy decomposition into \(f_1, \dotsc, f_k\) upon \(f_0\).

Conversely, let us assume that \(f \in \mapsset{C} (\manifold{M}, \manifold{N})\) has a free homotopy decomposition into \(f_1, \dotsc, f_k\) upon \(f_0\), and let \(g, g_0 \in \mapsset{C} (\manifold{M}, \manifold{N})\) and the balls \(B_{\rho_0} (a_0), \dotsc, B_{\rho_k} (a_k)\) be given by \cref{homotopy_sum_manifold}. Without loss of generality, we can assume that \(B_{\rho_0} (a_0) = B_{\rho/2} (a)\).
Since the maps \(f_0\) and \(g_0\) are homotopic, there exists a homotopy \(H \in \mapsset{C} (\manifold{M} \times [0, 1], \manifold{N})\) such that \(H (\cdot, 0) = g_0\)  and \(H (\cdot, 1) = f_0 \) on \(\manifold{M}\) and such that for every \(t \in [0, 1]\), \(H (\cdot, t)\) is constant on \(B_\rho (a)\). 
This implies in turn the existence of a homotopy \(\Tilde{H} \in \mapsset{C} (\manifold{M} \times [0, 1], \manifold{N})\) such that \(\Tilde{H} (\cdot, 0) = g\) on \(\manifold{M}\), \(\Tilde{H} (\cdot, 1) = f_0\) on \(\manifold{M} \setminus B_{\rho} (a)\), for every \(t \in [0, 1]\), \(\Tilde{H} (\cdot, t)\) is radial on \(B_\rho (a) \setminus B_{\rho/2} (a)\) and \(\Tilde{H} (\cdot, t)= g\) on \(B_{\rho/2} (a)\). 
We set \(h \defeq \Tilde{H} (\cdot, 1)\) and we conclude by the observation that \(h \vert_{B_\rho (a)}\) has a free homotopy decomposition into \(f_1, \dotsc, f_k\) under the identification \(\Sset^m \simeq B_\rho (a)/\partial B_\rho (a)\) and that by transitivity of homotopies \(h\) is homotopic to \(f\).
\end{proof}

As a consequence we of \cref{proposition_free_decomposition_manifold_sphere}, we prove a counterpart of \cref{proposition_finite_homotopy_classes} for free homotopy decompositions upon a mapping asserting the finiteness of homotopy classes sharing a given free homotopy decomposition.

\begin{proposition}
If \(m \ge 2\) and every orbit of the action of \(\pi_1 (\manifold{N})\) on \(\pi_m (\manifold{N})\) is finite, then for every \(f_0 \in \mapsset{C} (\manifold{M}, \manifold{N})\) and \(f_1, \dotsc, f_k \in \mapsset{C} (\Sset^m, \manifold{N})\), there exists a finite set \(\mapsset{G} \subset \mapsset{C} (\Sset^m, \manifold{N})\) such that 
every map \(f \in \mapsset{C} (\Sset^m, \manifold{N})\) that has a free homotopy decomposition into \(f_1, \dotsc, f_k\) upon \(f_0\) is homotopic to some \(g \in \mapsset{G}\).   
\end{proposition}

\begin{proof}
We assume up to a homotopy and without loss of generality that \(f_0 = b \in \manifold{N}\) on a trivial ball \(B_\rho (a) \subset \manifold{M}\).
We consider \(\gamma_1, \dotsc, \gamma_k \in \pi_m (\manifold{N}, b)\) respectively homotopic to \(f_1, \dotsc, f_k\) and we set 
\[
    \Gamma 
  \defeq 
    \bigl\{
          \beta_1 \cdot \gamma_1 
        + 
          \dotsb 
        + 
          \beta_k \cdot \gamma_k 
      \st 
          \beta_1, \dotsc, \beta_k 
        \in 
          \pi_1 (\manifold{N}, b)
    \bigr\}
 \eofs.
\]
Since by assumption for every \(i \in \{1, \dotsc, k\}\) the set \(\{\beta_i \cdot \gamma_i \st \beta_i \in \pi_1 (\manifold{N}, b)\}\) is finite, the set \(\Gamma\) is finite and we can construct \(\Tilde{\mapsset{G}} \subset \mapsset{C} (B_\rho (a), \manifold{N})\) as a finite set of mappings taking the constant value \(b\) on \(\partial B_\rho (a)\) and such that under the identification \(\Sset^m \simeq  B_\rho (a)/\partial B_\rho (a)\), every element of \(\Gamma\) is homotopic to some map in \(\Tilde{\mapsset{G}}\). We define now 
\[
  \mapsset{G} 
 \defeq
  \bigl\{ 
    g \in \mapsset{C} (\manifold{M}, \manifold{N})
  \st
    g = f_0 \text{ in \(\manifold{M} \setminus B_\rho (a)\) and }
    g \vert_{B_{\rho} (a)} \in \Tilde{\mapsset{G}}
    \,
    \bigr\}.
\]
By  \cref{proposition_free_decomposition_manifold_sphere}, and  \cref{free_decomposition_pi_m} any map \(f \in \mapsset{C} (\manifold{M}, \manifold{N})\) that has a free homotopy decomposition into \(f_1, \dotsc, f_k\) upon \(f_0\) is homotopic to some \(g \in \mapsset{G}\).   
\end{proof}

\subsection{Estimates of free homotopy decompositions}
The counterpart of \cref{theorem_bounded_finite_bubbles} when the domain is a general compact Riemannian manifold manifold \(\manifold{M}\) is the following

\begin{theorem}
[Free decomposition estimate by a gap potential]
\label{theorem_bounded_finite_bubbles_manifold}
Let \(\manifold{M}\) and \(\manifold{N}\) be connected compact Riemannian manifolds.
If \(\varepsilon > 0\) is small enough, then there is a constant \(C > 0\) such that for every \(\lambda > 0\), there exist finite sets \(\mapsset{F}^\lambda \subset \mapsset{C} (\Sset^m, \manifold{N})\) and \(\mathcal{F}_0^\lambda \subset \mapsset{C} (\manifold{M}, \manifold{N})\) such that any \(f \in \mapsset{C} (\manifold{M}, \manifold{N})\) satisfying
\begin{equation*}
          \iint
            \limits_{
              \substack{
                (x, y) \in \manifold{M} \times \manifold{M}\\                 
                d_{\manifold{N}} (f (y), f (x)) > \varepsilon}} 
              \frac{1}{d_{\manifold{M}} (y, x)^{2 m}} 
            \diff x 
            \diff y 
        \le 
          \lambda
          \eofs,
\end{equation*}
has a free homotopy decomposition into \(f_1, \dotsc, f_k \in \mapsset{F}^\lambda\) upon \(f_0 \in\mapsset{F}_0^\lambda\), with \(k \le C \lambda\).
\end{theorem}

In view of \cref{remark_manifold_sphere}, \cref{theorem_bounded_finite_bubbles} corresponds to the particular case \(\manifold{M} = \Sset^m\) in \cref{theorem_bounded_finite_bubbles_manifold}.

In order to follow in the proof of \cref{theorem_bounded_finite_bubbles_manifold} the same strategy as in the proof of \cref{theorem_bounded_finite_bubbles}, we construct a Riemmanian manifold that is the counterpart of the hyperbolic space \(\Hset^{m + 1}\) for \(\Sset^{m}\). 
We define the manifold \(\manifold{M}^\star \defeq \manifold{M} \times (0, + \infty) 
\) and we endow it with a metric \(g^{\manifold{M}^\star}\) defined as a quadratic form for each point \((x, t) \in \manifold{M}^\star\) and each tangent vector \((v, w) \in T_{(x, t)} \manifold{M}^\star \simeq T_{x} \manifold{M} \times \Rset\) by 
\begin{equation}
\label{def_g_star}
    g^{\manifold{M}^\star}_{(x, t)} (v, w)
  \defeq
    \frac{g_x (v) + w^2}{t^2}, 
\end{equation}
where \(g\) is the metric of the original manifold \(\manifold{M}\).
When \(\manifold{M} = \Rset^m\), the manifold \(\manifold{M}^\star\) is the \emph{Poincaré half-space model} of the hyperbolic space. 
The formula \eqref{def_g_star} shows that the manifold \(\manifold{M}^\star\) is conformally equivalent to the Riemannian cartesian product \(\manifold{M} \times (0, + \infty)\).

\begin{remark}
The manifold \(\manifold{M}^\star\) is in fact a \emph{warped product}: if \(\manifold{M}^{\diamond} \defeq \manifold{M} \times \Rset\) is endowed with the metric 
\(g^{\manifold{M}^\diamond}\) defined as a quadratic form for each \((x, t) \in \manifold{M}^\star\) and \((v, w) \in T_{(x, t)} \manifold{M}^\star \simeq T_{x} \manifold{M} \times \Rset\) by \(g^{\manifold{M}^\diamond}_{(x, t)} (v, w) \defeq e^{-2s} g_x (v) + w^2\), the mapping \((x, s) \in \manifold{M}^\diamond \defeq \manifold{M} \times \Rset \mapsto (x, e^{s}) \in \manifold{M}^\star\) is an isometry. (In fact \(\manifold{M}^\star\) is the conformal representation of the warped product \(\manifold{M}^\diamond\) \cite{Petersen_2016}*{\S 4.3}.)
\end{remark}

\subsubsection{Extension}

The first tool that we need is a controlled extension to \(\manifold{M}\) corresponding to \cref{lemmaExtensionHyperbolic}.

\begin{proposition}
[Extension to \(\manifold{M}^\star\)]
\label{proposition_extension_manifold}
Let \(\manifold{M}\) be compact Riemannian manifold.
There exists a constant \(C > 0\) such that for every \(f \in \mapsset{C} (\Sset^m, \Rset^\nu)\),
there exists a function \(F \in \mapsset{C} (\manifold{M} \times [0, +\infty), \Rset^\nu) \cap \mapsset{C}^\infty (\manifold{M} \times (0, +\infty), \Rset^\nu)\) such that 
\begin{enumerate}[(i)]
 \item 
  \label{itExtensionManifoldHyperbolicTrace} 
  \(F (\cdot, 0) = f\) on \(\manifold{M}\),
 \item \label{itExtensionManifoldHyperbolicLipschitz} for every \(z \in \manifold{M}^\star \simeq \manifold{M} \times (0, + \infty)\), 
  \begin{equation*}
      \abs{D F (z)}_{\manifold{M}^\star} 
    \le 
      C 
      \osc_{\Sset^m} f
  \eofs,
  \end{equation*}
 \item \label{itExtensionManifoldHyperbolicMeasure} if \(\delta > \varepsilon\), then
 \begin{equation*}
   \mu_{\manifold{M}^\star} 
     \bigl(\bigl\{ x \in \manifold{M}^\star
                   \st \dist \bigl(F (x), f (\Sset^m)\bigr) \ge \delta 
     \bigr\}\bigr)
   \le 
    \frac
      {C}
      {\delta - \varepsilon}  
      \iint \limits_{\manifold{M} \times \manifold{M}}
        \frac
          {\bigl(\abs{f (y) - f (x)}-\varepsilon)_+}
          {d_{\manifold{M}} (y, x)^{2 m}} 
        \diff y 
        \diff x
  \eofs .
 \end{equation*}
\end{enumerate}
\end{proposition}

The proof of \cref{lemmaExtensionHyperbolic} relied on the hyperharmonic extension defined in \eqref{eq_def_hyperharmonic}. 
In order to define a similar extension, we introduce a suitable integral kernel.

\begin{lemma}
[Approximation of the identity on \(\manifold{M}\) along \(\manifold{M}^\star\)]
\label{lemma_manifold_kernel}
Let \(\manifold{M}\) be an \(m\)--dimensional compact Riemannian manifold.
There exists a function \(\Phi \in \mapsset{C}^\infty (\manifold{M}^\star \times \manifold{M}, [0, + \infty))\) and a constant \(C > 0\) such that 
\begin{enumerate}[(i)]
\item
\label{it_manifold_kernel_1}
for every \((x, t) \in \manifold{M}^\star = \manifold{M} \times (0, +\infty)\),
 \begin{equation*}
  \int_{\manifold{M}} \Phi (x, t, \cdot) = 1
  \eofs,
 \end{equation*}
\item %
\label{it_manifold_kernel_bound}%
for every \((x, t, y) \in \manifold{M}^\star \times \manifold{M}  = \manifold{M} \times (0, +\infty) \times \manifold{M}\),
\begin{equation*}
    \Phi (x, t, y) 
  \le 
    \frac
      {C \,t^m}
      {d_{\manifold{M}} (y, x)^{2 m}}
  \eofs.
\end{equation*}
\item 
 \label{it_manifold_kernel_derivative} if we set for every \((x, t, y) \in \manifold{M}^\star \times \manifold{M}  = \manifold{M} \times (0, +\infty) \times \manifold{M}\), \(\Phi_y (x, t) = \Phi (x, t, y)\), then for every \((x, t) \in \manifold{M}^\star = \manifold{M} \times (0, +\infty)\),
\begin{equation*}
    \int_{\manifold{M}} 
      \abs{D \Phi_y (x, t)}_{\manifold{M}^\star} 
      \diff y
  \le 
    C
\eofs.
\end{equation*}
\end{enumerate}
\end{lemma}
\begin{proof}%
\resetconstant
We choose a function \(\varphi : \mapsset{C}^\infty (\Rset, [0, + \infty))\) such that \(\varphi (0) > 0\), and \(\varphi  = 0\) on \(\Rset \setminus (-1, 1)\), and 
\begin{equation*}
    \int_{\Rset^m} 
      \varphi (\abs{z}^2) 
      \diff z 
  = 
    1
 \eofs.
\end{equation*}
Since the manifold \(\manifold{M}\) is compact, there exists \(\delta > 0\) such that if \(E \defeq \{(x,y) \in \manifold{M} \times \manifold{M}\st d_{\manifold{M}} (x, y) < \delta\}\) the function \((x, y) \in E \mapsto d_{\manifold{M}} (x, y)^2\) is smooth. We fix a function \(\eta \in \mapsset{C}^\infty ((0,+ \infty), \Rset)\) in such a way that \(\eta (t) = 0\) when \(t \le \delta/3\) and \(\eta (t) = 1\) when \(t \ge 2\delta/3 \).

 We define successively the functions \(\Tilde{\Phi} : \manifold{M}^\star \times \manifold{M} \to \Rset\) and  \(\Phi : \manifold{M}^\star \times \manifold{M} \to \Rset\) for every \((x, t, y) \in \manifold{M}^\star \times \manifold{M}  = \manifold{M} \times (0, +\infty) \times \manifold{M}\) by 
\begin{equation*}
  \Tilde{\Phi} (x, t, y) 
 \defeq
    \bigl(1 - \eta (t)\bigr)
    \, 
    \varphi 
      \biggl(
        \frac{d_{\manifold{M}} (y, x)^2}{t^2} 
      \biggr) 
  + 
    \eta (t),
\end{equation*}
and 
\begin{equation*}
    \Phi (x, t, y) 
  = 
    \frac{\Tilde{\Phi} (x, t, y)}{\int_{\manifold{M}} \Tilde{\Phi} (x, t, y) \diff y}.
\end{equation*}
We verify immediately that \eqref{it_manifold_kernel_1} holds by construction.

The second assertion \eqref{it_manifold_kernel_bound} follows from the observations that the function \(\Tilde{\Phi}\) is bounded, that \(\Tilde{\Phi} (x, t, y) = 0\) if \(d_{\manifold{M}} (x, y) \ge t\) and that for some constant \(\Cl{cst_raeph5ueGh} > 0\), for every \(x, t \in \manifold{M}^\star = \manifold{M} \times (0, +\infty)\),
\[
    \int_{\manifold{M}} 
      \Tilde{\Phi} (x, t, y) 
      \diff y
  \ge
    \Cr{cst_raeph5ueGh}\,
    t^{m}
    \eofs.
\]

For the last assertion \eqref{it_manifold_kernel_derivative} we observe that the map 
\begin{equation*}
 (x, t) \in \manifold{M}^\star \mapsto \int_{\manifold{M}} \abs{D \Phi_y (x, t)}_{\manifold{M}^\star}
\end{equation*}
is continuous, that if \(t \ge \delta\), 
\begin{equation*}
    \int_{\manifold{M}} 
      \abs{D \Phi_y (x, t)}_{\manifold{M}^\star} 
  = 
    0
\end{equation*}
and that if \((x_j)_{j \in \Nset}\) is any sequence in \(\manifold{M}\) and if \((t_j)_{j \Nset}\) is a sequence in \((0, + \infty)\) converging to \(0\), then 
\begin{multline*}
    \lim_{j \to \infty} 
      \int_{\manifold{M}} 
        \abs{D \Phi_y (x_j, t_j)}_{\manifold{M}^\star}\\
  = 
    \int_{\Rset^m} 
      \sqrt{\bigabs{2 \varphi' (\abs{z}^2) z}^2 + \bigabs{2 \varphi' (\abs{z}^2) \abs{z}^2 + m \varphi (\abs{z}^2)}} 
      \diff z < + \infty
      \eofs;
\end{multline*}
hence \eqref{it_manifold_kernel_derivative} follows from the classical extreme value theorem for continuous functions.
\end{proof}

\begin{proof}[Proof of \cref{proposition_extension_manifold}]%
\resetconstant
We define the function \(F : \manifold{M}^\star \to \Rset^\nu\) by setting for every \((x, t) \in \manifold{M}^\star = \manifold{M} \times (0, + \infty)\),
\begin{equation*}
    F (x, t) 
  \defeq 
    \int_{\manifold{M}} 
      \Phi (x, t, y) 
      f (y) 
      \diff y
 \eofs,
\end{equation*}
where the function \(\Phi \in \mapsset{C}^\infty (\manifold{M}^\star \times \manifold{M})\) is given by \cref{lemma_manifold_kernel}.

We first observe that for every \(x \in \manifold{M}\), \(\lim_{(y, t) \to (0, 0)} F (y, t) = f (x)\) and thus assertion \eqref{itExtensionManifoldHyperbolicTrace} holds.

For \eqref{itExtensionManifoldHyperbolicLipschitz}, we note that by \cref{lemma_manifold_kernel} \eqref{it_manifold_kernel_1} we have for every \((z, s) \in \manifold{M}^\star = \manifold{M} \times (0, + \infty)\) and  every \(x \in \manifold{M}\), 
\begin{equation*}
 F (z, s) - f (x)
  = 
    \int_{\manifold{M}} 
      \Phi (z, s, y)
      \,
      \bigl(f (y) - f (x)\bigr) 
      \diff y
      \eofs, 
\end{equation*}
and thus by differentiating with respect to \((z, s)\) at the point \((x, t)\), we obtain
\begin{equation*}
    D F (x, t) 
  = 
    \int_{\manifold{M}} 
      D \Phi_y (x, t)
      \,
      \bigl(f (y) - f (x)\bigr) 
      \diff y
\eofs,
\end{equation*}
and thus we deduce from \cref{lemma_manifold_kernel} \eqref{it_manifold_kernel_derivative} that 
\begin{equation*}
    \abs{D F (x, t)}_{\manifold{M}^\star}
  \le 
    \osc_{\manifold{M}} f 
    \int_{\manifold{M}}
      \abs{D \Phi_y (x, t)}_{\manifold{M}^\star}
      \diff y 
  \le 
    \C 
    \osc_{\manifold{M}} f
  \eofs.
\end{equation*}

For the last part \eqref{itExtensionManifoldHyperbolicMeasure}, we first observe that for each \((x, t) \in \manifold{M}^\star \simeq \manifold{M} \times (0, +\infty)\), we have 
\begin{equation*}
    \dist \bigl(F (x, t), f (\manifold{M})\bigr)
  \le 
    \abs{F (x, t) - f (x)}
  \le 
    \int_{\manifold{M}} 
      {\Phi (x, t, y)} 
      \,
      \abs{f (y) - f (x)} 
    \diff y \eofs .
\end{equation*}
Hence we infer from \cref{lemma_manifold_kernel} \eqref{it_manifold_kernel_bound},
\begin{equation}
\label{ineq_distance_manifold_F_f}
\begin{split}
  \dist \bigl(F (x, t), f (\manifold{M})\bigr)
 &\le 
      \varepsilon 
      \int_{\manifold{M}} 
        {\Phi (x, t, y)} 
        \diff y
    + 
      \int_{\manifold{M}} 
        {\Phi (x, t, y)}
        \,
        \bigl(\abs{f (y) - f (x)} - \varepsilon\bigr)_+
        \diff y\\
 &\le 
      \varepsilon 
    + 
      \Cl{cst_ineq_distance_manifold_F_f}
      \,
      t^m 
      \int_{\manifold{M}}  
        \frac
          {\bigl(\abs{f (y) - f (x)} - \varepsilon\bigr)_+}
          {d_{\manifold{M}} (y, x)^{2 m} } 
        \diff y \eofs .
\end{split}
\end{equation}
We define now the set 
\begin{equation*}
    A_\delta 
  \defeq
    \bigl\{
        (x, t) \in \manifold{M}^\star
      \st 
        \dist \bigl(F (x, t), f (\manifold{M})\bigr) \ge \delta
    \bigr\}
\end{equation*}
and, for each \(x \in \manifold{M}\), the quantity 
\begin{equation*}
    \tau_\delta (x) 
  \defeq
    \inf\,
      \bigl\{
          t \in (0, +\infty) 
        \st 
          (x, t) \in A_\delta
      \bigr\}
\eofs,
\end{equation*}
and we compute 
\begin{equation}
\label{ineq_manifold_mu_A_delta}
    \mu_{\manifold{M}^\star} (A_\delta)
  =
    \iint_{(x, t) \in A_\delta} 
      \frac{1}{t^{m + 1}} 
      \diff x 
      \diff t
  \le 
    \int_{\manifold{M}} 
      \int_{\tau_\delta (x)}^{+\infty}\;
        \frac{1}{t^{m + 1}} 
      \diff t
    \diff x\\
  \le
    \int_{\manifold{M}} \,
      \frac{1}{m \, \tau_\delta (x)^{m}} 
    \diff x\eofs.
\end{equation}
By \eqref{ineq_distance_manifold_F_f}, we have 
\begin{equation*}
      \delta 
    - 
      \varepsilon
  = 
      \dist \bigl(F (x, \tau_\delta (x)), f (\manifold{M})\bigr)
    -
      \varepsilon
  \le 
    \Cr{cst_ineq_distance_manifold_F_f}
    \,
    \tau_\delta (x)^m
    \int_{\manifold{M}}
      \frac
        {\bigl(d_{\manifold{N}} (f (y), f (x)) - \varepsilon\bigr)_+}
        {d_{\manifold{M}} (y, x)^{2 m} } 
      \diff y
\eofs,
\end{equation*}
and thus by \eqref{ineq_manifold_mu_A_delta}, we conclude that 
\begin{equation*}
  \mu_{\manifold{M}} (A_\delta)
 \le
  \frac
    {\Cr{cst_ineq_distance_manifold_F_f}}
    {m\, (\delta - \varepsilon)} 
  \iint
    \limits_{\manifold{M} \times \manifold{M}}
    \frac{\bigl(d_{\manifold{N}} (f (y), f (x)) - \varepsilon\bigr)_+}{d_{\manifold{M}} (y, x)^{2 m} } 
  \diff y \diff x
\eofs .
\qedhere
\end{equation*}
\end{proof}

\subsubsection{Merging balls}
We will use a generalization of \cref{lemmaDisjointCovering}, that incorporates also a \emph{horoball}, that is, a ball centered at the point at infinity of \(\manifold{M}^\star\). 

\begin{lemma}%
[Merging balls and a horoball on \(\manifold{M}^\star\)]%
\label{lemma_merging_Mstar}
Given an nonnegative integer \(\ell \in \Nset_*\), some 
points \(a_1, a_2, \dotsc, a_\ell \in \manifold{M}^\star\), radii \(r_1, r_2, \dotsc, r_\ell \in (0, + \infty)\) and \(T \in (0, +\infty)\), 
there exists
\(\ell' \in \{0, 1, \dotsc, \ell\}\), points \(a_1', \dotsc, a_\ell' \in \manifold{M}^\star\), radii \(r_1', \dotsc, r_\ell'\) and \(T' \in (0, T)\) such that 
\begin{enumerate}[(i)]
\item
  \(
    \displaystyle
        \bigcup_{i = 1}^\ell 
        B^{\manifold{M}^\star}_{r_i} (a_i) 
      \cup \manifold{M} \times (T, + \infty)
    \subseteq 
      \bigcup_{i = 1}^{\ell'} 
        B^{\manifold{M}^\star}_{r_i'} (a_i') 
      \cup 
        \manifold{M} \times (T', + \infty)
  \),
\item
  \(
    \displaystyle
        \frac{1}{2}
        \ln \frac{1}{T'} 
      + 
        \sum_{i = 1}^{\ell'} r_i' 
    \le 
        \frac{1}{2}
        \ln \frac{1}{T} 
      + 
        \sum_{i = 1}^\ell r_i
  \),
\item 
for every \(i, j \in \{1, 2, \dotsc, \ell'\}\) such that \(i \ne j\), one has 
\(
      \Bar{B}^{\manifold{M}^\star}_{r_i'} (a_i') 
    \cap 
      \Bar{B}^{\manifold{M}^\star}_{r_j'} (a_j') 
  = 
    \emptyset
\),
\item 
for every \(i \in \{1, 2, \dotsc, \ell'\}\), 
one has 
  \(
      \Bar{B}^{\manifold{M}^\star}_{r_i} (a_i') 
    \subseteq 
      \manifold{M} \times (0, T')
  \).
\end{enumerate}
\end{lemma}

In contrast to \cref{lemmaDisjointCovering}, we must allow in the conclusion \(\ell' = 0\) if \(T\) was too small at the beginning.

In order to prove \cref{lemma_merging_Mstar} we need to have good estimates on the distances between points. 
It turns out that this distance can be computed exactly in terms of the distance on \(\manifold{M}\).

\begin{lemma}[Distance formula on \(\manifold{M}^\star\)]
\label{distancesStar}
For every \((x, t), (y, s) \in \manifold{M}^\star = \manifold{M} \times (0, +\infty)\), we have 
\begin{equation*}
\begin{split}
    d_{\manifold{M}^\star} \bigl((x, t), (y, s)\bigr)
  &= 
    2 
    \sinh^{-1} 
      \left(
        \frac
          {\sqrt{\vphantom{\big\vert}d_{\manifold{M}} (x, y)^2 + (t - s)^2\,}}
          {2\, \sqrt{ts\,}}
        \,
      \right)\\
  &=
   \cosh^{-1} 
      \Biggl(
        \frac
          {d_{\manifold{M}} (x, y)^2 +  t^2 + s^2}
          {2\, t s} \Biggr)
   \eofs.
\end{split}
\end{equation*}
\end{lemma}
\begin{proof}
We consider a geodesic \(\gamma : \Rset \to \manifold{M}\) such that \(\gamma (0) = x\) and \(\gamma (d_{\manifold{M}} (x, y)) = y\). 
If \(\Hset^2\) denotes the Poincaré half-space model of the hyperbolic plane
and if we define \(\Phi : \Hset^2 \to \manifold{M}^\star\) for \((z, r)\in \Hset^2 \simeq \Rset \times (0, + \infty)\) by \(\Phi (z, r) \defeq (\gamma(z), r)\), we observe that since \(\gamma\) is a geodesic, the map \(\Phi\) is a local isometry and thus is globally nonexpansive from the hyperbolic plane \(\Hset^2\) to \(\manifold{M}^\star\).
Therefore, we have, by a classical computation of the hyperbolic distance \cite{Fenchel_1989}*{III.4 (1) and (2)}, for every \(s, t \in (0, + \infty)\),
\begin{equation*}
\begin{split}
  d_{\manifold{M}^\star} \bigl((x, t), (y, s)\bigr) 
  &= d_{\manifold{M}^\star} \bigl(\Phi (0, t), \Phi(d_{\manifold{M}} (x, y), s)\bigr)\\
  &
  \le 
    d^{\Hset^2} \bigl((0, t), (d_{\manifold{M}} (x,y), s)\bigr)\\
  &=    
    2
    \sinh^{-1} 
      \left(
        \frac
          {\sqrt{\vphantom{\big\vert}d_{\manifold{M}} (x, y)^2 + (t - s)^2\,}}
          {2 \sqrt{ts\,}}
        \,
      \right)\\
  &
  =
    \cosh^{-1} 
      \Biggl(
        \frac
          {d_{\manifold{M}} (x, y)^2 +  t^2 + s^2}
          {2\, t s} \Biggr)
      \eofs.
\end{split}
\end{equation*}

For the converse, given \((x, t), (y, s) \in \manifold{M}^\star = \manifold{M} \times (0, +\infty)\), we consider a path \(\gamma^{\star} = (\gamma, \tau) \in \mapsset{C}^1 ([0, 1], \manifold{M}^\star)= \mapsset{C}^1 ([0, 1], \manifold{M} \times (0, + \infty))\) such that \(\gamma^{\star} (0) = (\gamma (0), \tau (0)) = (x, t)\) and  \(\gamma^{\star} (1) = (\gamma (1), \tau (1)) = (y, s)\).
We define the function \(\xi : [0, 1] \to \Rset\) for each \(r \in [0, 1]\) by 
\begin{equation*}
  \xi (r)
  \defeq 
    \int_0^r 
      \sqrt{g (\gamma' (\rho))\,} 
      \diff \rho
  \eofs.
\end{equation*}
We observe than that 
\begin{equation*}
\begin{split}
    \int_0^1 
      \frac{g (\gamma' (r)) + \abs{\tau' (r)}^2}
      {\tau (r)^2} 
      \diff r
  &
  =
    \int_0^1
      \frac{\abs{\xi' (r)}^2 + \abs{\tau' (r)}^2}
      {\tau (r)^2} 
      \diff r\\    
  &
    \ge d^{\Hset^2} \bigl((0, t), (\xi (1), s)\bigr)\\
  &
    =  2 \sinh^{-1} \left(\frac{\sqrt{\vphantom{\big\vert} \xi(1)^2 + (t - s)^2}}{2 \sqrt{ts\,}}\right)
\eofs ,
\end{split}
\end{equation*}
and thus since \(\xi (1) \ge d_{\manifold{M}} (x, y)\) by definition of geodesic distance
\begin{equation*}
    \int_0^1 
      \frac{g (\gamma' (r)) + \abs{\tau' (r)}^2}
      {\tau (r)} 
      \diff r    
  \ge 2 \sinh^{-1} \left(\frac{\sqrt{\vphantom{\big\vert} d_{\manifold{M}} (x, y)^2 + (t - s)^2}}{2 \sqrt{ts\,}}\right)
  \eofs.
  \qedhere
\end{equation*}
\end{proof}

As a corollary of \cref{distancesStar}, we have the estimate.
\begin{equation}
\label{eq_dist_ln_upperbound}
  d_{\manifold{M}^\star} \bigl((x, t), (y, s)\bigr)
 \ge 2 \sinh^{-1} \Biggl(\frac{\abs{ t - s}}{2 \,\sqrt{ts\,}}\Biggr)
 = 2 \ln \frac{\abs{t - s} + \abs{t + s}}{2 \sqrt{ts\,}} 
 = \biggabs{\ln \frac{t}{s}}
 \eofs.
\end{equation}
(This estimate can also be proved directly by a crude lower bound on the metric on \(\manifold{M}^\star\).)

Another consequence of \cref{distancesStar} is the completeness of the manifold \(\manifold{M}^\star\).

\begin{lemma}
\label{lemma_star_complete}
If the manifold \(\manifold{M}\) is complete, then the manifold \(\manifold{M}^\star\) is complete.
\end{lemma}
\begin{proof}
Assume that \((x_j, t_j)_{j \in \Nset}\) is a Cauchy sequence in \(\manifold{M}^\star = \manifold{M} \times (0, + \infty)\). By \eqref{eq_dist_ln_upperbound}, the sequence \((\ln t_j)_{j \in \Nset}\) is a Cauchy sequence in \(\Rset\). 
By completeness of \(\Rset\), \((\ln t_j)_{j \in \Nset}\) converges thus to some \(\tau_* \in \Rset\). In particular, the sequence \((t_j)_{j \in \Nset}\) converges to \(t_* = e^{\tau_*}\) and is bounded. It follows then from \cref{distancesStar}, that the sequence \((x_j)_{j \in \Nset}\) is a Cauchy sequence in the manifold \(\manifold{M}\) and thus converges to some \(x_*\) by completeness of \(\manifold{M}\).
\end{proof}

We are now in position to prove our ball-merging lemma.

\begin{proof}[Proof of \cref{lemma_merging_Mstar}]
We proceed by induction on \(\ell \in \Nset\). In the case \(\ell = 0\), the conclusion holds trivially. We assume now that \(\ell \ge 1\) and that the conclusion holds for \(\ell - 1\).
If the balls are all disjoint and contained in \(\manifold{M} \times (0, T)\), then the conclusion holds by setting \(\ell' = \ell\), for each \(i \in \{1, \dotsc, \ell'\}\), \(a_i' = a_i\) and \(r_i' = r_i\), and \(T' = T\).

If the balls are not all disjoint, then up to reordering the balls we can assume that \(\Bar{B}^{\manifold{M}^*}_{r_{\ell - 1}} (a_{\ell - 1}) \cap \Bar{B}^{\manifold{M}^*}_{r_\ell} (a_\ell) \ne \emptyset\). 
By the triangle inequality, this implies that \(d_{\manifold{M}^\star} (a_{\ell - 1}, a_{\ell}) \le r_{\ell - 1} + r_{\ell}\). 
Since by \cref{lemma_star_complete} the manifold \(\manifold{M}^*\) is complete and since \(d_{\manifold{M}^\star}\) is a geodesic distance, there exists a point \(\Tilde{a}_{\ell - 1} \in \manifold{M}\) such that 
\begin{align*}
    d_{\manifold{M}} (\Tilde{a}_{\ell - 1}, a_{\ell - 1}) 
  &= 
    \tfrac{1}{2}\bigl(d_{\manifold{M}^\star}(a_{\ell - 1}, a_\ell) + r_{\ell} - r_{\ell - 1}\bigr)
    \eofs,
    \\
\intertext{and}
  d_{\manifold{M}} (\Tilde{a}_{\ell - 1}, a_{\ell}) 
  &= 
    \tfrac{1}{2}
    \bigl(d_{\manifold{M}^\star}(a_{\ell - 1}, a_\ell) + r_{\ell-1} - r_{\ell}\bigr)
    \eofs.
\end{align*}
If we define \(\Tilde{r}_{\ell - 1} \defeq \frac{1}{2}\bigl(d_{\manifold{M}^\star} (a_{\ell - 1}, a_\ell) + r_{\ell-1} + r_{\ell}\bigr)\) we have 
\begin{align*}
\Tilde{r}_{\ell - 1} &\le r_{\ell - 1} + r_\ell &
&\text{and}&
      B_{r_{\ell - 1}}^{\manifold{M}^\star} (a_{\ell - 1}) 
    \cup 
      B_{r_{\ell}}^{\manifold{M}^\star} (a_{\ell}) & 
  \subseteq 
    B_{\Tilde{r}_{\ell - 1}}^{\manifold{M}^\star} (\Tilde{a}_{\ell - 1})\eofs.
\end{align*}
We set, for \(i \in \{1, 2, \dotsc, \ell - 2\}\), \(\Tilde{a}_i = a_i\) and \(\Tilde{r}_i = r_i\) and \(\Tilde{\ell} = \ell - 1\). We conclude by applying our induction hypothesis to \(\Tilde{\ell}\), \(\Tilde{a}_1, \Tilde{a}_2, \dotsc, \Tilde{a}_{\Tilde{\ell}}\) and  \(\Tilde{r}_1, \Tilde{r}_2, \dotsc, \Tilde{r}_{\Tilde{\ell}}\).

If all the balls are disjoint, it can still happen that they do not all lie in \(\manifold{M} \times (0, T)\).
We can assume, by a permutation of the indices, that \(\Bar{B}_{r_\ell}^{\manifold{M}^\star} (a_\ell) \cap \manifold{M} \times [T, +\infty) \ne \emptyset\). We define now 
\(
    \Tilde{T}
  \defeq 
    T e^{-2 r_\ell}
\).
By definition of the metric on \(\manifold{M}_*\), in view of \cref{distancesStar} and its consequence \eqref{eq_dist_ln_upperbound}, this implies that 
\begin{equation*}
      B^{\manifold{M}^\star}_{r_\ell} (a_\ell)
    \subset
      \manifold{M} \times (\Tilde{T}, +\infty).
\end{equation*}
We conclude by setting for each \(\Tilde{\ell} = \ell - 1\), for every \(i \in \{1, 2, \dotsc, \ell - 1\}\), \(\Tilde{a}_i \defeq a_i\) and \(\Tilde{r}_i \defeq r_i\) and applying our induction hypothesis to \(\Tilde{\ell}\), \(\Tilde{a}_1, \Tilde{a}_2, \dotsc, \Tilde{a}_{\Tilde{\ell}}\), \(\Tilde{r}_1, \Tilde{r}_2, \dotsc, \Tilde{r}_{\Tilde{\ell}}\) and \(\Tilde{T}\).
\end{proof}

\subsubsection{Proof of the theorem}

In order to prove \cref{theorem_bounded_finite_bubbles_manifold} we need a lower bound on the measure of balls, which will be an ingredient in the control of the covering. In the proof of \cref{theorem_bounded_finite_bubbles_linear} this bound followed directly from the transitivity of isometries of the hyperbolic space.

\begin{lemma}[Lower bound on the volume of balls on \(\manifold{M}^\star\)]
\label{lemma_lower_bound_volume}
There exists a constant \(C > 0\) such that for every \((x, t) \in \manifold{M}^* = \manifold{M} \times (0, + \infty)\),
\begin{equation*}
  \mu_{\manifold{M}^\star} 
      \bigl(B^{\manifold{M}^\star}_{\rho}(x, t)\bigr)
  \ge 
    C \frac{(\tanh \rho)^{m + 1}}{1 + (t \sinh \rho)^m}
\end{equation*}
\end{lemma}

In particular, if \(t \in (0, + \infty)\) and \(\rho \in (0, + \infty)\) remain in a bounded region, we have by \cref{lemma_lower_bound_volume}
\[
   \mu_{\manifold{M}^\star} 
      \bigl(B^{\manifold{M}^\star}_{\rho}(x, t)\bigr)
    \ge c' \rho^{m + 1}.
\]

\begin{proof}%
[Proof of \cref{lemma_lower_bound_volume}]
\resetconstant
We observe that if \(s,t \in (0, + \infty)\), \(\rho \in (0, + \infty)\) and 
\[
    (s - t \cosh \rho)^2 
  \le 
     \frac{(t\sinh \rho)^2}{2}
    \eofs,
\]
then 
\[
    s^2 + t^2
  \le 
      2 \, st \cosh \rho 
    -
      \frac{(t \sinh \rho)^2}{2} 
\eofs .    
\]
If moreover, \(x, y \in \manifold{M}\) satisfy 
\[
 d_{\manifold{M}} (x, y) \le \frac{(t \sinh \rho)^2}{2}
\]
then by the distance formula of \cref{distancesStar},  
\[
    d_{\manifold{M}^\star} \bigl((y, s), (x, t)\bigr)
  =
         \cosh^{-1} 
      \Biggl(
        \frac
          {d_{\manifold{M}} (x, y)^2 +  t^2 + s^2}
          {2\, t s} \Biggr)
  \le 
    \rho\eofs.
\]
We have thus for every \((x, t) \in \manifold{M}^\star = \manifold{M} \times (0, +\infty)\) and \(\rho \in (0, +\infty)\),
\[
      B^\mathcal{M}_{\frac{t}{\sqrt{2}} \sinh \rho} (x) 
    \times 
      \bigl[t \cosh \rho - \tfrac{t}{\sqrt{2}} \sinh \rho, t \cosh \rho + \tfrac{t}{\sqrt{2}} \sinh \rho\bigr]
  \subset
    B^{\manifold{M}^\star}_{\rho} (x, t)
    \eofs.
\]

We infer then from the monotonicity property of measures that
\begin{equation}
\label{eq_woojaco1Mu}
\begin{split}
 \mu_{\manifold{M}^\star} \bigl(B^{\mathcal{M}^\star}_\rho (x, t)\bigr)
 &\ge 
  \int
    _{t \cosh \rho - \tfrac{t}{\sqrt{2}} \sinh \rho}
    ^{t \cosh \rho + \tfrac{t}{\sqrt{2}} \sinh \rho}
    \frac
      {\mu_\manifold{M} \bigl(B_{t/\sqrt{2} \sinh \rho} (x)\bigr)}
      {s^{m + 1}} 
    \diff s\\
 &\ge 
  \frac
    {
      \sqrt{2}
      \, 
      \mu_{\manifold{M}} 
      \bigl(B_{\frac{t}{\sqrt{2}} \sinh \rho} (x)\bigr)
    }
    {
      t^m
      \,
      (\cosh \rho)^{m + 1}
    }
  \sinh \rho
  \eofs,
\end{split}
\end{equation}
by convexity of the function \(s \in (0, + \infty) \mapsto s^{-(m + 1)}\).
Since the manifold \(\manifold{M}\) is compact, there exists a constant \(\Cl{cst_aeChaipiz3}\) such that for every \(x \in \manifold{M}\) and \(\sigma > 0\) we have
\begin{equation}
\label{eq_AhW4uh7Aw7}
    \mu_{\manifold{M}} \bigl(B_{\sigma} (x)\bigr)
  \ge 
    \Cr{cst_aeChaipiz3} \min (\sigma^m, 1)
\eofs.
\end{equation}
It follows then from \eqref{eq_woojaco1Mu} and \eqref{eq_AhW4uh7Aw7} that
\begin{equation*}
    \mu_{\manifold{M}^\star} \bigl(B^{\mathcal{M}^\star}_\rho (x, t)\bigr) 
  \ge 
    \Cr{cst_aeChaipiz3} 
    \min
      \,
      \biggl(
        \frac{(\tanh \rho)^{m + 1}}{2^{m/2}},
        \frac{\tanh \rho}{ (t \cosh \rho)^m}
      \biggr)
\eofs .
\qedhere
\end{equation*}
\end{proof}

The free homotopy decomposition will be made through Lipschitz-continuous maps on spheres in \(\manifold{M}^\star\). The next lemma ensures that the shape of these small spheres is controlled and will serve as a substitute to \cref{lemma_HyperbolicSphere}.

\begin{lemma}%
[Lower bound on the injectivity radius on \(\manifold{M}^\star\)]
\label{lemma_M_star_injectivity_radius}
For every \(\eta > 0\), there exists \(T > 0\) such that if 
\((x, t) \in \manifold{M} \times (0, T]\), then 
\begin{equation*}  
  \operatorname{inj}_{\manifold{M}_*} (x, t)
 =
  \sinh^{-1} \biggl(\frac{\operatorname{inj}_{\manifold{M}} (x)}{t}\biggr)
\end{equation*}
and for every \(\eta \in (0, 1)\), the exponential map, its inverse are controlled uniformly in \(\mapsset{C}^1\) 
on ball of radius \(\eta \operatorname{inj}_{\manifold{M}_*} (x, t)\) around \((x, t) \in \manifold{M}\).
\end{lemma}
\begin{proof}
We observe by \cref{distancesStar}, that for every \((x,t), (y, s) \in \manifold{M}^\star = \manifold{M} \times (0, +\infty)\),
\[
\begin{split}
    d_{\manifold{M}} (x, y)^2
  &=
      2 
      \,
      st 
      \cosh d_{\manifold{M}^\star} ((x, t), (y, s))
    - 
      t^2
    - 
      s^2
  \\
  &= 
      t^2
      \,
      \bigl(\sinh d_{\manifold{M}^\star} ((x, t), (y, s)) \bigr)^2
   -
      \bigl(s - t \cosh d_{\manifold{M}^\star} ((x, t), (y, s))\bigr)^2
  \\
  &
  \le 
    t^2
    \,
    \bigl(\sinh d_{\manifold{M}^\star} ((x, t), (y, s))\bigr)^2
    \eofs.
\end{split}
\]
It follows then that 
\[
  \operatorname{inj}_{\manifold{M}_*} (x, t)
 =
  \sinh^{-1} \biggl(\frac{\operatorname{inj}_{\manifold{M}} (x)}{t}\biggr).
\]
The bounds follow then from the compactness of \(\manifold{M}\) and the homogeneity of the metric on \(\manifold{M}^\star\).
\end{proof}

\Cref{theorem_bounded_finite_bubbles_manifold} will follow from the following statement.

\begin{theorem}%
[Estimate of free homotopy decompositions by a truncated fractional energy]
\label{theorem_bounded_finite_bubbles_manifold_bis}
Let \(\manifold{M}\) and \(\manifold{N}\) be compact Riemannian manifolds.
If \(\varepsilon > 0\) is small enough, then there is a constant \(C > 0\) such that for every \(\lambda > 0\), there exists finite sets \(\mapsset{F}^\lambda \subset \mapsset{C} (\Sset^m, \manifold{N})\) and \(\mathcal{F}_0^\lambda \subset \mapsset{C} (\manifold{M}, \manifold{N})\) such that any map  \(f \in \mapsset{C} (\manifold{M}, \manifold{N})\) satisfying
\begin{equation*}
\iint\limits_{\substack{\Sset^m \times \Sset^m}} \frac{\bigl(d_{\manifold{N}} (f (y), f (x)) - \varepsilon\bigr)_+}{d_{\manifold{M}} (y, x)^{2 m}} \diff x \diff y
        \le 
          \lambda,
\end{equation*}
has  a free homotopy decomposition into \(f_1, \dotsc, f_k \in \mapsset{F}^\lambda\) upon \(f_0 \in \mapsset{F}_0^\lambda\) with \(k \le C \lambda\).
\end{theorem}

\begin{proof}%
\resetconstant%
Let \(F \in \mapsset{C}^\infty (\manifold{M} \times [0, + \infty])\) be the function given by \cref{proposition_extension_manifold}, let \(\manifold{N}_\delta = \{ \xi \in \Rset^\nu \st \dist (y, \manifold{N}) < \delta\}\) and \(A_\delta = F^{-1} (\Rset^\nu \setminus \manifold{N}_\delta)\).
Since the manifold \(\manifold{N}\) is embedded smoothly into \(\Rset^\nu\) there exists \(\delta_* > 0\) and a Lipschitz-continuous retraction \(\Pi : \manifold{N}_{\delta_*} \to \manifold{N}\).

By the estimate given on \(F\) by \cref{proposition_extension_manifold} \eqref{itExtensionManifoldHyperbolicLipschitz}, there exists \(\rho > 0\), independent of \(F\), such that if \(x \in A_{\delta_*}\), then \(B_\rho (x) \subset A_{\delta_*/2}\). We consider a maximal set \(A \subset A_{\delta_*}\) such that if for every \(a, b \in A\) such that \(a \ne b\), one has \(d^{\manifold{M}^\star} (a, b) \ge 2 \rho\). This implies immediately that 
\begin{equation*}
    A_{\delta_*} 
  \subseteq
    \bigcup_{a \in A} 
      B^{\manifold{M}^\star}_{2 \rho} (a)
      \eofs.
\end{equation*}
Since the balls \(\bigl(B^{\manifold{M}^\star}_\rho (a)\bigr)_{a \in A}\) are disjoint, we have by \cref{lemma_lower_bound_volume} and \cref{proposition_extension_manifold} \eqref{itExtensionManifoldHyperbolicMeasure}
\begin{equation*}
\begin{split}
    \# \bigl(A \cap \manifold{M} \times (0, 1]\bigr)
  &\le
    \Cl{cst_ungeV3Xeiv}
    \sum_{a \in A \cap \manifold{M} \times (0, 1]} 
    \mu_{\manifold{M}^\star} 
      (B^{\manifold{M}^\star}_{\rho}(a))
      \\
  &\le 
      \Cr{cst_ungeV3Xeiv}
      \,
      \mu_{\manifold{M}^\star} \bigl(A_{\delta_*/2}\bigr)\\
  &\le 
      \Cl{cst_tdt}
      \iint 
        \limits_{\manifold{M} \times \manifold{M}}  
        \frac
          {\bigl(d_{\manifold{N}} (f (y), f (x)) - \varepsilon\bigr)_+}
          {d_{\manifold{M}} (y, x)^{2 m} } 
        \diff y 
        \diff x
      \le \Cr{cst_tdt} \lambda 
\eofs.
\end{split}
\end{equation*}
We have then
\begin{equation}
\label{eq_thowoo2Chi}
    A_\delta 
  \subseteq 
      \bigcup_{a \in A \cap \manifold{M} \times (0, 1]} 
        B_{2 \rho} (a)
    \cup 
      (e^{-2\rho}, + \infty)
 \eofs.
\end{equation}

Thanks to \cref{lemma_M_star_injectivity_radius}, there exists \(T \in (0, e^{-2\rho})\) such that if 
\(
    \sigma 
  \le 
    \Cr{cst_tdt}
    \,
    2\rho
     \lambda
    \)
and \((x, t) \in \manifold{M} \times (0, T]\), then the exponential map, its inverse and their derivatives are controlled on \(B_\sigma (x, t)\).

By \cref{lemma_merging_Mstar}, there exists balls \((B^{\manifold{M}^\star}_{\rho_i} (a_i))_{1 \le i \le \ell}\) and \(T' \in (0, T)\) such that in view of \eqref{eq_thowoo2Chi}
\[
  A_\delta \subseteq \bigcup_{i = 1}^\ell B^{\manifold{M}^\star}_{\rho_i} (a_i) \cup \manifold{M} \times (T', +\infty)\eofs,
\]
with the estimate
\[
  \frac{1}{2} \ln \frac{T}{T'} + \sum_{i = 1}^\ell \rho_i
 \le 2 \rho \, \#\bigl(A \cap \manifold{M} \times [0, T_2]\bigr)
 \le  
    2 
    \rho
    \,
    \Cr{cst_tdt}
    \lambda
\eofs.
\]
This implies in particular, since \(T' \le T\) and \(\rho_i \ge \rho\) that 
\begin{align}
\label{eq_zeeHao8ueY}
    \ln \frac{T}{T'} 
  &\le 
    2 
    \rho
    \,
    \Cr{cst_tdt} 
      \lambda 
&
   \ell 
  &\le
    2
    \,
    \Cr{cst_tdt} 
      \lambda 
      &
&\text{and}
&
\rho_i 
  &\le 
    2 
    \rho
    \, 
    \Cr{cst_tdt} 
    \lambda
\eofs.
\end{align}

Since \(\Pi \compose F\) is Lipschitz continuous on \(\manifold{M}^\star \setminus A_\delta\), it follows then that the map \(f\) has a free homotopy decomposition into \(\Pi  \compose F\vert_{\partial_{B_{\rho_1}} (a_1)}, \dotsc, \Pi  \compose F\vert_{\partial_{B_{\rho_\ell}} (a_\ell)}\) upon \(\Pi \compose F\vert_{\manifold{M} \times \{T'\}}\).
Since by \cref{lemma_M_star_injectivity_radius} the exponenial map is controlled in \(\manifold{M} \times (0, T]\) by a bound depending on \eqref{eq_zeeHao8ueY}, the Lipschitz constants of the maps \(\Pi  \compose F\vert_{\partial_{B_{\rho_1}} (a_1)}, \dotsc, \Pi  \compose F\vert_{\partial_{B_{\rho_\ell}} (a_\ell)}\) upon \(\Pi \compose F\vert_{\manifold{M} \times \{T'\}}\) are bounded independently of \(f\) and the geometry of their domains are controlled by quantities depending only on \(\lambda\), and thus by Ascoli's compactness theorem the maps \(\Pi  \compose F\vert_{\partial_{B_{\rho_1}} (a_1)}, \dotsc, \Pi  \compose F\vert_{\partial_{B_{\rho_\ell}} (a_\ell)}\) and \(\Pi \compose F\vert_{\manifold{M} \times \{T'\}}\) are respectively homotopic to maps taken in some finite sets \(\mapsset{F}^\lambda \subset \mapsset{C} (\Sset^m, \manifold{N})\) and \(\mapsset{F}_0^\lambda \subset \mapsset{C} (\manifold{M}, \manifold{N})\).
\end{proof}

\subsection{Estimates of free homotopy decompositions by a scaled gap potential}
We obtain a version of \cref{theorem_bounded_finite_bubbles_manifold} that scales optimally with respect to \(\varepsilon\), which generalizes \cref{theorem_bounded_finite_bubbles_linear_scaling} to a general domain \(\manifold{M}\).

\begin{theorem}%
[Estimate of free homotopy decompositions by a scaled gap potential]
\label{theorem_bounded_finite_bubbles_manifold_scaling}
Let \(\manifold{M}\) and \(\manifold{N}\) be compact Riemannian manifolds.
If \(\dim \manifold{M} = m \ge 2\), then there are constants \(C > 0\) and \(\varepsilon_0 > 0\) such that for every \(\lambda > 0\), there exists finite sets \(\mapsset{F}^\lambda \subset \mapsset{C} (\Sset^m, \manifold{N})\) and \(\mathcal{F}_0^\lambda \subset \mapsset{C} (\manifold{M}, \manifold{N})\) such that if \(\varepsilon \in (0, \varepsilon_0)\), then any \(f \in \mapsset{C} (\manifold{M}, \manifold{N})\) satisfying 
\begin{equation*}
          \iint
            \limits_{
              \substack{
                (x, y) \in \manifold{M} \times \manifold{M}\\                 
                d_{\manifold{N}} (f (y), f (x)) > \varepsilon}} 
              \frac{\varepsilon^m}{d_{\manifold{M}} (y, x)^{2 m}} 
            \diff x 
            \diff y 
        \le 
          \lambda
\eofs,
\end{equation*}
has a free homotopy decomposition into \(f_1, \dotsc, f_k \in \mapsset{F}^\lambda\) upon \(f_0 \in \mapsset{F}_0^\lambda\) with \(k \le C \lambda\).
\end{theorem}

\Cref{theorem_bounded_finite_bubbles_manifold_scaling} will be deduced from \cref{theorem_bounded_finite_bubbles_manifold} thanks to \cref{proposition_fract_gap_integral_growth_sphere}, \cref{proposition_p_q_scaling_Omega} and the counterpart of \cref{lemma_pair_covering}.

\begin{lemma}
[Covering pairs of a compact manifold]
\label{lemma_manifold_diagonal_covering}
If \(\manifold{M}\) is a connected compact manifold of dimension \(m\), then there exists open sets \(A_1, \dotsc, A_\ell \subset \manifold{M}\) such that for each \(i \in \{1, \dotsc, \ell\}\), the set \(\Bar{A}_i\) is diffeomorphic to the closed ball \(\mathbb{B}_1 \subset \Rset^m\) and such that 
\[
    \manifold{M} \times \manifold{M} 
  \subseteq 
    \bigcup_{i = 1}^\ell 
      A_i \times A_i 
\eofs.
\]
\end{lemma}
\begin{proof}
Since \(\manifold{M}\) is connected, for every \(x, y \in \manifold{M}\), there exists an open set \(A \subset \manifold{M}\) such that \(\Bar{A}\) is diffeomorphic to the closed ball \(\Bar{\Bset}^{m} \subset \Rset^m\) and such that \(\{x, y\} \subset A\) and therefore \((x, y) \in A \times A\). 
The conclusion then follows by compactness of \(\manifold{M} \times \manifold{M}\).
\end{proof}

\begin{proof}%
[Proof of \cref{theorem_bounded_finite_bubbles_manifold_scaling}]
This follows from \cref{theorem_bounded_finite_bubbles_manifold},
\cref{proposition_p_q_scaling_Omega} and \cref{lemma_manifold_diagonal_covering}.
\end{proof}

\subsection{Estimates on the induced cohomology homorphism}
As for maps from the sphere, it is possible to obtain \emph{linear bounds} for \emph{cohomology invariants} of mappings \(f : \manifold{M} \to \manifold{N}\).

  If \(f : \manifold{M} \to \manifold{N}\) is a smooth map, then its pullback \(f_\sharp\) induces a homomorphism \(f^* : H^*_{\mathrm{dR}} (\manifold{N}) \to H^*_{\mathrm{dR}} (\manifold{N})\) on the de Rham cohomology \cite{Lee_2009}{\S 10}. Indeed, if \(\omega \in \mapsset{C}^\infty (\bigwedge^{\ell} \manifold{N})\) and \(d \omega = 0\), then \(d (f_\sharp \omega) = 0\) and moreover if \(\theta \in \mapsset{C}^{\infty} (\bigwedge^{\ell - 1} \manifold{N})\), then \(f_\sharp (\omega + d \theta) = f_\sharp \omega + d (f_\sharp \theta)\).
This induced homomorphism is invariant under homotopies. 
Cohomology induced homomorphism are linear operator on the finite dimensional de Rham cohomology groups; this allows one to define a norm \(\norm{f}_*\) on cohomology induced homomorphims \(f_*\).

If \(\manifold{M} = \Sset^m\), then \(H^\ell_{\mathrm{dR}} (\Sset^m) \ne \{0\}\) if and only if \(m = \ell\); the induced cohomology homomorphism \(f^*\) is then completely described by the Hurewicz homomorphism.

The following theorem generalizes the estimate for the Hurewicz homomorphism \cref{proposition_Hurewicz_nonlocal} to the cohomology homomorphism.

\begin{theorem}
[Estimate of the cohomology induced homorphism by a scaled gap potential]
Let \(\manifold{M}\) and \(\manifold{N}\) be a compact Riemannian manifolds.
If \(\dim \manifold{M} \ge 2\) and if \(\varepsilon > 0\) is small enough, then for every \(f \in \mapsset{C}^1 (\manifold{M}, \manifold{N})\), one has 
\[
 \norm{f^*}
 \le 
   C \iint\limits_{\substack{(x, y) \in \Sset^m \times \Sset^m \\ d_{\manifold{N}} (f (y), f (x)) > \varepsilon}} \frac{\varepsilon^m}{d_{\manifold{M}} (y, x)^{2 m}} \diff x \diff y \eofs .
\]
\end{theorem}
\begin{proof}%
\resetconstant
Since \(\manifold{M}\) and \(\manifold{N}\) are compact, the dimension of the de Rham cohomology is finite and thus, in view of the Poincaré duality \cite{Lee_2009}*{\S 10.4}, it suffices to estimate for every \(\ell \in \{0, \dotsc, m\}\), for every \(\omega \in \mapsset{C}^\infty (\bigwedge^\ell \manifold{N})\) and every \(\theta \in \mapsset{C}^\infty (\bigwedge^{m - \ell} \manifold{M})\) such that \(d\omega = 0\) and \(d \theta = 0\) the quantity
\[
  \int_{\manifold{M}} 
    f_\sharp \omega \wedge \theta
    \eofs .
\]
As before, since \(\manifold{N}\) is a compact embedded submanifold of \(\Rset^\nu\), there exists an open set \(U \subset \Rset^\nu\) with \(\manifold{N} \subset U\) and a smooth retraction \(\Pi \in \mapsset{C}^\infty (U, \manifold{N})\).
We fix a function 
\(\eta \in \mapsset{C}^{\infty} (\Rset^\nu, \bigwedge^0 \Rset^\nu)\) such that \(\eta = 1\) on \(\manifold{N}\) and \(\operatorname{supp} \eta \subset U\) and 
a canonical projection map \(P : \manifold{M} \times [0, + \infty) \to \manifold{M}\) defined for each \((x, t) \in \manifold{M} \times [0, + \infty)\)
by \(P (x, t) \defeq x\).

We consider the map \(F : \manifold{M}^\star \to \Rset^\nu\) given by \cref{proposition_extension_manifold}. By the Stokes--Cartan formula we have 
\begin{equation*}
    \int_{\manifold{M}} 
      f_\sharp \omega \wedge \theta
  = 
    \int_{\manifold{M}} 
      F_\sharp (\eta \wedge \Pi_\sharp \omega) \wedge P_\sharp \theta
  = 
    \int_{\manifold{M}^\star} 
      d (F_\sharp (\eta \wedge \Pi_\sharp \omega) \wedge P_\sharp \theta)
\eofs.
\end{equation*}
We compute then, since \(d \omega = 0\) and \(d \theta = 0\),
\begin{equation*}
    d (F_\sharp (\eta \wedge \Pi_\sharp \omega) \wedge P_\sharp \theta)
  = 
    F_\sharp (d\eta \wedge \Pi_\sharp \omega)\wedge P_\sharp \theta
\eofs,
\end{equation*}
and we conclude by \cref{proposition_extension_manifold}
that 
\begin{equation*}
\begin{split}
    \biggabs{
      \int_{\manifold{M}} 
        f_\sharp \omega \wedge \theta
        \,
    }
  &\le
    \C 
    \,
    \norm{DF}_{L^\infty (\manifold{M}^\star)}^{\ell + 1}
    \,
    \norm{d\eta}_{L^\infty (\Rset^\nu)}
    \,
    \norm{\omega}_{L^\infty (\manifold{N})}
    \,
    \norm{\theta}_{L^\infty (\manifold{M})}
    \\
    &\qquad \qquad 
    \mu_{\manifold{M}^\star} 
      \bigl(
        \bigr\{ 
          (x, t) \in \manifold{M}^\star 
        \st
          \dist (F (x, t), \manifold{N}) \ge \delta\}\bigr)\\
   &\le \C \,
    \norm{\omega}_{L^\infty (\manifold{N})}
    \,
    \norm{\theta}_{L^\infty (\manifold{M})}
   \iint \limits_{\manifold{M} \times \manifold{M}}
         \frac{ \bigl(d_{\manifold{N}} (f (y), f (x))-\varepsilon)_+}{d_{\manifold{M}} (y, z)^{2 m}} \diff y \diff x
\eofs .
\end{split}
\end{equation*}
By considering all admisible \(\ell \in \{0, \dotsc, m\}\), every  \(\omega \in \mapsset{C}^\infty (\bigwedge^\ell \manifold{N})\) and every \(\theta \in \mapsset{C}^\infty (\bigwedge^{m - \ell} \manifold{M})\) such that \(d \omega = 0\) and \(d \theta = 0\), we conclude that  
\begin{equation}
    \norm{f^*}
   \le \C 
   \iint \limits_{\manifold{M} \times \manifold{M}}
         \frac{ \bigl(\abs{f (y) - f (z)}-\varepsilon)_+}{d_{\manifold{M}} (y, z)^{2 m}} \diff y \diff z\eofs . 
\end{equation}
By \cref{lemma_fract_gap_integral_growth}, \cref{proposition_p_q_scaling_Omega} and \cref{lemma_manifold_diagonal_covering}, the conclusion then follows.
\end{proof}

\section{Further problems}
\label{section_problems}

A first question that remains open at the end of the present work is whether estimates with optimal scaling can be proved when \(m = 1\).

\begin{openproblem}
\label{problem_1d_scaling}
Does \cref{theorem_bounded_finite_bubbles_scaling} hold for \(m = 1\)?
\end{openproblem}

A variant of this problem would be to obtain estimates with optimal scaling on the Hurewicz homomorphism when \(m = 1\).

\begin{openproblem}
\label{problem_1d_scaling_hurewicw}
Does \cref{proposition_Hurewicz_nonlocal_scaling} hold when \(m = 1\)?
\end{openproblem}

The problem is already open for maps for the degree of maps from the circle \(\Sset^1\) to the circle \(\Sset^1\), that is when \(\manifold{N} = \Sset^1\) (see \cite{Nguyen_2017}). It is striking that the present work and \familyname{Nguyên} Hoài-Minh followed quite different strategies of proof but encountered the same restriction that \(m > 1\).

The solution of \cref{theorem_bounded_finite_bubbles_scaling,problem_1d_scaling_hurewicw} could be connected to the following more technical question of scaling of truncated integrals.

\begin{openproblem}
\label{problem_scaling}
If \(p \in [0, 1)\) and \(m \in \Nset\), 
does there exists a constant \(C > 0\) such that 
for every convex set \(\Omega \subset \Rset^m\)
and for every map \(f : \Omega \to \Rset^\nu\),
if \(\delta < \varepsilon\), then
\begin{multline*}
    \iint
    \limits_{
      \substack{
        (x, y) \in \Omega \times \Omega\\
        d_{\manifold{N}} (f (y), f (x)) \ge \varepsilon}
      }
      \hspace{-1.5em}
      \frac
        {\bigl(d_{\manifold{N}} (f (y), f (x)) - \varepsilon\bigr)^p}
        {\abs{y - x}^{2 m}}
      \diff x
      \diff y\\[-2em]
  \le 
    C 
    \,
    \biggl(
    \frac
      {\delta}
      {\varepsilon}
    \biggr)^{m - p}
    \hspace{-1.5em}
    \iint
    \limits_{
      \substack{
        (x, y) \in \Omega \times \Omega\\
        d_{\manifold{N}} (f (y), f (x)) \ge \delta}
      }
      \hspace{-1.5em}
      \frac
        {\bigl(d_{\manifold{N}} (f (y), f (x)) - \delta\bigr)^p}
        {\abs{y - x}^{2 m}}
      \diff x
      \diff y
\eofs?
\end{multline*}
\end{openproblem}

By \cref{lemma_fract_gap_integral_growth}, the answer is positive when \(p \in [1, +\infty)\) and when \(p \in [0, 1)\), the estimate holds with an exponent \(m - 1\) instead of \(m - p\).

As we have mentioned in the introduction, for every \(\lambda > 0\), there exists a finite collection of maps \(\mapsset{F}^\lambda\) such that every \(f \in (\mapsset{C} \cap W^{1, 1}) (\Sset^1, \manifold{N})\) is homotopic to some map in \(\mapsset{F}^\lambda\). The proof is done by showing that \(f\) is homotopic to a constant speed reparametrization and reduces thus the problem to Lipschitz-continuous maps to which the Ascoli theorem applies. This raises naturally the question about \(W^{m, 1} (\Sset^m, \manifold{N})\).

\begin{openproblem}
Is it true that if for \(m \ge 2\), if \(\lambda > 0\), then the maps \(f \in (\mapsset{C} \cap W^{m, 1}) (\Sset^m, \manifold{N})\), such that 
\(\int_{\Sset^m} \abs{D^m f} \le \lambda\) belong to finitely many homotopy classes?
\end{openproblem}

Finally, for maps from \(\Sset^{2n - 1}\) to \(\Sset^{n}\), the Hopf invariant can be computed through formulas that yield Rivière's estimate \eqref{ineq_Hopf_Riviere}.
The next logical step would be to obtain corresponding estimates in fractional Sobolev spaces.

\begin{openproblem}
\label{problem_hopf_Sobolev}
If \(p > 2n - 1\), does there exists a constant such that every \(f \in (W^{s, p} \cap \mapsset{C}) (\Sset^{2 n - 1}, \Sset^n)\), with \(s = (2n - 1)/p\) satisfies
\begin{equation*}
    \abs{\deg_H (f)}
  \le
    C
    \,
    \Biggl(
    \;
    \iint 
      \limits_{\Sset^{2 n - 1} \times \Sset^{2n - 2}}
      \frac
        {d_{\manifold{N}} (f (y), f (x))^p}
        {\abs{y - x}^{2 (2 n - 1)}}  
      \diff y \diff x
     \Biggr)^{1 + \frac{1}{2n - 1}} 
    \eofs?
\end{equation*}
\end{openproblem}

A further question would be to obtain gap potential estimates:

\begin{openproblem}
\label{problem_hopf}
Does there exists a constant such that if \(f \in \mapsset{C} (\Sset^{2 n - 1}, \Sset^n)\),
\begin{equation*}
    \abs{\deg_H (f)}
  \le
    C
    \Biggl(
    \iint 
      \limits_{\substack{y, z \in \Sset^{2 n - 1} \\ \abs{f (y) - f (z)} \ge \varepsilon}}
         \frac{\diff y \diff z}{\abs{y - z}^{2 (2 n - 1)}} 
     \Biggr)^{1 + \frac{1}{2n - 1}} 
    \eofs?
\end{equation*}
\end{openproblem}

The exponent in \cref{problem_hopf_Sobolev} and \cref{problem_hopf} is justified by the following lower bound.

\begin{proposition}
If \(n \in \Nset\) is even, then there exists a sequence of maps \((f_k)_{k \in \Nset}\) in \(\mapsset{C} (\Sset^{2n - 1}, \Sset^n)\) such that \(\abs{\deg_H f_k} \to \infty\)
and 
\begin{equation*}
\iint \limits_{\substack{x, y \in \Sset^{2 n - 1} \\ \abs{f_k (y) - f_k (x)} \ge \varepsilon}}
         \frac{\diff y \diff x}{\abs{y - x}^{2 (2 n - 1)}} 
         \le C \bigl(\deg_H (f_k)\bigr)^{1 - \frac{1}{2n}}
\eofs.
\end{equation*}
\end{proposition}

\begin{proof}
The proof follow Tristan \familyname{Rivière}'s proof \cite{Riviere_1998}*{lemma III.1}.
We construct for every \(k \in \Nset\), the map \(f_k = \varphi_k \compose \varphi\) where \(\psi : \Sset^{2n - 1} \to \Sset^{n}\) has a nontrivial Hopf degree and \(\varphi_k : \Sset^n \to \Sset^n\) has the property that \(\abs{D \varphi_k} \le k^{1/n}\) on \(\Sset^n\)
and its Brouwer degree satisfies \(\deg (\varphi_k) = k\). It follows that \(\abs{D f_k} \le k^{1/n}\) and \(\deg_H (f_k) = k^2\). Moreover we have 
\begin{equation*}
 \iint \limits_{\substack{x, y \in \Sset^{2 n - 1} \\ \abs{f_k (y) - f_k (x)} \ge \varepsilon}}
         \frac{\diff y \diff x}{\abs{y - x}^{2 (2 n - 1)}} 
         \le k^{2 - \frac{1}{n}} \simeq \bigl(\deg_H (f_k)\bigr)^{1 - \frac{1}{2n}}
\eofs.
\qedhere
\end{equation*}
\end{proof}

\resetconstant
A strategy that follows the proof of \cref{theorem_bounded_finite_bubbles_linear} constructs for a given \(f \in \mapsset{C} (\Sset^{2n - 1}, \Sset^n)\) a decomposition into \(g_i : \Sset^{2n - 1} \to \Sset^n\), with \(1 \le i \le k\), which have a Lipschitz constant controlled by \(\Cl{cst_eeha6iW4ae} \sinh \rho_i\), with \(\sum_{i = 1}^k \rho_i \le \Cl{cst_aeghayoe3C} \lambda\). It follows then by Rivière's bound \eqref{ineq_Riviere_n} and by convexity that 
\begin{equation}
 \deg_H (f) = \sum_{i = 1}^k \deg_H (g_i)
 \le \C \sum_{i = 1}^k \bigl(\Cr{cst_eeha6iW4ae} \sinh \rho_i \bigr)^{2 n}
 \le \C \sinh \C \lambda
 \eofs,
\end{equation}
which is quite far from the estimate proposed in \cref{problem_hopf}.


\begin{bibdiv}
\begin{biblist}

\bib{Ahlfors1981}{book}{
   author={Ahlfors, Lars V.},
   title={M\"obius transformations in several dimensions},
   series={Ordway Professorship Lectures in Mathematics},
   publisher={University of Minnesota School of Mathematics},
   place={Minneapolis, Minn.},
   date={1981},
   pages={ii+150},
}


\bib{Baues_1977}{book}{
   author={Baues, Hans J.},
   title={Obstruction theory on homotopy classification of maps},
   series={Lecture Notes in Mathematics}, 
   volume={628},
   publisher={Springer}, 
   address={Berlin--New York},
   date={1977},
   pages={xi+387},
   isbn={3-540-08534-3},
}
                

\bib{Bott_Tu_1982}{book}{
   author={Bott, Raoul},
   author={Tu, Loring W.},
   title={Differential forms in algebraic topology},
   series={Graduate Texts in Mathematics},
   volume={82},
   publisher={Springer},
   address={New York--Berlin},
   date={1982},
   pages={xiv+331},
   isbn={0-387-90613-4},
}


\bib{BourgainBrezisMironescu2005CPAM}{article}{
   author={Bourgain, Jean},
   author={Brezis, Ha\"{\i}m},
   author={Mironescu, Petru},
   title={Lifting, degree, and distributional Jacobian revisited},
   journal={Comm. Pure Appl. Math.},
   volume={58},
   date={2005},
   number={4},
   pages={529--551},
   issn={0010-3640},
   doi={10.1002/cpa.20063},
}       

\bib{BourgainBrezisNguyen2005CRAS}{article}{
   author={Bourgain, Jean},
   author={Brezis, Ha\"\i m},
   author={Nguyen, Hoai-Minh}*{inverted={yes}},
   title={A new estimate for the topological degree},
   journal={C. R. Math. Acad. Sci. Paris},
   volume={340},
   date={2005},
   number={11},
   pages={787--791},
   issn={1631-073X},
   doi={10.1016/j.crma.2005.04.007},
}

\bib{Boutet_Georgescu_Purice}{article}{
   author={Boutet de Monvel-Berthier, Anne},
   author={Georgescu, Vladimir},
   author={Purice, Radu},
   title={A boundary value problem related to the Ginzburg-Landau model},
   journal={Comm. Math. Phys.},
   volume={142},
   date={1991},
   number={1},
   pages={1--23},
   issn={0010-3616},
}

\bib{Brezis_2011}{book}{
   author={Brezis, Haim},
   title={Functional analysis, Sobolev spaces and partial differential
   equations},
   series={Universitext},
   publisher={Springer}, 
   address={New York, N.Y.},
   date={2011},
   pages={xiv+599},
   isbn={978-0-387-70913-0},
}
                
\bib{BrezisLi2000CRAS}{article}{
   author={Brezis, Ha\"\i m},
   author={Li, Yan Yan}*{inverted={yes}},
   title={Topology and Sobolev spaces},
   journal={C. R. Acad. Sci. Paris S\'er. I Math.},
   volume={331},
   date={2000},
   number={5},
   pages={365--370},
   issn={0764-4442},
   doi={10.1016/S0764-4442(00)01656-6},
}

\bib{BrezisLi2001JFA}{article}{
   author={Brezis, Ha\"\i m},
   author={Li, Yan Yan}*{inverted={yes}},
   title={Topology and Sobolev spaces},
   journal={J. Funct. Anal.},
   volume={183},
   date={2001},
   number={2},
   pages={321--369},
   issn={0022-1236},
   doi={10.1006/jfan.2000.3736},
}
\bib{Brezis_Nguyen_2011}{article}{
   author={Brezis, Ha\"{i}m},
   author={Nguyen, Hoai-Minh}*{inverted={yes}},
   title={On a new class of functions related to \(\mathrm{VMO}\)},
   journal={C. R. Math. Acad. Sci. Paris},
   volume={349},
   date={2011},
   number={3-4},
   pages={157--160},
   issn={1631-073X},
   doi={10.1016/j.crma.2010.11.026},
}
\bib{BrezisNirenberg1995}{article}{
   author={Brezis, H.},
   author={Nirenberg, L.},
   title={Degree theory and \(\mathrm{BMO}\)},
   part={I},
   subtitle={Compact manifolds without boundaries},
   journal={Selecta Math. (N.S.)},
   volume={1},
   date={1995},
   number={2},
   pages={197--263},
   doi={10.1007/BF01671566},
}

\bib{ChenLi2014}{article}{
   author={Chen, Jingyi}*{inverted={yes}},
   author={Li, Yuxiang}*{inverted={yes}},
   title={Homotopy classes of harmonic maps of the stratified 2-spheres and
   applications to geometric flows},
   journal={Adv. Math.},
   volume={263},
   date={2014},
   pages={357--388},
   issn={0001-8708},
   doi={10.1016/j.aim.2014.07.001},
}


\bib{Druet_Hebey_Robert}{book}{
   author={Druet, Olivier},
   author={Hebey, Emmanuel},
   author={Robert, Fr\'ed\'eric},
   title={Blow-up theory for elliptic PDEs in Riemannian geometry},
   series={Mathematical Notes},
   volume={45},
   publisher={Princeton University Press}, 
   address={Princeton, N.J.},
   date={2004},
   pages={viii+218},
   isbn={0-691-11953-8},
   doi={10.1007/BF01158557},
}


\bib{DuzaarKuwert1998}{article}{
   author={Duzaar, Frank},
   author={Kuwert, Ernst},
   title={Minimization of conformally invariant energies in homotopy
   classes},
   journal={Calc. Var. Partial Differential Equations},
   volume={6},
   date={1998},
   number={4},
   pages={285--313},
   issn={0944-2669},
   doi={10.1007/s005260050092},
}

\bib{Dyda2004}{article}{
   author={Dyda, Bart\l omiej},
   title={A fractional order Hardy inequality},
   journal={Illinois J. Math.},
   volume={48},
   date={2004},
   number={2},
   pages={575--588},
   issn={0019-2082},
}
\bib{Fenchel_1989}{book}{
   author={Fenchel, Werner},
   title={Elementary geometry in hyperbolic space},
   series={De Gruyter Studies in Mathematics},
   volume={11},
   publisher={de Gruyter},
   address={Berlin},
   date={1989},
   pages={xii+225},
   isbn={3-11-011734-7},
   doi={10.1515/9783110849455},
}
\bib{Fleming1966TAMS}{article}{
   author={Fleming, Wendell H.},
   title={Flat chains over a finite coefficient group},
   journal={Trans. Amer. Math. Soc.},
   volume={121},
   date={1966},
   pages={160--186},
   issn={0002-9947},
   doi={10.2307/1994337},
}

\bib{GastelNerf2013}{article}{
   author={Gastel, Andreas},
   author={Nerf, Andreas J.},
   title={Minimizing sequences for conformally invariant integrals of higher
   order},
   journal={Calc. Var. Partial Differential Equations},
   volume={47},
   date={2013},
   number={3-4},
   pages={499--521},
   issn={0944-2669},
   doi={10.1007/s00526-012-0525-0},
}
                

\bib{GoldsteinHajlasz2012}{article}{
   author={Goldstein, Pawe\l },
   author={Haj\l asz, Piotr},
   title={Sobolev mappings, degree, homotopy classes and rational homology
   spheres},
   journal={J. Geom. Anal.},
   volume={22},
   date={2012},
   number={2},
   pages={320--338},
   issn={1050-6926},
   doi={10.1007/s12220-010-9194-4},
}



\bib{Gromov_1999}{article}{
   author={Gromov, Mikhael},
   title={Quantitative homotopy theory},
   conference={
      title={Prospects in mathematics},
      address={Princeton, N.J.},
      date={1996},
   },
   book={
      publisher={Amer. Math. Soc.}, 
      address={Providence, R.I.},
   },
   date={1999},
   pages={45--49},
}

\bib{Gromov_1999_book}{book}{
   author={Gromov, Misha},
   title={Metric structures for Riemannian and non-Riemannian spaces},
   series={Progress in Mathematics},
   volume={152},
   contribution={
    type={appendices},
    author={Katz, M.},
    author={Pansu, P.},
    author={Semmes, S.},
    },
   translator={Bates, Sean Michael},
   publisher={Birkh\"{a}user}, 
   address={Boston, Mass.},
   date={1999},
   pages={xx+585},
   isbn={0-8176-3898-9},
}
                

\bib{HajlaszIwaniecMalyOnninen2008}{article}{
   author={Haj\l asz, Piotr},
   author={Iwaniec, Tadeusz},
   author={Mal\'y, Jan},
   author={Onninen, Jani},
   title={Weakly differentiable mappings between manifolds},
   journal={Mem. Amer. Math. Soc.},
   volume={192},
   date={2008},
   number={899},
   issn={0065-9266},
   doi={10.1090/memo/0899},
}


\bib{Hang_Lin_2001_MRL}{article}{
   author={Hang, Fengbo},
   author={Lin, Fanghua},
   title={Topology of Sobolev mappings},
   journal={Math. Res. Lett.},
   volume={8},
   date={2001},
   number={3},
   pages={321--330},
   issn={1073-2780},
   doi={10.4310/MRL.2001.v8.n3.a8},
}

\bib{HangLin2003Acta}{article}{
   author={Hang, Fengbo}*{inverted={yes}},
   author={Lin, Fanghua}*{inverted={yes}},
   title={Topology of Sobolev mappings},
   part={II},
   journal={Acta Math.},
   volume={191},
   date={2003},
   number={1},
   pages={55--107},
   issn={0001-5962},
   doi={10.1007/BF02392696},
}

\bib{Hang_Lin_2003_CPAM}{article}{
   author={Hang, Fengbo},
   author={Lin, Fanghua},
   title={Topology of Sobolev mappings},
   part={III},
   journal={Comm. Pure Appl. Math.},
   volume={56},
   date={2003},
   number={10},
   pages={1383--1415},
   issn={0010-3640},
   doi={10.1002/cpa.10098},
}
                

\bib{Hang_Lin_2005}{article}{
   author={Hang, Fengbo},
   author={Lin, Fanghua},
   title={Topology of Sobolev mappings},
   part={IV},
   journal={Discrete Contin. Dyn. Syst.},
   volume={13},
   date={2005},
   number={5},
   pages={1097--1124},
   issn={1078-0947},
   doi={10.3934/dcds.2005.13.1097},
}
                


                
\bib{Hatcher_2002}{book}{
   author={Hatcher, Allen},
   title={Algebraic topology},
   publisher={Cambridge University Press}, 
   address={Cambridge},
   date={2002},
   pages={xii+544},
   isbn={0-521-79160-X},
   isbn={0-521-79540-0},
}


\bib{Hu_1959}{book}{
   author={Hu, Sze-tsen},
   title={Homotopy theory},
   series={Pure and Applied Mathematics, Vol. VIII},
   publisher={Academic Press, New York-London},
   date={1959},
   pages={xiii+347},
   }

\bib{Jerrard_1999}{article}{
   author={Jerrard, Robert L.},
   title={Lower bounds for generalized Ginzburg-Landau functionals},
   journal={SIAM J. Math. Anal.},
   volume={30},
   date={1999},
   number={4},
   pages={721--746},
   issn={0036-1410},
   doi={10.1137/S0036141097300581},
}        
\bib{Kuwert1998}{article}{
   author={Kuwert, Ernst},
   title={A compactness result for loops with an $H^{1/2}$-bound},
   journal={J. Reine Angew. Math.},
   volume={505},
   date={1998},
   pages={1--22},
   issn={0075-4102},
   doi={10.1515/crll.1998.117},
}

\bib{Lee_2009}{book}{
   author={Lee, Jeffrey M.},
   title={Manifolds and differential geometry},
   series={Graduate Studies in Mathematics},
   volume={107},
   publisher={American Mathematical Society}, 
   address={Providence, R.I.},
   date={2009},
   pages={xiv+671},
   isbn={978-0-8218-4815-9},
   doi={10.1090/gsm/107},
}
                
\bib{Magnus_Karrass_Solitar_1966}{book}{
   author={Magnus, Wilhelm},
   author={Karrass, Abraham},
   author={Solitar, Donald},
   title={Combinatorial group theory: Presentations of groups in terms of
   generators and relations},
   publisher={Interscience Publishers (John Wiley \& Sons)},
   address={New York--London--Sydney},
   date={1966},
   pages={xii+444},
}                

\bib{Mironescu}{article}{
   author={Mironescu, Petru},
   title={Sobolev maps on manifolds: degree, approximation, lifting},
   conference={
      title={Perspectives in nonlinear partial differential equations},
   },
   book={
      series={Contemp. Math.},
      volume={446},
      publisher={Amer. Math. Soc.}, 
      address={Providence, R.I.},
   },
   date={2007},
   pages={413--436},
   doi={10.1090/conm/446/08642},
}

\bib{Muller2000}{article}{
   author={M\"uller, Thomas},
   title={Compactness for maps minimizing the $n$-energy under a free
   boundary constraint},
   journal={Manuscripta Math.},
   volume={103},
   date={2000},
   number={4},
   pages={513--540},
   issn={0025-2611},
   doi={10.1007/PL00005864},
}

\bib{Nguyen_2006}{article}{
   author={Nguyen, Hoai-Minh}*{inverted={yes}},
   title={Some new characterizations of Sobolev spaces},
   journal={J. Funct. Anal.},
   volume={237},
   date={2006},
   number={2},
   pages={689--720},
   issn={0022-1236},
   doi={10.1016/j.jfa.2006.04.001},
}

\bib{Nguyen_2007_CRAS}{article}{
   author={Nguyen, Hoai-Minh}*{inverted={yes}},
   title={$\Gamma$-convergence and Sobolev norms},
   journal={C. R. Math. Acad. Sci. Paris},
   volume={345},
   date={2007},
   number={12},
   pages={679--684},
   issn={1631-073X},
   doi={10.1016/j.crma.2007.11.005},
}
                
                
\bib{Nguyen_2007}{article}{
   author={Nguyen, Hoai-Minh}*{inverted={yes}},
   title={Optimal constant in a new estimate for the degree},
   journal={J. Anal. Math.},
   volume={101},
   date={2007},
   pages={367--395},
   issn={0021-7670},
   doi={10.1007/s11854-007-0014-0},
}

\bib{Nguyen_2008}{article}{
   author={Nguyen, Hoai-Minh}*{inverted={yes}},
   title={Further characterizations of Sobolev spaces},
   journal={J. Eur. Math. Soc. (JEMS)},
   volume={10},
   date={2008},
   number={1},
   pages={191--229},
   issn={1435-9855},
   doi={10.4171/JEMS/108},
}
                


\bib{Nguyen_2008_CRAS}{article}{
   author={Nguyen, Hoai-Minh},
   title={Inequalities related to liftings and applications},
   journal={C. R. Math. Acad. Sci. Paris},
   volume={346},
   date={2008},
   number={17-18},
   pages={957--962},
   issn={1631-073X},
   doi={10.1016/j.crma.2008.07.026},
}

\bib{Nguyen_2011}{article}{
   author={Nguyen, Hoai-Minh}*{inverted={yes}},
   title={$\Gamma$-convergence, Sobolev norms, and BV functions},
   journal={Duke Math. J.},
   volume={157},
   date={2011},
   number={3},
   pages={495--533},
   issn={0012-7094},
   doi={10.1215/00127094-1272921},
}
                
\bib{Nguyen_2011_CVAR}{article}{
   author={Nguyen, Hoai-Minh},
   title={Some inequalities related to Sobolev norms},
   journal={Calc. Var. Partial Differential Equations},
   volume={41},
   date={2011},
   number={3-4},
   pages={483--509},
   issn={0944-2669},
   doi={10.1007/s00526-010-0373-8},
}
                
                
\bib{Nguyen_2014}{article}{
   author={Nguyen, Hoai-Minh}*{inverted={yes}},
   title={Estimates for the topological degree and related topics},
   journal={J. Fixed Point Theory Appl.},
   volume={15},
   date={2014},
   number={1},
   pages={185--215},
   issn={1661-7738},
   doi={10.1007/s11784-014-0182-3},
}
        

\bib{Nguyen_2017}{article}{
   author={Nguyen, Hoai-Minh}*{inverted={yes}},
   title={A refined estimate for the topological degree},
   journal={C. R. Math. Acad. Sci. Paris},
   volume={355},
   date={2017},
   number={10},
   pages={1046--1049},
   issn={1631-073X},
   doi={10.1016/j.crma.2017.10.007},
}

\bib{Nguyen_Pinamonti_Squassina_Vecchi_2018}{article}{
   author={Nguyen, Hoai-Minh},
   author={Pinamonti, Andrea},
   author={Squassina, Marco},
   author={Vecchi, Eugenio},
   title={New characterizations of magnetic Sobolev spaces},
   journal={Adv. Nonlinear Anal.},
   volume={7},
   date={2018},
   number={2},
   pages={227--245},
   issn={2191-9496},
   doi={10.1515/anona-2017-0239},
}

\bib{Parker_2003}{article}{
  author={Parker, Thomas H.},
  title={What is… a bubble tree?},
  journal={Notices Amer. Math. Soc.},
  volume={50},
  date={2003},
  number={6},
  pages={666--667},
  issn={0002-9920},
}


\bib{Petersen_2016}{book}{
   author={Petersen, Peter},
   title={Riemannian geometry},
   series={Graduate Texts in Mathematics},
   volume={171},
   edition={3},
   publisher={Springer, Cham},
   date={2016},
   pages={xviii+499},
   isbn={978-3-319-26652-7},
   isbn={978-3-319-26654-1},
   doi={10.1007/978-3-319-26654-1},
}
                
\bib{PetracheRiviere2015}{article}{
   author={Petrache, Mircea},
   author={Rivi{\`e}re, Tristan},
   title={Global gauges and global extensions in optimal spaces},
   journal={Anal. PDE},
   volume={7},
   date={2014},
   number={8},
   pages={1851--1899},
   issn={2157-5045},
   doi={10.2140/apde.2014.7.1851},
}

\bib{PetracheVanSchaftingen2017}{article}{
  author={Petrache, M.},
  author={Van Schaftingen, Jean},
  title={Controlled singular extension of critical trace Sobolev maps 
  from spheres to compact manifolds}, 
  journal={Int. Math. Res. Not. IMRN},
  date={2017}, 
  number={12}, 
  pages={3467--3683},
  doi={10.1093/imrn/rnw109},
}

\bib{Riviere_1998}{article}{
   author={Rivi\`ere, Tristan},
   title={Minimizing fibrations and $p$-harmonic maps in homotopy classes
   from $\Sset^3$ into $\Sset^2$},
   journal={Comm. Anal. Geom.},
   volume={6},
   date={1998},
   number={3},
   pages={427--483},
   issn={1019-8385},
   doi={10.4310/CAG.1998.v6.n3.a2},
}

\bib{SacksUhlenbeck1981}{article}{
   author={Sacks, J.},
   author={Uhlenbeck, K.},
   title={The existence of minimal immersions of $2$-spheres},
   journal={Ann. of Math. (2)},
   volume={113},
   date={1981},
   number={1},
   pages={1--24},
   issn={0003-486X},
   doi={10.2307/1971131},
}


\bib{Sandier_1998}{article}{
   author={Sandier, Etienne},
   title={Lower bounds for the energy of unit vector fields and
   applications},
   journal={J. Funct. Anal.},
   volume={152},
   date={1998},
   number={2},
   pages={379--403},
   issn={0022-1236},
   doi={10.1006/jfan.1997.3170},
}
                
\bib{Sandier_Serfaty_2007}{book}{
   author={Sandier, Etienne},
   author={Serfaty, Sylvia},
   title={Vortices in the magnetic Ginzburg--Landau model},
   series={Progress in Nonlinear Differential Equations and their
   Applications},
   volume={70},
   publisher={Birkh\"auser},
   address={Boston, Mass.},
   date={2007},
   pages={xii+322},
   isbn={978-0-8176-4316-4},
   isbn={0-8176-4316-8},
}

\bib{SchoenWolfson2001}{article}{
   author={Schoen, R.},
   author={Wolfson, J.},
   title={Minimizing area among Lagrangian surfaces: the mapping problem},
   journal={J. Differential Geom.},
   volume={58},
   date={2001},
   number={1},
   pages={1--86},
   issn={0022-040X},
   doi={10.4310/jdg/1090348282},
}

\bib{Serre_1951}{article}{
   author={Serre, Jean-Pierre},
   title={Homologie singuli\`ere des espaces fibr\'es. Applications},
   journal={Ann. of Math. (2)},
   volume={54},
   date={1951},
   pages={425--505},
   doi={10.2307/1969485},
}

                

\bib{Toda_1962}{book}{
   author={Toda, Hirosi}*{inverted={yes}},
   title={Composition methods in homotopy groups of spheres},
   series={Annals of Mathematics Studies, No. 49},
   publisher={Princeton University Press}, 
   address={Princeton, N.J.},
   date={1962},
   pages={v+193},
}

\bib{White1986JDG}{article}{
   author={White, Brian},
   title={Infima of energy functionals in homotopy classes of mappings},
   journal={J. Differential Geom.},
   volume={23},
   date={1986},
   number={2},
   pages={127--142},
   issn={0022-040X},
   doi={10.4310/jdg/1214440023},
}



\bib{Whitehead_1947}{article}{
   author={Whitehead, J. H. C.},
   title={An expression of Hopf's invariant as an integral},
   journal={Proc. Nat. Acad. Sci. U. S. A.},
   volume={33},
   date={1947},
   pages={117--123},
   issn={0027-8424},
}
                
\bib{Willem_2013}{book}{
   author={Willem, Michel},
   title={Functional analysis},
   series={Cornerstones},
   subtitle={Fundamentals and applications},
   publisher={Birkh\"auser/Springer}, 
   address={New York, N.Y.},
   date={2013},
   pages={xiv+213},
   isbn={978-1-4614-7003-8},
   isbn={978-1-4614-7004-5},
   doi={10.1007/978-1-4614-7004-5},
}
                
\end{biblist}
 
\end{bibdiv}
\end{document}